\documentclass[11 pt]{amsart}
\usepackage{hyperref} 

\allowdisplaybreaks
\usepackage{changebar}
\usepackage{latexsym}
\usepackage{amssymb, amsmath}
\usepackage{amsthm}
\usepackage{geometry}
\usepackage{comment}
\usepackage[T1]{fontenc}
\usepackage{lmodern}
\usepackage[normalem]{ulem}

\usepackage[stretch=10]{microtype}
\usepackage{tikz}
\usetikzlibrary{calc}
\usetikzlibrary{decorations.pathreplacing,angles,quotes}

\numberwithin{equation}{section}

\newtheorem{theorem}{Theorem}[section]
\newtheorem{lemma}[theorem]{Lemma}
\newtheorem{cor}[theorem]{Corollary}
\newtheorem{sublem}[theorem]{Sublemma}
\newtheorem{proposition}[theorem]{Proposition}

\newtheorem{defin}[theorem]{Definition}

\newtheorem{remark}[theorem]{Remark}

\newcommand{\cA}{\mathcal{A}}
\newcommand\cB{{\mathcal B}}
\newcommand{\BB}{\mathcal{B}}
\newcommand{\cC}{\mathcal{C}}
\newcommand\cD{{\mathcal D}}

\newcommand{\cF}{\mathcal{F}}
\newcommand{\cG}{\mathcal{G}}
\newcommand{\cH}{\mathcal{H}}
\newcommand{\cI}{\mathcal{I}}

\newcommand{\cK}{\mathcal{K}}
\newcommand{\cL}{\mathcal{L}}
\newcommand{\LL}{\mathcal{L}}
\newcommand{\cM}{\mathcal{M}}
\newcommand{\cN}{\mathcal{N}}
\newcommand{\cO}{\mathcal{O}}
\newcommand{\cP}{\mathcal{P}}

\newcommand{\cQ}{\mathcal{Q}}
\newcommand\cR{{\mathcal R}}
\newcommand{\cS}{\mathcal{S}}
\newcommand{\cV}{\mathcal{V}}
\newcommand{\cW}{\mathcal{W}}
\newcommand{\cT}{\mathcal{T}}

\newcommand{\bH}{\mathbb{H}}
\newcommand{\bI}{\mathbb{I}}
\newcommand{\bL}{\mathbb{L}}

\newcommand{\bS}{\mathbb{S}}
\newcommand\bR{{\mathbb R}}
\newcommand\bZ{{\mathbb Z}}
\newcommand\integer{{\mathbb Z}}
\newcommand{\bV}{\mathbb{V}}

\newcommand{\bpsi}{\overline{\psi}}

\newcommand{\wW}{\widetilde{\cW}}

\newcommand{\tpsi}{\widetilde{\psi}}
\newcommand{\tJ}{\widetilde{J}}
\newcommand{\tnu}{\tilde{\nu}}

\newcommand{\hW}{\widehat{\cW}}

\newcommand{\Int}{{\mbox{Int}\, }}

\newcommand{\musrb}{{\mu_{\tiny{\mbox{SRB}}}}}
\newcommand{\hatmusrb}{\hat \mu_{\tiny{\mbox{SRB}}}}
\newcommand{\ximusrb}{\mu^\xi_{\tiny{\mbox{SRB}}}}

\newcommand{\vf}{\varphi}

\newcommand{\ve}{\varepsilon}

\newcommand{\vv}{v}

\def\beq{\begin{equation}}
\def\eeq{\end{equation}}

\newcommand{\diam}{\mbox{diam}}

\begin{document}

\title{Thermodynamic formalism for dispersing billiards}
\author{Viviane Baladi  \and Mark F. Demers}
\address{Laboratoire de Probabilit\'es, Statistique et Mod\'elisation (LPSM),  
CNRS, Sorbonne Universit\'e, Universit\'e  Paris Cit\'e,
4, Place Jussieu, 75005 Paris, France}
\email{baladi@lpsm.paris}
\address{Department of Mathematics, Fairfield University, Fairfield CT 06824, USA}
\email{mdemers@fairfield.edu}

\thanks{We thank Yuri Lima for stimulating questions and comments, and the anonymous referee for 
valuable clarifications.
MD was partly supported by NSF grants DMS 1800321 and DMS 2055070. VB's research is supported
by the European Research Council (ERC) under the European Union's Horizon 2020 research and innovation programme (grant agreement No 787304).
VB is grateful to the Knut and Alice Wallenberg Foundation for an invitation to Lund
University in 2020.}

\date{\today}

\begin{abstract}
For any finite horizon Sinai billiard map $T$
on the two-torus,  we find $t_*>1$ such that for each 
 $t\in (0,t_*)$ there exists a unique equilibrium
state $\mu_t$ for $- t\log J^uT$, and $\mu_t$ is
$T$-adapted.  (In particular, the SRB measure is the unique equilibrium state
for $- \log J^uT$.)
We show that $\mu_t$ is exponentially mixing for H\"older observables, and
the pressure function $P(t)=\sup_\mu \{h_\mu -\int t\log J^uT d \mu\}$ is analytic
on $(0,t_*)$. In addition,  $P(t)$ is strictly convex if and only if $\log J^uT$ is
not $\mu_t$-a.e. cohomologous to a constant, while,   
if there exist $t_a\ne t_b$ with $\mu_{t_a}= \mu_{t_b}$, then $P(t)$ is affine on $(0,t_*)$.
An additional
sparse recurrence condition gives  $\lim_{t\downarrow 0} P(t)=P(0)$.
\end{abstract}
\maketitle

\vspace{-14 pt}
\tableofcontents

\section{Introduction and Statement of Main Results}
\label{sec:intro}

\subsection{Set-up}

A Sinai  or dispersive billiard table $Q$  on the two-torus $\mathbb{T}^2$
is a set
  $
 Q=\mathbb{T}^2 \setminus \cup_{i=1}^\Omega \cO_i
 $,
for some finite number $\Omega\ge 1$ of 
pairwise disjoint closed domains $\cO_i$ (the obstacles, or scatterers) with $C^3$ boundaries 
having strictly positive curvature $\cK$. (In particular, the domains are strictly convex.)
 The
billiard flow, also called a periodic Lorentz gas,  is 
the motion of a point particle traveling in $Q$ at unit speed 
 and
undergoing specular reflections at the boundary of
the scatterers.   (At a tangential --- also called grazing --- collision, the
reflection does not change the direction of the particle.)

We study here the associated billiard map $T:M\to M$  on the compact set
$M = \partial Q \times [-\frac{\pi}{2}, \frac{\pi}{2} ]$, defined to be
the first collision map on the boundary of $Q$. We
use the standard coordinates $x = (r,\vf)$, where $r$ is arclength along
$\partial \cO_i$ and $\vf$ is the angle the post-collision trajectory makes with the normal to 
$\partial \cO_i$.
Grazing
collisions cause discontinuities in the map $T$.
 We remark, however, that since the flow is continuous, the map $T$ is well-defined 
and bijective on $M$.  There is no need to reduce the domain to a smaller set.

For $x \in M$, let $\tau(x)$ denote the distance from $x$ to $T(x)$ (the  free flight time). Set $\cK_{\max}=\sup \cK<\infty$,
 $\cK_{\min}=\inf \cK>0$,  and $\tau_{\min}=\inf \tau>0$. 
Then \cite{chernov book}
the cones in $\mathbb{R}^2$ defined by
\begin{align*}C^u &= \{ (dr, d\vf) 
 : \cK_{\min} \le \frac {d\vf}{dr} \le \cK_{\max} +\frac 1{\tau_{\min}} \}\, ,
\, C^s 
= \{ (dr, d\vf) 
: - \cK_{\min} \ge  \frac {d\vf}{dr} \ge -\cK_{\max} -  \frac 1{\tau_{\min}}\}
\end{align*}
are strictly invariant under $DT$ and $DT^{-1}$, respectively, whenever these derivatives exist.

The map $T$ is uniformly hyperbolic, in the following sense:  Let 
\begin{equation}\label{defL}\Lambda := 1 + 2 \tau_{\min} \cK_{\min}>1\, .
\end{equation}
Then there exists $C_1>0$ such that, for all $x$ for which $DT^n(x)$,  respectively $DT^{-n}(x)$, is defined, 
\begin{equation}
\label{eq:hyp}
\| DT^n(x) \vv \| \ge C_1 \Lambda^n \| \vv \|, \,\, \forall \vv \in \cC^u, \quad
\| DT^{-n}(x) \vv \| \ge C_1 \Lambda^n \| \vv \|, \,\, \forall \vv \in \cC^s \, , \, \, \forall n\ge 0\, . 
\end{equation}
Let $\cS_0 = \{ (r, \vf) \in M: \vf = \pm \frac{\pi}{2} \}$ denote the set of tangential collisions on $M$.
Then  
\begin{equation}\label{defSn}
\cS_n = \cup_{i=0}^{-n} T^i\cS_0\, , \, n \in \mathbb{Z}\,  ,
\end{equation}
is the singularity set for $T^n$. In other words,
there exists $n\in \integer$ such that $DT^n(x)$ is 
not defined  if and only if $x$  belongs to
the   (invariant and dense,  \cite[Lemma 4.55]{chernov book}) 
set of curves
$\cup_{m \in \mathbb{Z}} \cS_m$.
Let
\begin{equation}\label{defM'} 
M'  =  M\setminus \cup_{m \in \mathbb{Z}} \cS_m  \, .
\end{equation}
The  spaces $E^u(x)$ and $E^s(x)$  are defined
at any $x\in M'$. 
Indeed, 
for each $n \ge 0$, let $x_n = T^nx$, and consider $v_n = DT^{-n}(x_n) v/\| DT^{-n}(x_n) v\|$ for some 
$v \in C^s$.  Since $x \in M'$, we have that
$DT^{-n}(x_n)$ is well-defined for each $n \ge 0$.
By uniform hyperbolicity,  the sequence $v_n$ converges to a vector
$v_\infty$.  The direction of $v_\infty$ is $E^s(x)$.  Similarly, for $y \in M \setminus \cup_{m \le 0} \cS_m$,
consider $y_n = T^{-n}y$ and $u_n = DT^{n}(y_n) u /\| DT^{n}(y_n) u\|$, for $n \ge 0$ and 
$u \in C^u$.  The limit of $u_n$ is $E^u(y)$.

We have \cite[Theorem 4.66, Theorem 4.75]{chernov book} that
  $\mbox{Lebesgue}(M'\setminus M)=\musrb(M'\setminus M)=0$,
where $\musrb= (2  |\partial Q|)^{-1} \cos \vf \, drd\vf$ is the unique absolutely continuous invariant
measure. 
Also, 
at each $x\in M'$, the unstable and stable Jacobians $J^uT(x)$ and $J^sT(x)$, with respect
to arclength along unstable, respectively stable, manifolds, 
are well-defined and nonzero.
Note also that, if $J_{\mathrm{Leb}}T$ denotes the
Jacobian of $T$ with respect to Lebesgue, then, setting $E(x) = \sin (\angle(E^s(x), E^u(x)))$,
\begin{equation}
\label{eq:jacs}
 J_{\mathrm{Leb}}T (x)= \frac{\cos (\vf(x))}{\cos (\vf ( T(x)))} = J^u T(x)\cdot  J^sT(x) \cdot \frac{E \circ T(x)}{E(x)} \, ,\,\,
\forall x\in M' \, .
\end{equation}

We assume that the billiard table $Q$ has
{\it finite horizon,} 
i.e., the  billiard flow on  $Q$ does not have any trajectories making only tangential collisions. This implies (but is not\footnote{We
 need the stronger condition e.g. in the proof of Proposition~\ref{prop:max}.}  equivalent to) 
 $\tau_{\max}:=\sup \tau<\infty$, see \cite[Remark 1.1]{max}.


\subsection{Potentials and Pressure. 
Theorems \ref{thm:equil}--\ref{strconv}. Corollaries \ref{cor1}-- \ref{CLT}.The Operator $\cL_t$}

Since $T$ admits a finite generating partition (see the beginning
of Section~\ref{next}), it follows that for any
$T$-invariant probability measure\footnote{All probability measures in the
present work are Borel measures.} $\mu$, the Kolmogorov entropy $h_\mu(T)$ is finite (\cite[Theorem 4.10, Theorem 4.17]{walters}).

Fix $t\ge 0$.
Let $\mu$ be  a $T$-invariant probability measure $\mu$.
If $\mu(M\setminus M')= 0$,
define the \emph{pressure of $\mu$} for the (so-called \emph{geometric) potential $- t \log J^uT$} by
\[
P_\mu(-t \log J^uT) = h_\mu(T) - t \int_M \log J^uT \, d\mu \, .
\]
If $\mu(M\setminus M')\ne 0$,
we set $\int_M \log J^uT \, d\mu=\infty$,
so that $P_\mu(-t \log J^uT)=-\infty$ if $t>0$.
 Due to the invariance of $\mu$, the bound
\eqref{eq:hyp} implies that $\int_M \log J^uT \, d\mu = \lim_{n \to \infty} \frac 1n \int_M \log J^uT^n \, d\mu \ge \log \Lambda$, thus the integral is
either well-defined and nonnegative or infinite. 
(See Remark~\ref{newremark} for more comments.)
It is known that 
\begin{equation}
\chi^u:= \int_M \log J^uT \, d\musrb =h_{\musrb}(T) \in(\log \Lambda,\infty)\, ,
\end{equation}
so that 
$P_{\musrb}( - \log J^uT)=0$ (this is the Pesin entropy formula, see e.g. \cite[Theorem 3.42]{chernov book}).

\smallskip

For a bounded function $g:M\to \bR$, we set $P_\mu(-t\log J^uT+g)=P_\mu(-t\log J^uT)+\int g \, d\mu$, and
we define the \emph{pressure  $P(t,g)$
of the potential $- t \log J^uT +g$} by
\begin{equation}
\label{abovv}
P(t,g):=\sup \{P_\mu(-t\log J^uT+g) : \mu \mbox{ a
$T$-invariant probability measure } \}\, , \, \,
P(t):=P(t,0) \, .
\end{equation}
We call $\mu$ an \emph{equilibrium state} for the potential $- t \log J^uT +g$ if 
$P_{\mu}(-t \log J^uT+g) = P(t,g)$.  

The case $t=0$, $g=0$,
corresponds to the measure of maximal entropy. Under an additional
condition of  ``sparse recurrence to the singularity set'' (see Definition~\ref{sparse}), 
a measure $\mu_0$ 
with $P(0)=P_{\mu_0}(0)$  was recently  constructed in \cite{max} 
($\mu_0$ was called $\mu_*$ there),
 shown to be mixing (in fact, Bernoulli), to be the unique 
measure $\mu$ satisfying $P_\mu(0)=P(0)$, and to satisfy the $T$-adapted condition  \eqref{adapt} below. 
(The speed of mixing 
of $\mu_0$ is
not known.)

For $t=1$, we mentioned above that $P_\musrb(-\log J^uT)=0$. In addition, $\musrb$ is $T$-adapted and, 
for any $T$-invariant  probability  measure $\mu$ giving small
enough weight to neighbourhoods of
singularity sets  \cite[Part~IV, Theorem~1.1]{katok strelcyn},
the Ruelle inequality
$P_\mu( - \log J^uT)\le 0$ holds. The measure $\musrb$ is mixing, in fact, correlations for H\"older observables  decay exponentially \cite{Y98}.

For $t$ in a small interval\footnote{The interval depends on the exponential rate of return
(itself close to $1$)
to the Young tower coupling magnet.}   around $1$, 
\cite{zhang} established the existence of equilibrium
states for the potential $-t \log J^uT$ using a Young tower construction with exponential tails, 
proving that
these measures
are exponentially mixing on H\"older observables and are unique in the class of measures that lift 
to the Young tower.  

\smallskip
We  establish a thermodynamic formalism for Sinai billiards for 
$t \in (0, t_*)$, with $t_*>1$   defined by
\begin{equation}
\label{eq:t* def}
t_* := \sup \{ t>0 : \Lambda^{-t} < e^{P(t)} \} = \sup \left\{ t>0 : t > - \frac{P(t)}{\log \Lambda} \right\} \, .
\end{equation}
(That $t_*>1$ follows since $\Lambda>1$
from \eqref{defL}, while
$P(t) \ge 0$ for $t \le 1$.)  The definition of $t_*$  can be viewed as a pressure gap condition, controlling  by $P(t)$ the contribution from pieces that constantly get cut by the singularities. 
In particular, for any $t < t_*$, we may\footnote{It is in fact enough to require there that  $\theta^t < e^{P_*(t)}$.} choose $\theta \in(\Lambda^{-1}, 1)$ in the one-step expansion Lemma~\ref{lem:one step}
 so that $\theta^t < e^{P(t)}$.  This complexity bound
permits us to prove the required growth lemmas essential to our analysis.
Our first main result is the following theorem:

\begin{theorem}[Thermodynamic Formalism for Sinai Billiards]
\label{thm:equil}
For each $t \in (0, t_*)$,
the potential $- t \log J^uT$ admits a unique equilibrium state $\mu_t$.
The measure $\mu_t$ is mixing,
gives positive mass to any nonempty open set, and  does not have atoms. Moreover, $\mu_t$ is $T$-adapted, 
that is\footnote{The $T$-adapted property appears in particular in the work of Lima--Matheus \cite{LM}.}
 \begin{equation}\label{adapt}
 \int |\log d(x, \cS_{\pm 1})| \, d\mu_t<\infty\, . 
 \end{equation}
In addition, $\mu_t$  has exponential decay of correlations
for H\"older observables. Finally, if $T$ satisfies the sparse
 recurrence condition 
then $\lim_{t\downarrow 0} P(t)=P(0)$.
\end{theorem}

We  prove Theorem~\ref{thm:equil} for $t\in (0,t_*)$ in three steps:
\begin{itemize}
\item
First, we introduce in Section~\ref{altt} an equivalent  (topological)  
expression $P_*(t)$ for $P(t)$,  generalising what was done in \cite{max} for $t=0$,
and we show that $P_*(t)$ is convex and strictly
decreasing (Proposition~\ref{prop:Ptilde}), and  that $P(t)\le P_*(t)$  
(Proposition \ref{prop:pressure}), for all $t>0$.
\item Next, for  $t\in (0,t_*)$,  we prove 
 the following properties for the transfer operator
\begin{equation}
\label{eq:L}
\cL_t f = \frac{f \circ T^{-1}}{|J^sT|^{1-t} \circ T^{-1}} \, 
\end{equation}
 acting on an anisotropic
Banach\footnote{
We attract the reader's attention to Lemma~\ref{lem:image} showing $\cL_t(C^1)\subset \cB$, which furnishes
 the proof of \cite[Lemma~4.9]{max}, which had been
omitted there, see Remark~\ref{omit}.}
space $\cB$ (Theorem~\ref{thm:spectral}):
The operator $\cL_t$ has spectral radius
$e^{P_*(t)}$, essential spectral radius strictly smaller than $e^{P_*(t)}$,
and the maximal eigenvectors
of $\cL_t$ and its dual give rise to  a $T$-invariant probability measure $\mu_t$. In addition,  $\cL_t$ has a spectral gap on $\cB$, so that $\mu_t$ is exponentially mixing on H\"older observables.
\item Finally, in Section~\ref{finall}, 
still for  $t\in (0,t_*)$, we show that $P_{\mu_t}(-t\log J^u T)=P_*(t)$ ,
so that $P(t)=P_*(t)$ (Corollary~\ref{cor:max}),
as well as the remaining claims about $\mu_t$: in particular  that $\mu_t$ is the
unique  equilibrium state among all $T$-invariant Borel probability measures
realising the variational principle $P(t)=P_*(t)$  (Theorem~\ref{thm:geo var}), and
that sparse recurrence implies that $P(t)$ tends to $P(0)$ as $t\downarrow 0$ (Proposition~\ref{zerolim}).
Our proof of uniqueness also gives
a more general variational principle, $P(t,g)=P_*(t,g)$, Theorem~\ref{thm:variational}.
\end{itemize}

We use the  Banach spaces $\cB$   introduced in \cite{dz2}, except that we work with (exact) stable manifolds
$\cW^s$ (as in \cite{max}) instead
of cone stable curves $\hW^s$ (see Section~\ref{defstuff}).
More importantly, we must tune the parameters
used to define $\cB=\cB(t_0,t_1)$ in Section~\ref{sec:norms} to an interval
$[t_0,t_1]\subset (0,t_*)$ containing  $t$.
In particular, the decay rate
 $k^{-q}$ defining the homogeneity strips \eqref{eq:H} in \cite{dz2} was $q=2$, while we need to assume $qt>1$ here (due to \eqref{eq:k0}).
Also,  we need to let the parameter $p$ used in the definition \eqref{stab}
of the strong stable norm tend to infinity when $t\to t_*$
(see Lemma~ \ref{lem:t ok}). It
follows that our bound for the essential spectral radius of $\cL_t$ on $\cB(t_0,t_1)$ deteriorates as $t_0\to 0$ or $t_1\to t_*$, and we lose the spectral gap
in both limits.

\smallskip 

The keys to the proof of the spectral Theorem~\ref{thm:spectral} are the delicate growth lemmas given in Sect.~\ref{GrL}. To prove these growth lemmas, subtle modifications of the
fundamental ideas of Chernov \cite{chernov book} and of the original techniques
 introduced in \cite{dz1, max} were necessary. In particular, the analysis for $t>1$ required a new bootstrap argument (see the beginning of Sect.~\ref{GrL} and Sect.~\ref{sec:t>1} 
 and \ref{sec:boot}).

\medskip

In Section~\ref{analytic}, a more careful study of the operator $\cL_t$ yields
our second main result:
\begin{theorem}[Strict Convexity]\label{strconv}
The function $t\mapsto P(t)$ is analytic on $(0,t_*)$, with
\begin{equation}\label{formula0}
P'(t)=\int \log J^s T \, d \mu_t =-\int \log J^u T \, d \mu_t<0\, , 
\end{equation}
and
\begin{equation}\label{formula}
P''(t)= \sum_{k\ge 0} \left[ \int  (\log J^s T \circ T^k)  \log J^s T\, d \mu_t - (P'(t))^2 \right] \ge 0\, .
\end{equation}
Moreover, $P''(t)=0$ if and only if $\log J^s T=f-f\circ T+ \int \log J^sT \, d\mu_t$
($\mu_t$ a.e.)  for some  $f\in L^2(\mu_t)$.
Finally, both $t \mapsto \int \log J^uT \, d\mu_t$ and
$t \mapsto h_{\mu_t}$ are decreasing functions of $t$.
\end{theorem}

The formula for $P'(t)$ in \eqref{formula0} implies that, if
there exist $t_a\ne t_b$ in $[0,t_*)$ such that $\mu_{t_a}= \mu_{t_b}$,   then
$P(t)$ is not strictly convex: indeed, $P'(t)$ is constant on $[t_a, t_b]$.
By analyticity, we  then deduce that 
$P(t)$ must be affine on $(0,t_*)$. 
Therefore, we get an immediate corollary of Theorem~\ref{strconv}:
\begin{cor}\label{cor1}
If there exist $t_a\ne t_b$ in $(0,t_*)$ such that $\mu_{t_a}= \mu_{t_b}$
then $P(t)$ is affine on $(0,t_*)$,
and $\log J^sT$ is $\mu_t$ a.e.
cohomologous to its average $\int \log J^sT \, d\mu_t$ for all $t\in (0,t_*)$.
\end{cor}
We expect that there does not exist any Sinai billiard table such that
$\log J^sT$ is  $\mu_t$ a.e. cohomologous to a constant on $M'$ for some $t\in [0,t_*)$.
If we only want to verify that $\mu_0 \neq \mu_1=\musrb$, it is enough to show
that $P''(1)\ne 0$. Note that in
\cite{max}, assuming sparse recurrence (see Definition~\ref{sparse}), we showed that $\mu_0=\musrb$ (i.e., $\mu_0=\mu_1$)
only if  $\frac{1}{p} \log |\det (DT^{-p}|_{E^s}(x))|= P(0)$
for every nongrazing periodic orbit $T^p(x)=x$.

\medskip

The proof of analyticity of $P(t)$
via analyticity of $\cL_t$ in Theorem~\ref{strconv} gives  the following two corollaries.

\begin{cor}[Continuity of Equilibrium States]
\label{contin eq}
For each $\psi \in C^1(M)$, $\mu_t(\psi)$ is analytic for $t \in (0, t_*)$.  Moreover, 
the measures $\mu_t$ vary continuously in the weak topology.
\end{cor}

\begin{cor}[Uniform Rates of Mixing]
\label{unifrate}
The exponential rate of mixing of $\mu_t$ 
for $C^1$ observables is  uniformly bounded away from $1$ in any compact
subinterval of $(0,t_*)$.
\end{cor}

In addition, the proof of the claim on $P''(t)=0$ in Theorem~\ref{strconv}  gives:\footnote{Our approach gives other limit theorems (large deviation
estimates, invariance principles, see \cite[Sect. 6]{dz1}).}
\begin{cor}[Central Limit Theorem]
\label{CLT}For any $t\in (0,t_*)$ such that $P''(t)\ne 0$, setting
$\chi_t:= P'(t)$ and
$\sigma_t:=P''(t)$,
we have 
$
\lim_{k \to\infty}
\mu_t\bigl (\frac 1 {\sqrt k}\sum_{j=0}^{k-1}(\log J^s T -\chi_t) \circ T^j\le z
\bigr )
=\frac 1 {\sqrt {2\pi \sigma_t}}
\int_{-\infty}^ze^{-{v^2}/{(2\sigma_t^2)}}\, dv\, ,
$ for any $z\in \bR$.
\end{cor}

\smallskip

We next motivate \emph{heuristically} the choice of the weight
$1/|J^sT|^{1-t}$ in \eqref{eq:L}, by analogy with the theory for
  smooth hyperbolic $T$.  
For a transitive Anosov diffeomorphism $T$, the transfer operator
  whose maximal left and right
eigenvectors on an anisotropic Banach space give rise to $\mu_t$ is 
$\widetilde \cL_t(f )= \bigl (f/(|J^uT|^t J^sT)\bigr )\circ T^{-1}$
(see \cite{GL2} or \cite[Chapter 7]{Bbook}). A  coboundary argument,
reflecting the fact that $C^1$ functions are interpreted
as distributions via integration with respect to the SRB measure 
$\musrb= (2  |\partial Q|)^{-1} \cos \vf \, drd\vf$ here (see below
Proposition~\ref{embeds}),
but with respect to Lebesgue in  \cite{GL2, Bbook}, will
 replace $1/(|J^uT|^t J^sT)$
 by $1/|J^sT|^{1-t}$: Indeed, \eqref{eq:jacs} gives (on $M'$)
\begin{align}
\nonumber 
- \log \big( |J^uT|^t J^sT \big) &= - \log |J^sT J^uT|^t - \log |J^sT|^{1-t}
\\
\label{cobb} &= - t \log \left( \frac{E \, \cos \vf}{(E \cos \vf) \circ T}  \right) - (1-t) \log J^sT \, .
\end{align}
The first term of \eqref{cobb} is a coboundary. Thus 
we can expect that the  operators  $\widetilde \cL_t$
and $\cL_t$ from \eqref{eq:L}
have isomorphic spectral data, which motivates intuitively our study of $\cL_t$.

\medskip
 The rest of the paper is organized as follows.  In Section~\ref{sec:top}, after defining our notion
of topological pressure $P_*(t,g)$ that we will connect to the measure-theoretic pressure $P(t,g)$, we 
 state our strong variational principle and prove the preliminary result that
$P_*(t,g) \ge P(t,g)$.
In Section~\ref{GrL}, we carry out the main growth lemmas and estimates needed to
prove the exact exponential growth of the topological complexity $Q_n(t,g)$ (defined in Sect.~\ref{altt}).  These
estimates are uniform for $t \in [t_0, t_1] \subset (0, t_*)$.
In Section~\ref{sec:spec}, we define the Banach spaces on which our operators $\cL_t$ act
and prove inequalities which furnish  a spectral gap, again uniform
for $t \in [t_0, t_1]$.
Section~\ref{finall} establishes the main properties of the measure $\mu_t$ constructed from the 
maximal eigenvectors of $\cL_t$ and $(\cL_t)^*$; in particular, $\mu_t$ is the unique
equilibrium state with pressure equal to $P_*(t)$.  Its existence provides the strong
version of the variational principle $P_*(t) = P(t)$ and completes the proof of Theorem~\ref{thm:equil}
(see also the sketch provided after the statement of Theorem 1.1 above).
Finally, Section~\ref{analytic} proves the analyticity of the pressure function and conditions for its
strict convexity as stated in Theorem~\ref{strconv}.


\section{Topological Formulation $P_*(t,g)$ for $P(t,g)$. Variational Principle
(Theorem~\ref{thm:geo var})}
\label{sec:top}

\subsection{Hyperbolicity and Distortion.  $\cW^s$, $\hW^s$,  $\cW^s_\bH$,  $\hW^s_H$. Families
$\cM_{-k}^n$, $\cM_{-k}^{n,\bH}$}
\label{defstuff}

For $n>0$, following \cite{max}, define $\cM_0^n$  to be the
set of maximal connected components of
$M \setminus \cS_n$, and $\cM_{-n}^0$ to be the maximal connected components of 
$M \setminus \cS_{-n}$.  Set $\cM_{-k}^n = \cM_{-k}^0 \bigvee \cM_0^n$.
Note that if $A \in \cM_0^n$, then $T^kA \in \cM_{-k}^{n-k}$ for each $0 \le k \le n$, and $T^kA$ is a union of elements of
$\cM_{-k}^0$ for each $k > n$.

To control distortion, we  introduce homogeneity strips whose spacing depends on 
$t_0\in (0,1)$ if $t\ge t_0$.  Choose\footnote{The standard choice for $t=1$ is $q=2$.} $q = q(t_0)>1$ such that  $qt_0 \ge 2$.  
For fixed $k_0  \in \mathbb{N}$ define
\begin{equation}
\label{eq:H}
\mathbb{H}_k = \{ (r, \vf) \in M : (k+1)^{-q} \le \frac{\pi}{2} - \vf  <  k^{-q} \} \, , \quad \mbox{for $k \ge k_0$,}
\end{equation}
and similarly $\mathbb{H}_{-k}$ is defined approaching $\vf = -\pi/2$.
A  connected component of $\mathbb{H}_k$, for some $|k|\ge k_0$,
or of the set 
$\bH_0=\{ (r, \vf)  : k_0^{-q} \le 
\min\{\frac{\pi}{2} - \vf, \frac{\pi}{2} + \vf\} \}$
is called a  \emph{homogeneity strip.} We let $\cH$ denote the partition of $M$ into homogeneity strips.
Let $\cS_0^\mathbb{H} = \cS_0 \cup (\cup_{|k| \ge k_0} \partial \mathbb{H}_k )$ and, for $n \in \mathbb{Z} $, let 
$\cS_n^\mathbb{H} = \cup_{i=0}^{-n} T^i\cS_0^\mathbb{H}$
 denote the \emph{extended singularity set
for $T^n$.}

Fix\footnote{The index  $k_0=k_0(t_0,t_1)$ and the length scale  $\delta_0= \delta_0(t_0,t_1)<1$ will be chosen
in Definition~\ref{def:k0d0}.} $\delta_0\in (0,1)$. 
Let
$\cW^s$ denote the set of all nontrivial connected subsets $W$
of local stable manifolds of $T$ of  length at most $\delta_0$.   Such curves have curvature bounded above
by a fixed constant \cite[Prop~4.29]{chernov book}, and
$T^{-n}\cW^s = \cW^s$ for all $n\ge 1$, up to subdivision of curves  according to the length scale $\delta_0$.
Let $\cW^s_\bH\subset \cW^s$ denote the set of 
nontrivial connected subsets $W$ of elements of $\cW^s$  with the property that
$T^nW$ belongs to a single homogeneity strip for each  $n \ge 0$.  Such curves are called\footnote{In \cite{chernov book}, these curves are called $H$-manifolds. This strong notion of homogeneity is needed   to prove  H\"older continuity of the conditional densities of  the SRB measure decomposed along stable manifolds --  needed to get valid test functions for our spaces --- using  the asymptotic limit of the ratio of stable Jacobians,   forward iterates must be contained in a single homogeneity strip (so that the ratio remains bounded).} \emph{homogeneous stable manifolds.}

We call a $C^2$ curve $W \subset M$ (cone) stable if at each point $x$ in $W$, the tangent 
vector $\cT_xW$ to $W$ lies in $\cC^s$.  
We denote by
$\hW^s$ the set of (cone) stable curves with second derivative bounded by a constant chosen sufficiently large
(\cite[Prop~4.29]{chernov book}) so that
$T^{-n} \hW^s \subset \hW^s$ for all $n\ge 1$, up to subdivision of  curves according to  $\delta_0$.   Finally,
${\hW^s_H} \subset \hW^s$ is the set  of elements
of $\hW^s$
 contained in a single homogeneity strip, while $\cW^s_H$ is the set of  elements of 
$\cW^s$ that are contained in a single homogeneity strip.
  Such curves are called  \emph{weakly homogeneous} (cone) stable curves 
and stable manifolds, respectively. 
 Obviously, $\cW^s_\bH \subset \cW^s_H \subset \cW^s \subset \hW^s$ and  $\cW^s_H \subset \hW^s_H$.

For every  $W\in \hW^s$, let $C^1(W)$
denote the space of $C^1$ functions on $W$, and 
for every $\eta\in (0,1)$
let $C^\eta(W)$ denote the closure\footnote{Using
the closure of $C^1$  will give injectivity of the inclusion
of the strong space in the weak one in Proposition~\ref{embeds}.} of $C^1(W)$ for the
$\eta$-H\"older norm  
defined by
\begin{equation}
\label{eq:holder def}
|\psi|_{C^\eta(W)}=\sup_W|\psi|+ H_W^\eta(\psi)\, ,\,\, H_W^\eta(\psi)=\sup_{\substack{x, y \in W \\ x \neq y}}\frac{|\psi(x)-\psi(y)|}{d(x,y)^\eta}\, .
\end{equation}

  The following lemma
extends standard distortion bounds for homogeneous curves to all exponents $t>0$.
(See Lemma~\ref{distlog} for a further generalisation.)

\begin{lemma}
\label{lem:distortion}
There
exists $\bar \delta_0>0$
and $C_d >0$,  depending on $k_0$ and $q$, 
such that for all $\delta_0<\bar \delta_0$,
all  $n \ge 0$,
 and any $W \in T^{-n} \hW^s$ such
that $T^iW \in  \hW^s_H$ for each $i = 0, \ldots, n-1$, we have 
\[
\biggl| 1 - \frac{|J_WT^n(x)|^t}{|J_WT^n(y)|^t} \biggr| \le 2^t C_d d(x,y)^{1/(q+1)} \, ,
\, \, \, \forall \, x, \, y \in W \, , \, \, \forall t >0\, , 
\]
where $J_WT^n(x)=|\det (DT^n_x|\cT_xW)|$ denotes the Jacobian of $T^n$ along $W$, 
and $d(\cdot, \cdot)$ denotes distance in $M$. 
\end{lemma}

\begin{proof}
There exists $C_d <\infty$,  independent
of $\delta_0$, but depending on $k_0$ and $q$
such that
\begin{equation}
\label{eq:distortion}
\left| 1 - \frac{J_WT^n(x)}{J_WT^n(y)} \right| \le C_d d(x,y)^{1/(q+1)} \,  ,
\, \, \, \forall \, x, \, y \in W\, , \, \,
\forall \, W \mbox{ as in the lemma.} 
\end{equation}
(For $q=2$, see e.g. \cite[Lemma 5.27]{chernov book} or \cite[App. A]{dz1}. The proofs there give \eqref{eq:distortion} for all $q>1$.)

 For $t \le 1$, the estimate is an immediate consequence of
\eqref{eq:distortion}, since for all $A>0$, we have
$|1 - A^t| \le |1-A|$.  
Now choose $\bar \delta_0$ such that $C_d \bar \delta_0^{1/(q+1)} \le 3/4$. Then,
for $t >1$, we set $A = \frac{J_WT^n(x)}{J_WT^n(y)}$.  By \eqref{eq:distortion}, this implies that
$1/4 \le A \le 2$ if $\delta_0<\bar \delta_0$.  For $A$ in this range, we have, again using \eqref{eq:distortion}, that
$|1-A^t| \le 2^t |1-A| \le 2^t C_d d(x,y)^{1/(q+1)}$.
\end{proof}
\smallskip

Next, recalling that $\cS_k^\mathbb{H} = \cup_{i=0}^{-k} T^i\cS_0^\mathbb{H}$, define for $n \ge 1$,
\begin{equation}
\label{eq:Mn}
\begin{split}
\cM_0^{n, \bH} & = \mbox{maximal connected components of 
$M \setminus \left(  T^{-n}\cS_0 \cup \cS_{n-1}^\mathbb{H} \right)$} \, ,\\
\cM_{-n}^{0, \mathbb{H}} & = \mbox{maximal connected components of 
$M \setminus \left( \cS_0 \cup T (\cS_{-(n-1)}^\mathbb{H}) \right)$ } \, , \\
\cM_{-k}^{n, \mathbb{H}} &= \cM_{-k}^{0, \mathbb{H}} \bigvee \cM_0^{n, \mathbb{H}}\, ,\,\, k \ge 1\,  .
\end{split}
\end{equation}
We comment on the  use of $\cS_0^{\mathbb{H}}$ in \eqref{eq:Mn}.  
First notice (just like for the sets 
$\cM_{-k}^{n}$ defined in the beginning of this subsection) that  if $A \in \cM_0^{n, \mathbb{H}}$, then $T^kA \in \cM_{-k}^{n-k, \mathbb{H}}$ for each $0 \le k \le n$, and $T^kA$ is a union of elements of
$\cM_{-k}^{0, \mathbb{H}}$ for each $k > n$.
Next, if $W \in  \hW^s_H$ is such that $V = T^{-1}W$ is a single curve, then $J_WT^{-1}(x) \approx 1/\cos \vf(T^{-1}x)$
while $J_VT(y) \approx \cos \vf(y)$.  Thus by \eqref{eq:distortion},
the definitions in \eqref{eq:Mn} guarantee that for any $W \in \hW^s_H$ such that $W \subset A \in \cM_{-n}^{0, \mathbb{H}}$,
the Jacobian 
$J_WT^{-n}$ has bounded distortion on $W$, while
$J_{T^{-n}W}T^n$ has bounded distortion on $T^{-n}W$ (which is contained in a single element of $\cM_0^{n, \mathbb{H}}$).

\smallskip

 We shall also need the following distortion bound.

\begin{lemma}[Distortion Relative to $\cM_{-n}^{0,\bH}$]
\label{lem:comparable}
There exists $C>0$ such that for all $n \ge 1$, for all $U, V \in  \hW^s_H$ such that
$U, V \subset A \in \cM_{-n}^{0, \mathbb{H}}$, and  all\footnote{ $\bar U$ denotes the closure of
$U$ in $M$.  The distortion bounds on $U$ and $V$ extend trivially to the boundaries of homogeneity strips, but not to real singularity lines, hence $\bar U \setminus \cS_{-n}$.} $u \in \bar U \setminus \cS_{-n}$, $v \in \bar V \setminus \cS_{-n}$,
\[
\left| \log \frac{J_UT^{-n}(u)}{J_VT^{-n}(v)} \right| \le C \, .
\]
\end{lemma}

The bound above is more general (and weaker) than the usual distortion bound along stable curves given by \eqref{eq:distortion} or 
between stable curves given by \cite[Theorem~5.42]{chernov book} (or more generally \cite[Appendix A]{dz1}) since we do not
assume that the points $u$, $v$ in $\bar A$,
with $A \in \cM_{-n}^{0, \mathbb{H}}$, lie on the same stable or unstable curve.

\begin{proof}
Let $n \ge1$, $u \in \bar U$, $v \in \bar V$, be as in the statement of the lemma.
Define $u_i = T^{-i}u$, $v_i = T^{-i}v$ for $i = 0, \ldots, n$,
and notice that $u_i, v_i$ belong to the closure
of the same element of $\cM_{-n+i}^{i, \mathbb{H}}$.  
By the uniform hyperbolicity of $T$, for $i = 0, \ldots, n$,  if $A \in \cM_{-n+i}^{i, \mathbb{H}}$, then
$\diam^u(\bar A) \le C \Lambda^{-i}$ and $\diam^s(\bar A) \le C \Lambda^{-n+i}$, 
where $\diam^u(B)$
is the maximum length of an unstable curve in $B$, and $\diam^s(B)$  is  the maximum length of a stable curve
in $B$.  Thus, due to the uniform transversality
of $\cC^s$ and $\cC^u$, we have
\begin{equation}
\label{eq:diam}
d(u_i, v_i) \le \bar C \max \{ \Lambda^{-i} , \Lambda^{-n+i} \} \, .
\end{equation}

By the time reversal of \cite[eq.~(5.24)]{chernov book}, we have that
\begin{equation}
\label{eq:stable jac}
\log J_{U_i}T^{-1}(u_i) =  \log
\frac{ \cos \vf(u_i) + \tau(u_{i+1})(\cK(u_i) - \cV(u_i))}
{ \cos \vf(u_{i+1})}
+ \log \frac{\sqrt{1+ \cV(u_{i+1})^2}}{\sqrt{1+\cV(u_i)^2}} \, ,
\end{equation}
where $\cV(u_i) = \frac{d\vf}{dr}(u_i) < 0$ is the slope of the tangent line to $U_i$ at $u_i$.
Summing over $i$, the last term above telescopes and the sum is uniformly bounded away from 0
and $\infty$, giving,
\begin{equation}
\label{eq:jac}
\left| \log \frac{J_UT^{-n}(u)}{J_VT^{-n}(v)} \right| \le C + \sum_{i=0}^{n-1} \left| \log \frac{\cos \vf(v_{i+1})}{\cos \vf(u_{i+1})} \right|  
+ \left| \log \frac{\cos \vf(u_i) + \tau(u_{i+1})(\cK(u_i) - \cV(u_i))}{\cos \vf(v_i) + \tau(v_{i+1})(\cK(v_i) - \cV(v_i))} \right|
\end{equation} 

Since $u_{i+1}$, $v_{i+1}$ lie in the same homogeneity strip for each $i$, using \eqref{eq:H} we have 
\begin{equation}
\label{eq:cos}
\left| \log \frac{\cos \vf(v_{i+1})}{\cos \vf(u_{i+1})} \right|
\le C \frac{|\vf(u_{i+1}) - \vf(v_{i+1})|}{\cos \vf(u_{i+1})}
\le C  d(u_{i+1}, v_{i+1})^{1/(q+1)} \, .
\end{equation}

Next, the  terms in the second set on the right-hand side of \eqref{eq:jac} are bounded and the denominator in the expression is at least
$\tau_{\min} \cK_{\min} > 0$.  Moreover, $\cK$ is differentiable while $\tau$
is $1/2$-H\"older continuous.\footnote{We cannot take advantage of the smoother
bounds on $\tau$ given by \cite[eq. (5.28)]{chernov book} since our points $u_i$ and $v_i$ may  lie on different stable
or unstable manifolds.}  Thus following \cite[eq. (5.26)]{chernov book}, we have
\[
\sum_{i=0}^{n-1}\left| \log \frac{\cos \vf(u_i) + \tau(u_{i+1})(\cK(u_i) - \cV(u_i))}{\cos \vf(v_i) + \tau(v_{i+1})(\cK(v_i) - \cV(v_i))} \right|
\le C \sum_{i=0}^{n-1} d(u_{i+1}, v_{i+1})^{1/2} + d(u_i, v_i) + |\Delta \cV_i| \, ,
\]
where $\Delta \cV_i = \cV(u_i) - \cV(v_i)$.  By \eqref{eq:diam}, the sums over all terms in \eqref{eq:jac} involving
$d(u_i,v_i)$ are uniformly bounded in $n$.  It remains to estimate
$\sum_{i=0}^{n-1} |\Delta \cV_i|$.
By \cite[eq.~(5.29)]{chernov book} and \eqref{eq:diam}, we
bound  $|\Delta \cV_i|$ by
\begin{align*}
 C \Big( |\Delta \cV_0| \Lambda^{-i} + \sum_{j=0}^i  \Lambda^{-j} d(u_{i-j}, v_{i-j})^{1/2} \Big)
&\le C \big( |\Delta \cV_0| \Lambda^{-i} + \sum_{j=0}^i \Lambda^{-j}(\Lambda^{(-i+j)/2} + \Lambda^{(-n+i-j)/2} ) \big) \\
& \le C' \big( |\Delta \cV_0| \Lambda^{-i} + \Lambda^{-i/2} + \Lambda^{(-n+i)/2} \big)  .
\end{align*}
Summing over $i$ completes the proof of the lemma.
\end{proof}

\begin{remark}[About $\int_M \log J^uT\, d\mu$]\label{newremark}
We note for further use a consequence of the proof
above: It is clear from \eqref{eq:stable jac} that for any $T$-invariant probability measure $\mu$ with
$\mu(M \setminus M') = 0$,
\[
\int_M \log J^sT \, d\mu = - \infty \iff \int_M \log \cos \vf \, d\mu = - \infty \, ,
\] 
since all other terms in \eqref{eq:stable jac} are bounded away from zero and infinity.  Similarly, the
time reversal of \eqref{eq:stable jac}  (see \cite[eq.~(5.24)]{chernov book}) implies
\[
\int_M \log J^uT \, d\mu = \infty \iff \int_M \log \cos \vf \circ T \, d\mu = -\infty \, .
\]
Thus, applying \eqref{eq:jacs}, for any $T$-invariant
probability measure $\mu$ on $M$,
\begin{equation}
\label{eq:switch potential}
\mbox{if } \mu(M\setminus M')=0 \mbox{ then } \int_M \log J^uT\, d\mu=- \int_M \log J^sT\, d\mu \, ,
\end{equation}
where equality holds also when both sides are infinite.
\end{remark}

\subsection{Topological Formulation $P_*(t,g)$ of the Pressure $P(t,g)$.  Theorem~\ref{thm:geo var}}
\label{altt}

In view of our proof of uniqueness in \S\ref{sec:unique}
(which uses differentiability of the pressure),
for  a  $C^1$  function $g :M\to \bR$   and $n\ge 1$, we set  $S_n g = \sum_{i=0}^{n-1} g \circ T^i$. The hyperbolicity
of $T$  implies the following distortion bounds: There exists
$C_*<\infty$ such that for all $n\ge 1$ and all $W \in \hW^s$ such that  $T^iW \in \hW^s\, ,\, \forall \, 
0 \le i \le n$, 
\begin{equation}
\label{eq:phi dist}
| e^{ S_ng(x) -  S_ng(y)} - 1| \le C_*   | \nabla g |_{C^0} \,  d(x,y)\,,\,  \forall x,y \in W \, . 
\end{equation}

Recalling
\eqref{defM'}, we define 
(aside from \S\ref{sec:unique} we only need  $g\equiv 0$),
\begin{equation}
\label{eq:Q_n}
Q_n(t, g) = \sum_{A \in \cM_0^{n, \bH}} \sup_{x \in A \cap M'} |J^sT^n(x)|^t 
e^{ S_ng(x)}\, , \,\, \,Q_n(t)=Q_n(t,0)\,,\,\,
 n\ge 1\, , 
\end{equation}
and
\begin{equation}
\label{eq:Pstar}
P_*(t, g) = \limsup_{n \to \infty} \frac 1n \log Q_n(t,g)\, , \,\, \, P_*(t):=P_*(t, 0).
\end{equation}

We will show the following result  in Section~\ref{next}:

\begin{proposition}[Topological Pressure]
\label{prop:pressure}
For all  $t>0$  and $g\in C^1$, we have\footnote{Recall our convention
that $\int_M \log J^uT \, d\mu=\infty$ if $\mu(M\setminus M')>0$.} $P(t,g) \le P_*(t,g)$.
\end{proposition}

For $t_*>1$  given by \eqref{eq:t* def}, the analysis carried out in
Sections 3--5 will yield:

\begin{theorem}[(Strong\footnote{By "strong" we mean that the supremum is a maximum, and it is attained at a unique measure.}) Variational Principle]
\label{thm:geo var}
If $t \in (0,t_*)$, then $P_*(t) = P(t)$ and the supremum is attained at the unique invariant
measure $\mu_t$ from Theorem~\ref{thm:equil}.
\end{theorem}

\begin{proof}
This follows from Proposition~\ref{prop:pressure}, 
Theorem~\ref{thm:spectral}, Corollary~\ref{cor:max}, and  Proposition~\ref{prop:unique}.
\end{proof}

Theorem~\ref{thm:variational} will give the generalisation of the above strong form of the
variational principle to $P(t,g)=P_*(t,g)$ for suitable $g$.

\smallskip

We first establish basic properties of $P_*(t,g)$:

\begin{proposition}
\label{prop:Ptilde}
For each $t>0$ and $g \in C^1$ the limsup \eqref{eq:Pstar} defining $P_*(t,g)$ is a limit in   
$(-\infty, \infty)$. The function
 $t\mapsto P_*(t,g)$ is convex and strictly decreasing on $(0,\infty)$.
\end{proposition}

\begin{remark}
It is not hard to show,  using Lemma~\ref{lem:comparable}, that
for each $t > 0$, there exists $C_D>0$ such that 
$Q_n(t) \le C_D^t \sum_{A \in \cM_0^{n, \mathbb{H}}} \inf_{x \in A \cap M'} |J^sT^n(x)|^t$ for all  $n \ge 1$, so that replacing
the supremum by an infimum in the definition of $Q_n(t)$ 
does not change the value of $P_*(t)$.
\end{remark}

\begin{proof}[Proof of Proposition~\ref{prop:Ptilde}]
We first set $g=0$. The partition $\cM_0^1$ is finite, and each
element of  $\cM_0^1$ is subdivided by curves in $\cS_0^{\mathbb{H}}$ to comprise a
 union of elements of $\cM_0^{1, \bH}$, according to \eqref{eq:Mn}.
Thus,
\[
Q_1(t) = \sum_{A \in \cM_0^{1, \mathbb{H}}} \sup_{x \in A \cap M'} |J^sT(x)|^t
\le C \sum_{E \in \cM_0^1} \sum_k \sup_{x \in \mathbb{H}_k} |\cos \vf(x)|^t
\le C \# \cM_0^1 \sum_k k^{-qt} \, , 
\]
where the sum over $k$ includes $k=0$ and $|k| \ge k_0$;    the sum converges since $qt \ge 2>1$.   

We next show that $Q_n(t)$ is submultiplicative: 
\begin{align}
\nonumber Q_{n+k}(t) & = \sum_{A \in \cM_0^{n, \mathbb{H}}} \sum_{\substack{ B \in \cM_0^{n+k, \mathbb{H}} \\ B \subset A} } \sup_{x \in B \cap M'}
|J^sT^{n+k}(x)|^t \\
\label{inner} & \le \sum_{A \in \cM_0^{n, \mathbb{H}}} 
\sup_{y \in A \cap M'} |J^sT^n(y)|^t  
\sum_{\substack{ B \in \cM_0^{n+k, \mathbb{H}} \\ B \subset A} }   \sup_{x \in B \cap M'}  |J^sT^k(T^nx)|^t \, ,\,\,
\forall k, n\ge 1 \, .
\end{align}
Notice that if $B , B' \subset A \in \cM_0^{n, \mathbb{H}}$ are distinct elements of $\cM_0^{n+k, \mathbb{H}}$, 
then $T^nB, T^nB' \in \cM_{-n}^{k, \mathbb{H}} = \cM_{-n}^{0, \mathbb{H}} \bigvee \cM_0^{k, \mathbb{H}}$ 
are both contained in $T^nA \in \cM_{-n}^{0, \mathbb{H}}$ and so must
lie in distinct elements of
$\cM_0^{k, \mathbb{H}}$.  Thus the inner sum in \eqref{inner} is bounded by $Q_k(t)$ for each $A$, and the outer sum is bounded
by $Q_n(t)$, proving submultiplicativity. 
If $g\ne 0$, it is easy to see that we also have
$Q_{n+k}(t, g) \le Q_n(t, g) Q_k(t, g)$. Therefore, since $Q_1(t,g)<\infty$, the sequence in \eqref{eq:Pstar} converges to a limit  in $[-\infty,\infty)$.

To see that $P_*(t, g)>-\infty$, 
let $x_p$ be a periodic point of period $p$ with no tangential collisions,\footnote{Such a periodic
point always exists.  For example, since two adjacent scatterers are in convex opposition, there is a period 2 orbit whose trajectory is normal to both scatteres.} and let 
$\chi_p^-$ denote the negative Lyapunov exponent of $x_p$.  Then,
$Q_{np}(t, g) \ge |J^sT^{np}(x_p)|^t e^{ S_{np}  g(x_p)} =  e^{npt \chi_p^-} e^{ n S_p g(x_p)} $,
and so $P_*(t, g) \ge t \chi_p^- + \frac{1}{p} S_p g(x_p) > -\infty$.  

To prove convexity, pick $t, t' > 0$ and $\alpha \in [0,1]$.  Then using 
the H\"older inequality,
\begin{align*}
&Q_n(\alpha t + (1-\alpha) t', g)  =
\sum_{A \in \cM_0^{n, \bH}} \sup_{x \in A \cap M'} |J^sT^n|^{\alpha t + (1-\alpha) t'} 
e^{\alpha S_ng(x)}e^{(1-\alpha) S_ng(x)}  \\
&\quad \le \bigl(\sum_{A \in \cM_0^{n, \bH}} \sup_{x \in A \cap M'} |J^sT^n|^{t}
e^{ S_ng(x)} \bigr)^\alpha
\bigl(\sum_{A \in \cM_0^{n, \bH}} \sup_{x \in A \cap M'} |J^sT^n|^{t'} e^{ S_ng(x)} \bigr)^{1-\alpha}  
= Q_n(t,g)^\alpha Q_n(t',g)^{1-\alpha} \, .
\end{align*}
Taking logarithms, dividing by $n$, and letting $n \to \infty$ proves convexity.

Next, fixing $t>0$ and applying
\eqref{eq:hyp}, we find for $s>0$,
\[
Q_n(t+s, g) = \sum_{A \in \cM_0^{n, \bH}} \sup_{x \in A \cap M'} |J^sT^n|^{t+s}
e^{ S_ng(x)}
\le C_1^{-s} \Lambda^{-ns} Q_n(t, g) \, ,
\] 
so that $P_*(t+s,g) \le  P_*(t,g) - s\log \Lambda$, that is,  $P_*(t,g)$ is strictly decreasing in $t$.
\end{proof}


\subsection{Proof that $P_*(t,g)\ge P(t,g)$ (Proposition~\ref{prop:pressure})}
\label{next}

If $\cQ$ is a partition of $M$ we let
$\Int \cQ$ denote the set of interiors of elements of $\cQ$.
In \cite{max},  we worked with $\cP$, the (finite) partition of $M$ into maximal connected sets on which
$T$ and $T^{-1}$ are continuous, noticing
that the set $\Int \cP$  coincides with
 $\cM^1_{-1}$, while the refinements $\cP_{-k}^{n}=\bigvee_{i=-k}^n T^{-i}\cP$ may also contain
 isolated points if three or more scatterers have a common tangential trajectory (see \cite[Fig.1]{max}).
(Note that $\cP$ is a set-theoretical partition: zero measure sets do not need to
be ignored.)
We also observed that $\cP$ is a generator for any $T$-invariant Borel probability
measure $\mu$,  since $\bigvee_{i=-\infty}^\infty T^{-i} \cP$ separates\footnote{In fact, all points $x\ne y$ may be separated while the definition of a generator in \cite{parry}
allows a zero measure set of
pathological pairs.} points in the compact metric space $M$: if $x\ne y$  there exists $k\in \bZ$ such that $T^k(x)$
and $T^k(z)$ lie in different elements of $\cP$.
Let  $\bar \cP$ be the partition of $M$ into maximal connected sets on which
$T$ is continuous. Then  $\cP=\bar \cP\bigvee T(\bar \cP)$,  so $\bar \cP$ is also a generator
for $T$. We have $\Int \bar \cP=\cM^1_{0}$. More generally, $\Int(\bigvee_{k=0}^{n-1} T^{-k} \bar \cP) = \cM_0^n$ for $n\ge 1$.

\begin{proof}[Proof of Proposition~\ref{prop:pressure}]
If a $T$-invariant probability measure
$\mu$ gives positive weight to   $M \setminus M'$  or, more generally,
if $\int_M \log J^u T \, d\mu = \infty$, then $P_\mu(t, g) = -\infty$, so 
$P_\mu(t, g) <  P_*(t, g)$.  
We can thus assume without loss of generality that $\mu$ is a $T$-invariant probability
measure  with  $\int_M \log J^uT \, d\mu < \infty$, in particular  $\mu(\cS_n)=0$ for each $n \in \mathbb{Z}$. Then
\begin{equation}
\label{remem}H_\mu\left( \bigvee_{k=0}^{n-1} T^{-k}\bar \cP \right) = H_\mu( \cM_0^n)\, ,
\end{equation}
 since the boundary of any element of $\bigvee_{k=0}^{n-1} T^{-k} \bar \cP$ is contained in
$\cS_n$.

Since $\bar \cP$ is a generator, we have
$h_{\mu}(T) = h_{\mu}(T, \bar \cP)$
 for any  $T$-invariant probability measure $\mu$ on $M$  (see e.g. \cite[Theorem~4.17]{walters}).
Then,  using
\eqref{eq:switch potential},
we find, adapting the classical argument
(see e.g. \cite[Prop 9.10]{walters}), that
\begin{align*}
h_\mu(T, \bar \cP) &- t \int_M \log J^uT \, d\mu = \lim_{n \to \infty} \frac 1n H_\mu\left(\bigvee_{k=0}^{n-1} T^{-k}\bar \cP\right) 
+ \lim_{n \to \infty} \frac 1n  \int_M  t \sum_{k=0}^{n-1}\log J^sT \circ T^k \, d\mu\\
&\le \lim_{n \to \infty} \frac 1n
\biggl ( \sum_{ A\in \cM_0^n }
\mu(A) \big[ -\log \mu(A) +  \sup_{A\cap M'} t \log J^sT^n   \big] \biggr ) \\
&\le \lim_{n \to \infty} \frac 1n
 \sum_{A \in  \cM_0^n}  
\mu(A) \log \frac{\sup_{A \cap M'} |J^sT^n|^t}{\mu(A)}   
 \le \lim_{n \to \infty} \frac 1n \log \sum_{A \in \cM_0^n} \sup_{A \cap M'} |J^sT^n|^t  \, ,
\end{align*}
where we used \eqref{remem} in the third line,  and the convexity of the logarithm 
in the  fifth line. 
Finally, notice that each element of $\cM_0^n$ is a union of elements of $\cM_0^{n, \bH}$,
modulo the boundaries of homogeneity strips.  But since the distortion bound
Lemma~\ref{lem:comparable} extends to the boundaries of homogeneity strips, we have
\[
\sup_{A \cap M'} |J^sT^n|^t 
= \sup_{\substack{ B \in \cM_0^{n, \bH} \\ B \subset A}} \sup_{B \cap M'} |J^sT^n|^t
\le \sum_{\substack{ B \in \cM_0^{n, \bH} \\ B \subset A}} 
\sup_{B \cap M'} |J^sT^n|^t \, ,
\]
for each $A \in \cM_0^n$.
Using this bound in the previous estimate and applying
 Proposition~\ref{prop:Ptilde} implies
$h_\mu(T) - t \int_M \log J^uT \, d\mu \le P_*(t)$ for every $T$-invariant
probability measure $\mu$. 

If $g\ne 0$, we may write
\begin{align*}
h_\mu(T, \bar\cP) + \int (t\log J^sT +  g)  \, d\mu  
& = \lim_{n \to \infty} \frac 1n H_\mu \left( \bigvee_{k=0}^{n-1} T^{-k}\bar\cP \right) +
\lim_{n \to \infty} \frac 1n \int S_n(t \log J^sT +  g) \, d\mu \\
& \le \lim_{n \to \infty} \frac 1n \log \sum_{A \in \cM_0^n} \sup_{A \cap M'} |J^sT^n|^t 
e^{ S_ng} \, ,
\end{align*}
and this last expression is bounded by $P_*(t,g)$ by the same reasoning as above, 
using that the analogue of Lemma~\ref{lem:comparable} holds for $e^{S_n g}$:
Recalling \eqref{eq:phi dist},  for all $n >0$, 
all $A \in \cM_0^{n,\bH}$, and $x,y \in A$,
since the diameter of $T^iA$ is bounded by \eqref{eq:diam},
\begin{equation}
\label{eq:phi A}
e^{S_n g(x) - S_ng(y)} \le 1+ \bar C\, C_* \cdot | \nabla g |_{C^0} \, .
\end{equation}
\end{proof}


\section{Growth Lemmas}
\label{GrL}

 In this section, after preliminaries
in \S\ref{1step}, introducing in particular the contraction rate $\theta$
and sets $\cG_n(W)$
appearing when iterating the transfer operator $\cL_t$, we prove a series of growth
and complexity lemmas which will allow us to control   the
sums over 
$\cG_n(W)$ for 
$W \in \hW^s$.  This culminates in the lower bound of
Proposition~\ref{prop:lower bound}, which 
implies  
exact exponential growth (Proposition~\ref{prop:exact}) of
$Q_n(t,g)$.  (This exact exponential growth  is essential
to control the peripheral spectrum of $\cL_t$.)

Since $J^sT$ is not bounded away from zero, we shall use different strategies for $t\in (0,1]$ and 
$t>1$.  Several important growth lemmas are proved for $t \in (0,1]$ in 
Sections~\ref{sec:t<1} and \ref{sec:complexity}.  We then use the results for $t\le 1$ to bootstrap an analogous
set of lemmas for $t>1$ in Sections~\ref{sec:t>1}, \ref{sec:complexity},  and \ref{sec:boot}.


\subsection{One-Step Expansion: $\theta(t_1)$. Choice of  $q(t_0)$,  $k_0(t_0,t_1)$, $\delta_0(t_0,t_1)$.  $\cG_n(W)$, $\cI_n(W)$}
\label{1step}

We begin by proving an adaptation of the one-step expansion
(see e.g. \cite[Lemma 5.56]{chernov book}) for our choice of
potential and homogeneity strips.  
 Using the notation from \eqref{eq:stable jac},  recall $t_* > 1$ from \eqref{eq:t* def}, and the
adapted metric from \cite[Section~5.10]{chernov book}:
$$\| dx \|_*
= \frac{\cK+ |\cV|}{\sqrt{1+\cV^2}} \|dx\|\, .
$$ 

\begin{lemma}[One-Step Expansion]
\label{lem:one step}
For $t_1 \in (1, t_*)$, fix\footnote{This is  possible by definition of $t_*$
and Proposition~\ref{prop:pressure}.  It implies $\theta^t < e^{P_*(t)}$ for $t \le t_1$ since $P_*(t)$ is decreasing.} $\theta \in (\Lambda^{-1}, \Lambda^{-1/2})$ such that 
$\theta^{t_1} <  e^{P_*(t_1)}$.
Then for each $\bar t_0\in (0,1)$ and $q>2/\bar t_0$, there exist
$ \bar k_0 =  \bar k_0(\bar t_0, t_1,q) \ge 1$ and  $ \bar \delta_0 =  \bar \delta_0(\bar t_0, t_1,q)>0$ such that 
\begin{equation}
\label{eq:theta}
\sum_i |J_{V_i}T|_{C^0(V_i),*}^t < \theta^t\, ,\; \;  \forall W \in \hW^s
\mbox{ with } |W|<\bar \delta_0  \, ,\, \,
\forall t\ge \bar t_0 \, ,
\end{equation}
where the $V_i$ range over the maximal connected weakly $(q, \bar k_0)$-homogeneous\footnote{By weakly $(q, \bar k_0)$-homogeneous, we mean weakly homogeneous as defined in Sect.~\ref{defstuff}
using the parameters $q$ and $\bar k_0$ for the homogeneity strips.  Note that $W$ is not necessarily weakly homogeneous, but each $V_i$ is.} components of $T^{-1}W$, 
and $|J_{V_i}T|_{C^0(V_i),*}$ denotes
the maximum on $V_i$ of the Jacobian of $T$ along $V_i$
for the metric $\| \cdot \|_*$.
\end{lemma}

\begin{proof} 
Note that $|J_{V_i}T|_{C^0(V_i),*} \le \Lambda^{-1}$ and, if $V_i \subset \mathbb{H}_k$, then $|J_{V_i}T|_{C^0(V_i),*} \le Ck^{-q}$ for
some
$C > 0$ \cite[eq.~(5.36)]{chernov book}.  
There exists $\bar\delta > 0$ such that if $W \in \hW^s$ with $|W| \le \bar\delta$, then
$T^{-1}W$ has at most $\frac{\tau_{\max}}{\tau_{\min}} + 1$ connected components, and all but at most 
one component experience nearly tangential collisions (see \cite[Sect.~5.10]{chernov book}).

For  $\bar t_0\in(0,1)$ and $q > 2/\bar t_0$, choose $\bar k_0 = \bar k_0(\bar t_0, t_1, q)$ such that
\begin{equation}
\label{eq:k0}
\Lambda^{-\bar t_0} + \frac{\tau_{\max}}{\tau_{\min}} \sum_{|k| \ge \bar k_0} C^{\bar t_0} k^{-q\bar t_0}
\le \Lambda^{-\bar t_0} +  \frac{\tau_{\max}}{\tau_{\min}}\,  C^{\bar t_0}\,  \bar k_0^{-1} < 
\theta^{\bar t_0} \, .
\end{equation}
For all $W \in \hW^s$, we have $|T^{-1}W| \le  C' |W|^{1/2}$ \cite[Exercise~4.50]{chernov book} for some
$ C' >0$ independent of $W$.
Next, choose $\bar \delta_0(\bar t_0,  t_1)$ so small that  $C'\bar \delta_0^{1/2} \le
\bar k_0^{-q}$, so that 
if $|W| \le |\bar \delta_0|$, then
each component of $T^{-1}W$ making a nearly tangential collision lies in a union of homogeneity strips
$\mathbb{H}_k$ for $k \ge \bar k_0$. 

Then if $|W| \le \bar \delta_0$, the quantity $\sum_i |J_{V_i}T|_{C^0(V_i),*}^{{\bar t_0}}$ is
 bounded by the left-hand side of \eqref{eq:k0},
proving \eqref{eq:theta}  for $t=\bar t_0$.  Finally,  for all $ t\ge \bar t_0$,
\begin{equation}
\label{eq:contract}
\sup_{\substack{W \in \hW^s \\ |W| \le \bar \delta_0}} \sum_i |J_{V_i}T|_{C^0(V_i),*}^t
\le \Lambda^{-t+\bar t_0} \sup_{\substack{W \in \hW^s \\ |W| \le \bar \delta_0}} \sum_i |J_{V_i}T|_{C^0(V_i),*}^{\bar t_0}
\le \Lambda^{-t+\bar t_0} \theta^{\bar t_0} \le \theta^t \, .
\end{equation}
\end{proof}

We  now choose the parameters defining
$\hW^s$, $\hW^s_\bH$ and  $\cW^s$, $\cW^s_\bH$ depending on $t_0\in (0,1)$,
$t_1 \in (1, t_*)$:

\begin{defin}
\label{def:k0d0}
Given $t_0\in (0,1)$, $t_1\in  (1, t_*)$,  we fix $q(t_0)>1$ such that $qt_0/2 \ge 2$, and fix $\theta=\theta(t_1)$,
   $k_0 = k_0(t_0, t_1, q) := \bar k_0( \tfrac{t_0}{2}, t_1,q)$, 
$\delta_0 = \delta_0(t_0,  t_1,q ) := \bar \delta_0(\tfrac{t_0}{2}, t_1 ,q)$
 as in Lemma~\ref{lem:one step}. 
Reduce $\delta_0$ if needed
so that $C_d \delta_0^{\frac{1}{q+1}} \le 3/4$, with $C_d$  from \eqref{eq:distortion}.  
This choice of $(t_0,t_1)$, $\theta$, $q$, $\delta_0$, and $k_0$ determines the
set of stable curves $\hW^s$, $\hW^s_\bH$ and stable manifolds $\cW^s$, $\cW^s_\bH$. 
\end{defin}

\noindent
Our proofs use sets $\cG_n(W)$, $\cI_n(W)$ associated with $\delta_0$
and $k_0$, and, for $\delta <\delta_0$, also $\cG_n^\delta(W)$, $\cI_n^\delta(W)$:

For $W \in \hW^s$, we let $\cG_1(W)$ denote the maximal, weakly homogeneous, connected components of $T^{-1}W$,
with long pieces subdivided to have length between $\delta_0/2$ and $\delta_0$.  
Inductively, we define\footnote{This definition of $\cG_n(W)$ is  as in
\cite{dz1,dz2}, but different from \cite{max} where homogeneity was
not required.} $\cG_n(W) = \cup_{W_i \in \cG_{n-1}(W)} \cG_1(W_i)$.  
Thus $\cG_n(W)$
is the countable collection of subcurves of $T^{-n}W$ subdivided according to the extended singularity set
$\cS_{-n}^{\mathbb{H}}$, and $T^j(V)$ is weakly homogeneous for all $0\le j\le n-1$ and all $V\in \cG_n(W)$, in particular $\cG_n(W) \subset \hW^s_\bH$.
For each $n \ge 1$, let $L_n(W)$ denote the elements of $\cG_n(W)$ whose length is at least $\delta_0/3$.
Let $\cI_n(W)$ denote those elements $W_i \in \cG_n(W)$ such that $T^kW_i$ is never contained in an element
of $L_{n-k}(W)$ for all $k = 0, \ldots, n-1$.

Finally, for $\delta < \delta_0$, define $\cG_n^\delta(W)$ like $\cG_n(W)$,
but subdividing long pieces into pieces of length between $\delta/2$ and $\delta$.  Similarly, denote by $L_k^\delta(W)$ those elements of $\cG_k^\delta(W)$ having
length at least $\delta/3$, and by $\cI_n^\delta(W)$ those elements $W_i \in \cG_n^\delta(W)$
such that $T^kW_i$ has never been contained in an element of $L_{n-k}^\delta(W)$ for 
all 
$k =0, \ldots, n-1$.


\subsection{Initial Lemmas for all $t>0$}
\label{sec:t>0}

We start with two easy lemmas. (In the present paper, the parameter $\varsigma$ 
appearing in Lemmas~\ref{lem:short growth}
and \ref{lem:extra growth} will be zero or $1/p$, for $p>q+1$ chosen in \eqref{eq:all}.)

\begin{lemma}
\label{lem:short growth}
Fix $t_0\in (0,1)$. 
There exists $C_0=C_0(t_0) > 0$ such that for  all 
 $g\in C^0$, 
every $t\ge t_0$, all $t_1\in (1, t_*)$, and 
all $0 \le \varsigma <  \min \{ t, 1, \frac{2t-t_0}{2-t_0} \} $, we have
\begin{equation}
\label{eq:short growth up}
\sum_{W_i \in \cI_n(W)} \frac{|W_i|^\varsigma}{|W|^\varsigma} |J_{W_i}T^n|^t_{C^0(W_i)} 
|e^{ S_ng}|_{C^0(W_i)}
\le C_0 \theta^{n(t - \varsigma)} e^{n  |g|_{C^0}} \,  , \, \forall W \in \hW^s\, ,\, \forall  n \ge 1\, .
 \end{equation}
\end{lemma}

\begin{proof}
The case $\varsigma = 0$, $g=0$ can be proved by induction on $n$ using \eqref{eq:theta}, since elements of $\cI_n(W)$
have been short at each intermediate step.  This is the same as in \cite[Lemma~3.1]{dz1} (the exponent $t$ changes nothing),
and the\footnote{\label{care} The sets $\hW^s$ and $\cI_n(W)$
become smaller if $t_1$ is larger; while
$\theta$ increases if $t_1$ is larger, this does not affect $C_0$.} constant $C_0^{ [\varsigma = 0]}$ 
comes from
switching from the metric induced by the adapted norm $\|  \cdot \|_*$ to the standard Euclidean norm at the last step.

For   $\varsigma>0$, $g=0$, we use a H\"older inequality,
\begin{align*}
\sum_{W_i \in \cI_n(W)} \frac{|W_i|^\varsigma}{|W|^\varsigma}&  |J_{W_i}T^n|_{C^0(W_i)}^t
\le \biggl( \sum_{W_i \in \cI_n(W)} \frac{|W_i|}{|W|}  |J_{W_i}T^n|_{C^0(W_i)} \biggr)^{ \varsigma}
\biggl( \sum_{W_i \in \cI_n(W)} |J_{W_i}T^n|_{C^0(W_i)}^{\frac{t - \varsigma}{1 - \varsigma}} \biggr)^{1-\varsigma}  \, .
\end{align*}
Since $|J_{W_i}T^n|_{C^0(W_i)} \le e^{C_d} \frac{|T^nW_i|}{|W_i|}$ by \eqref{eq:distortion}, 
the first sum is bounded by $e^{C_d \varsigma}$.
Then, since $\frac{t - \varsigma}{1-\varsigma} \ge \frac{ t_0}{2}$, Definition~\ref{def:k0d0}
and Lemma~\ref{lem:one step} for $\bar t_0=t_0/2$, together with the case $\varsigma=0$,
imply the second sum is  bounded by $\bigl ( C_0^{ [\varsigma = 0]}\bigr )^{1-\varsigma} \theta^{n(t - \varsigma)}$. This completes the proof of the
lemma in the case $g=0$.
For nonzero $g$,  use $|e^{ S_ng}|_{C^0(W_i)}\le e^{n  |g|_{C^0}}$ for all $W_i$ to bootstrap from the bound for $g=0$.
\end{proof}

\begin{lemma}
\label{lem:extra growth} Fix  $t_0\in (0,1)$  and $t_1 \in (1, t_*)$.  Let  $\tilde t_1>t_0$.
There exists  $C_2 = C_2 ( t_0,  t_1, \tilde t_1)  > 0$ such that, for all  $g\in C^0$
and all $\varsigma \in [0, 1]$, we have
\begin{equation}
\label{eq:extra growth up}
\sum_{W_i \in \cG_n(W)} \frac{|W_i|^\varsigma}{|W|^\varsigma} |J_{W_i}T^n|^{t + \varsigma}_{C^0(W_i)}
|e^{ S_ng}|_{C^0(W_i)} \le C_2\,  Q_n(t, g) \, ,\,
\forall  W \in \hW^s\, ,\,    \forall n \ge 1 \,,\,\,
 \forall t\in[ t_0,  \tilde t_1] \, .
\end{equation}
\end{lemma}

\begin{proof}
The case $\varsigma = 0$ and $g=0$ is trivial since by definition each $W_i \in \cG_n(W)$ is contained in a single element of $\cM_0^{n, \mathbb{H}}$.  Since there
can be at most $2/\delta_0$ elements of $\cG_n(W)$ in one element of $\cM_0^{n, \mathbb{H}}$ 
(with $\delta_0=\delta_0(t_0,  t_1, q)$  from Definition~\ref{def:k0d0}),  the
lemma holds with $C_2[0]= 2 \delta_0^{-1}  e^{\tilde t_1C}$,  where $C$ is from Lemma~\ref{lem:comparable} (recall also footnote \eqref{care}).
Next, for $\varsigma>0$ and $g=0$, notice that by \eqref{eq:distortion},
\[
|W_i| |J_{ W_i}T^n|_{C^0(W_i)} \le e^{C_d \delta_0^{1/(q+1)}} |T^nW_i| \le e^{C_d \delta_0^{1/(q+1)}} |W| \, ,
\]
so that the sum for $\varsigma > 0$ is bounded by the sum for $\varsigma =0$ times $e^{\varsigma C_d \delta_0^{1/(q+1)}}$.
If $g \ne 0$, then again using Lemma~\ref{lem:comparable} on each $W_i$, we have 
$$|J_{W_i}T^n|^t_{C^0(W_i)} |e^{S_ng}|_{C^0(W_i)} \le e^{tC} 
\sup_{W_i \cap M'}| |J^sT^n|^t e^{S_ng} | \, ,$$  
and the required bound follows with $C_2 = C_2[0] e^{\varsigma C_d \delta_0^{1/(q+1)}}$.
\end{proof}


\subsection{Growth Lemmas for $t \in (0,1]$}
\label{sec:t<1}
In this section, we prove two growth and complexity lemmas for
$t \in (0,1]$.
The first one shows that we can make the contribution from the sum over short pieces  small compared to the sum over all pieces in
$\cG_n(W)$ by choosing a small length scale.  
Recall the constant $C_*$ from \eqref{eq:phi dist}.

\begin{lemma}
\label{lem:short small}
Let $t_0\in (0,1)$  and $t_1 \in (1,t_*)$.  For any $\ve >0$, there exist
 $\delta_1 >0$ and $n_1 \ge 1$ such that
for all $W \in \hW^s$ with $|W| \ge \delta_1/3$, all $n \ge n_1$,
and all $g \in C^1$ with\footnote{The bound on $|\nabla g|$ bounds the distortion constant of $g$ by $2$, so that 
 $\delta_1$ and $n_1$ may be chosen uniformly in $g$.}
\begin{equation}
\label{eq:up one}
2|g|_{C^0} <  - t_0 \log \theta\, ,\,  
\mbox{ i.e. }\,\,
 e^{| g|_{C^0}} \theta^{{t_0}/2} <1  \, ,
 \, \mbox{ and } \,\, | \nabla g| \le (C_* \delta_0)^{-1} \, ,
\end{equation} 
we have
\[
\sum_{\substack{W_i \in \cG_n^{\delta_1}(W) \\ |W_i| < \delta_1/3}}
|J_{W_i}T^n|^t_{C^0(W_i)} |e^{S_ng}|_{C^0(W_i)}
\le \ve \sum_{W_i \in \cG_n^{\delta_1}(W) }
|J_{W_i}T^n|^t_{C^0(W_i)} |e^{ S_ng}|_{C^0(W_i)}\, , \,\, 
\forall t \in [t_0,1] \, .
\]
In particular, taking $\ve = 1/4$ gives
$\delta_1 < \delta_0$ and $n_1 \ge 1$ 
such that for all $n \ge n_1$, for all  $g\in C^0$ satisfying \eqref{eq:up one},
for all $W \in \hW^s$  with $|W| \ge \delta_1/3$, we have
\begin{equation}
\label{eq:delta1}
\sum_{W_i \in L_n^{\delta_1}(W)} |J_{W_i}T^n|^t_{C^0(W_i)} 
|e^{S_n g}|_{C^0(W_i)}
\ge \tfrac 34 \sum_{W_i \in \cG_n^{\delta_1}(W)} |J_{W_i}T^n|^t _{C^0(W_i)}
|e^{S_n} g|_{C^0(W_i)} \, ,
\,  \forall t \in [t_0, 1] \, .
\end{equation}
\end{lemma} 

\begin{proof}
First assume $g=0$.  
Let $\ve >0$ and choose $\bar\ve > 0$ so that $6C_1^{-1} \bar\ve/(1-\bar\ve) < \ve$ (where $C_1\in (0,1]$ is from \eqref{eq:hyp}).  
Next, choose $n_1$ such that 
$C_0\theta^{tn_1}  \Lambda^{-n_1(1-t)} < \bar\ve$ (where $C_0$ is from 
Lemma~\ref{lem:short growth} for $\varsigma=0$).
Recalling again that $|T^{-1}U| \le C |U|^{1/2}$ for any $U \in \hW^s$, 
we may choose $\delta_1>0$ such that, if $|U| < \delta_1$, then
each homogeneous connected component of $T^{-n}U$ has 
length shorter than $\delta_0$ for each $n \le 2n_1$.  Then using Lemma~\ref{lem:short growth}
with $\varsigma=0$, 
if $U \in \hW^s$ with $|U| \le \delta_1$,
\begin{equation}
\label{eq:n1 short}
\sum_{W_i \in \cG_n(U)} |J_{W_i}T^n|^t_{C^0(W_i)} \le C_0 \theta^{t n}, \quad \mbox{for all $n \le 2n_1$.}
\end{equation} 

Now for $n \ge n_1$, write $n = kn_1 + \ell$, for some $0 \le \ell < n_1$.  Let $W \in \hW^s$ with $|W| \ge \delta_1/3$.  
Looking only at times $m n_1$, $m=0, \ldots, k-1$, we group elements 
$W_i \in \cG_n(W)$ with $|W_i| < \delta_1/3$
according to the largest $m$ such that $T^{(k-m)n_1+ \ell}W_i \subset V_j \in L^{\delta_1}_{mn_1}(W)$.  This is similar
to using\footnote{The most recent long ancestor for $W_i \in \cG_n(W)$ corresponds to the maximal $m \le n$ such that $T^{n-m}W_i \subset V_j$ and $V_j \in L_{m}(W)$, not to be confused with the first long ancestor, see \eqref{FLA}.} the {\em most recent long ancestor,}
 except that we only look at times that are multiples of $n_1$.
We denote by $\bar I^\delta_{(k-m)n_1 + \ell}(V_j)$ the set of $W_i \in \cG_n(W)$ identified with $V_j \in L_{mk}^{\delta_1}(W)$
in this way.  Since $|W| \ge \delta_1/3$, every element of $\cG_n^{\delta_1}(W)$ must have a long ancestor.

Note that since $T^{(k-m')n_1+\ell}W_i$ is contained in an element of $\cG_{m'k}(W)$ that is shorter than  $\delta_1/3$ for
$m' < m$,
we may apply \eqref{eq:n1 short} inductively $k-m$ times.  Thus,
\begin{align*}
\sum_{\substack{W_i \in \cG_n^{\delta_1}(W) \\ |W_i| < \delta_1/3}}  |J_{W_i}T^n|^t_{C^0(W_i)}
& \le \sum_{m=0}^{k-1} \sum_{V_j \in L^{\delta_1}_{mn_1}(W)} |J_{V_j}T^{mn_1}|^t_{C^0(V_j)}
\sum_{W_i \in \bar I^{\delta_1}_{(k-m)n_1 + \ell}(V_j)} |J_{W_i}T^{(k-m)n_1+\ell}|^t_{C^0(W_i)} \\
& \le \sum_{m=0}^{k-1} \sum_{V_j \in L^{\delta_1}_{mn_1}(W)} |J_{V_j}T^{mn_1}|^t_{C^0(V_j)}  C_0 \theta^{tn_1(k-m)} \, .
\end{align*}

Next, notice that for $t \in (0,1]$, $V \in \hW^s$ and each $k \ge 1$,
using $|V|=\sum_{W_i \in \cG_k^{\delta_1}(V)} |T^k(W_i)|$, 
\begin{align}
\label{eq:grow}
\sum_{W_i \in \cG_k^{\delta_1}(V)} |J_{W_i}T^k|^t_{C^0(W_i)} 
& =\sum_{W_i \in \cG_k^{\delta_1}(V)} |J_{W_i}T^k|_{C^0(W_i)} |J_{W_i}T^k|_{C^0(W_i)}^{t-1} \\
\nonumber & \ge C_1 \Lambda^{k(1-t)} \sum_{W_i \in \cG_k^{\delta_1}(V)} \frac{|T^kW_i|}{|W_i|} \ge 
C_1 \Lambda^{k(1-t)} |V| \delta_1^{-1} \, .
\end{align}
Also note that by the proof of Lemma~\ref{lem:distortion}, we have
\begin{equation}
\label{eq:ratio}
\frac{\sup_{x \in W_i} |J_{W_i}T^n(x)|^t}{\inf_{y \in W_i} |J_{W_i}T^n(y)|^t } \le 1+C_d \delta_0^{1/(q+1)}
\le 2 \, ,
\end{equation}
since $t \le 1$.
Putting these estimates together, we obtain,
\begin{align}
\nonumber&
\frac{\displaystyle \sum_{\substack{W_i \in \cG_n^{\delta_1}(W) \\ |W_i| < \delta_1/3}}  |J_{W_i}T^n|^t_{C^0(W_i)} }{\displaystyle \sum_{W_i \in \cG_n^{\delta_1}(W)} |J_{W_i}T^n|^t_{C^0(W_i)}}\\
\nonumber &\qquad\qquad \le \sum_{m=0}^{k-1} \frac{ \displaystyle 2 \sum_{ V_j \in L^{\delta_1}_{mn_1}(W)} |J_{V_j}T^{mn_1}|^t_{C^0(V_j)}  C_0 \theta^{tn_1(k-m)}  }
{ \displaystyle \sum_{  V_j \in L^{\delta_1}_{mn_1}(W)} |J_{V_j}T^{mn_1}|^t_{C^0(V_j)}
\sum_{ W_i \in \cG_{(k-m)n_1+\ell}^{\delta_1}(V_j)} |J_{W_i}T^{(k-m)n_1+\ell}|^t_{C^0(W_i)}} \\
\nonumber &\qquad\qquad \le \sum_{m=0}^{k-1} \frac{ \displaystyle 2 \sum_{ V_j \in L^{\delta_1}_{mn_1}(W)} |J_{V_j}T^{mn_1}|^t_{C^0(V_j)}  C_0 \theta^{tn_1(k-m)}} 
{ \displaystyle \sum_{  V_j \in L^{\delta_1}_{mn_1}(W)} |J_{V_j}T^{mn_1}|^t_{C^0(V_j)} C_1 \Lambda^{(k-m)n_1(1-t)} |V_j| \delta_1^{-1} } \\
\label{eq:short bound}&\qquad\qquad \le 6C_1^{-1} \sum_{m=0}^{k-1} \bar\ve^{k-m} \le 6 C_1^{-1} \frac{\bar\ve}{1-\bar\ve} \, , 
\end{align}
where in the second inequality we used \eqref{eq:grow} on each $V_j \in L_{mn_1}^{\delta_1}(W)$. 
This ends the case $g=0$.

If $g \ne 0$, note that by \eqref{eq:phi dist} and using \eqref{eq:up one},
\[
\frac{\sup_{x \in W_i} e^{S_kg(x)} }{\inf_{y \in W_i} e^{S_kg(y)} } \le 1 + C_* |\nabla g| d(x,y) \le 2 \, ,
\]
for all $W_i \in \cG_k^{\delta_0}(W)$ and any $k >0$.
Letting $\ve>0$, choose  $\bar\ve > 0$ such that $12 C_1^{-1} \bar \ve / (1- \bar \ve) < \ve$. 
Then we take
$n_1$ such that $C_0 \theta^{tn_1} e^{2  n_1 | g|_{C^0}} \Lambda^{-n_1(1-t)} < \bar\ve$, and we choose $\delta_1>0$ such that, 
if $U \in \hW^s$ satisfies $|U|<\delta_1$, then
\begin{equation}
\label{eq:n1 short up}
\sum_{W_i \in \cG_n(U)} |J_{W_i}T^n|^t_{C_0(W_i)} |e^{ S_n g}|_{C^0(W_i)}
\le C_0 \theta^{tn} e^{n | g|_{C^0}} \, , \qquad \mbox{for all $n \le 2n_1$,}
\end{equation}
which is the analogue of \eqref{eq:n1 short}.
(For fixed $t_0$, note that $n_1$ and $\delta_1$ depend only on $\ve$,
uniformly in   $g$ satisfying \eqref{eq:up one}.)  The
proof above  can then be followed line by line,  inserting 
$e^{ S_ng}$.  Thus \eqref{eq:grow} becomes,
\begin{equation}
\label{eq:grow up}
\sum_{W_i \in \cG_k^{\delta_1}(V)} |J_{W_i}T^k|^t_{C_0(W_i)} |e^{ S_kg}|_{C^0(W_i)}
\ge C_1 \Lambda^{k(1-t)} e^{-  k |g|_{C^0}} |V| \delta_1^{-1} \, .
\end{equation}
Inserting this lower bound in \eqref{eq:short bound} and applying Lemma~\ref{lem:short growth}
with $\varsigma = 0$ yields,
\begin{align}
\label{eq:short bound up}
& \frac{\displaystyle \sum_{\substack{W_i \in \cG_n^{\delta_1}(W) \\ |W_i| < \delta_1/3}}  |J_{W_i}T^n|^t_{C^0(W_i)} |e^{ S_n g}|_{C^0(W_i)} }{\displaystyle \sum_{W_i \in \cG_n^{\delta_1}(W)} |J_{W_i}T^n|^t_{C^0(W_i)} |e^{ S_ng }|_{C^0(W_i)} }\\
\nonumber  &\qquad\qquad \le \sum_{m=0}^{k-1} \frac{ \displaystyle 4 \sum_{ V_j \in L^{\delta_1}_{mn_1}(W)} |J_{V_j}T^{mn_1}|^t_{C^0(V_j)}   |e^{S_{m n_1} g} |_{C^0(V_j)}   C_0 \theta^{tn_1(k-m)} e^{ n_1(k-m) |g|_{C^0}}  } 
{ \displaystyle \sum_{  V_j \in L^{\delta_1}_{mn_1}(W)} |J_{V_j}T^{mn_1}|^t_{C^0(V_j)}  |e^{S_{m n_1} g} |_{C^0(V_j)}   C_1 \Lambda^{(k-m)n_1(1-t)} e^{-  n_1(k-m) |g|_{C^0}}    |V_j| {\delta_1}^{-1} } \\
\nonumber &\qquad\qquad \le 12 C_1^{-1} \sum_{m=0}^{k-1} \bar\ve^{k-m} \le 12 C_1^{-1} \frac{\bar\ve}{1-\bar\ve} \, < \, \ve \, ,
\end{align}
where the extra factor of $2$ comes from the distortion constant of $g$.
\end{proof}

The second lemma
proves the analogue of Lemma~\ref{lem:short small} for elements of $\cM_0^{n, \bH}$, in anticipation of Proposition~\ref{prop:lower bound}.
For $A \in \cM_0^{n, \mathbb{H}}$, let $B_{n-1}(A)$ denote the element of 
$\cM_{-n+1}^{0, \mathbb{H}} \bigvee {\cH}$ containing 
$T^{n-1}A \in \cM_{-n+1}^{1, \mathbb{H}}$, recalling that ${\cH}$ is the partition of $M$ 
into homogeneity strips $\bH_k$.
We introduce this additional intersection with ${\cH}$ (omitted from the definition of
$\cM_{-n+1}^{0, \bH}$) since it will be convenient to work with homogeneous partition
elements
in what follows. For $\delta > 0$, define
\begin{equation}
\label{eq:An delta}
\cA_n(\delta) = \{ A \in \cM_0^{n, \mathbb{H}} : \diam^u(B_{n-1}(A)) \ge \delta/3 \} \, . 
\end{equation}
The following result shows that most of the weights contributing to $Q_n(t,g)$ come from 
elements of  $\cA_n(\delta)$ if $\delta$ is chosen small enough.

\begin{lemma}
\label{lem:most long}
Let $t_0 >0$. For any $\upsilon \ge 0$,
there exist $\delta_2 >0$ and $c_0 = c_0(\upsilon)>0$ with
$c_0(\upsilon') \ge c_0(\upsilon)$ if $\upsilon'\in [0, \upsilon],$ such that 
for any $g \in C^1$ satisfying \eqref{eq:up one}
with $ |\nabla g|_{C^0}\le \upsilon$, 
\[
\sum_{A \in \cA_n(\delta_2) }\sup_{x \in A \cap M'} |J^sT^n(x)|^t e^{ S_ng(x)}
\ge c_0 \, Q_n(t, g) \,,\,\, \forall n \in \mathbb{N}\, ,
\,\,  \forall t\in [t_0,1]  \, .
\]
\end{lemma}

\begin{proof}
Assume first $g=0$. We begin by relating $J^sT^n$ on $A \in \cM_0^{n,\mathbb{H}}$ with 
$J^uT^{-n+1}$ on $T^{n-1}A$.   Recalling the definition of $E(y)$ in \eqref{eq:jacs}, if $x \in A$ and $y = T^nx$, we have 
\[
J^sT^n(x) = \frac{ (E\cos \vf) \circ T^{-n}(y)}{(E\cos \vf)(y)} J^uT^{-n}(y) \, . 
\]
Here, $J^uT^{-n} = \det (DT^{-n} |_{E^u})$, where $E^u$ is the unstable direction for $T$ (not
$T^{-1}$), so that $J^uT^{-n}$ is a contraction.  Next, since\footnote{We use the notation
$A = C^{\pm 1} B$ to denote $C^{-1} B \le A \le C B$.} 
$J^uT^{-1}(y) = C^{\pm 1} \cos \vf(y)$, and the function $E$ is uniformly bounded away from 0, 
we conclude,
\begin{equation}
\label{eq:relate}
J^sT^n(x) = C^{\pm 1} \cos \vf(T^{-n}y) J^uT^{-n+1}( T^{-1}y ) \, .
\end{equation}

For brevity, for any set $A \subset M$, we will denote
\begin{equation}\label{forbrevity}
|J^sT^n|^t_A :=  \sup_{x \in A \cap M'} |J^sT^n(x)|^t \, \, \mbox{ and similarly, } \, \,
|J^uT^{-n}|^t_A := \sup_{x \in A \cap M'} |J^uT^{-n}(x)|^t \, .
\end{equation}

Next, we consider the evolution of elements of $\cM_{-k}^{0, \bH}$ under iteration by $T^j$ for
$j \ge 1$.
If $B \in \cM_{-k}^{0, \bH}$, then we subdivide $T^jB$ according to singularity curves and homogeneity
strips at each iterate, much as we would consider the evolution of an unstable curve $U$ under
$T^j$.  We write $T^jB = \cup_{B' \in G_j(B)} B'$, where $G_j(B)$ is the maximal decomposition
of $T^jB$ into elements of $\cM_{-k-j}^{0, \bH} \bigvee \cH$,  recalling that $\cH$ denotes
the partition of $M$ according to homogeneity strips.  
This last intersection with ${\cH}$ is necessary since
we will work with homogeneous elements $B' \subset T^jB$ (to maintain bounded distortion
for $J^uT^{-j}$ on $B'$).
Let  $L_j^{\delta}(B)$ denote those elements $B' \in G_j(B)$ with $\diam^u(B') \ge  \delta/3$.

 Now by Definition~\ref{def:k0d0} and applying the time reversal of the proof of Lemma~\ref{lem:short growth} (with $\varsigma = 0$), there exists $\delta_2 >0$ such that if
$\max \{ \diam^u(B), \diam^s(B) \} \le \delta_2$, then 
\begin{equation}
\label{eq:d1 comp}
\sum_{B' \in G_j(B)} |J^uT^{-j}|^t_{B'} \le C_0 \theta^{tj}, \qquad \mbox{for all $j \le n_1$},
\end{equation}
where $n_1=n_1(1/4)$ is from \eqref{eq:delta1}
in Lemma~\ref{lem:short small}.  For convenience, we choose 
$\delta_2 \le \delta_1(1/4)$.
Also, 
if $B \in \cM_{-k}^{0,\bH}$, then $\diam^s(B) \le C\Lambda^{-k}$ for some
uniform constant $C>0$.  We choose $n_2 \ge n_1$
so that $\diam^s(B) \le  \delta_2$ if $B \in \cM_{-k}^{0, \bH}$ for $k \ge n_2$. 

We fix $n \ge n_2+1$ and prove the lemma for such $n$.
For $B \in \cM_{-n+1}^{0, \bH} \bigvee {\cH}$, let $B_{-j}$ denote the element
of $\cM^{0, \bH}_{-n+1+j} \bigvee {\cH}$ containing $T^{-j}B$.  We call $B_{-j}$ the most recent
$u$-long ancestor of $B$ if $j$ is the minimal integer $k \le n-n_2$ such that 
$\diam^u(B_{-k}) \ge  \delta_2$.  If no such $j$ exists, we say that $B$ has been $u$-short
since time $n_2$.  (It follows from the definition of $n_2$, that 
$\diam^s(B_{-j}) \le  \delta_2$ for all $j \le n-n_2-1$.)
Let $\mathbb{L}_{-n+1+j}^{\delta_2}$ denote those elements of $\cM^{0, \bH}_{-n+1+j} \bigvee {\cH}$
which are $u$-long, and let $\mathbb{S}_{-n+1+j}^{\delta_2}$ denote those elements which
are $u$-short (in the length scale  $\delta_2$).
Similarly, let
$\mathbb{I}_j^{\delta_2}(B_{-j})$ denote the collection of 
$B \in \cM^{0, \bH}_{-n+1} \bigvee {\cH}$ whose most recent $u$-long ancestor is $B_{-j}$.
Note that $\mathbb{I}_j^{\delta_2}(B_{-j}) \subset G_j(B_{-j})$.

Thus if $k \ge n_2$ and $B' \in \cM_{-k}^{0, \bH}$,
then  estimating inductively as in
the proof of Lemma~\ref{lem:short growth}, 
\begin{equation}
\label{eq:short elements}
\sum_{B \in \mathbb{I}_j^{\delta_2}(B')} |J^uT^{-j}|^t_{B} \le C_0 \delta_2^{-1} \theta^{t j} \, \mbox{for all $j \ge 0$,}
\end{equation}
where the factor $\delta_2^{-1}$ is due to the fact that $B'$ itself may be $u$-long, in which case
it would be artificially subdivided into $\sim \delta_2^{-1}$ pieces of $u$-diameter less than
$\delta_2$ before being iterated.

Let $\cA_n^c(\delta_2) = \cM_0^{n,\bH} \setminus \cA_n({\delta_2})$.  By \eqref{eq:relate},
\[
\sum_{A \in \cA_n^c(\delta_2)} |J^sT^n|^t_A = C^{\pm 1} \sum_{A \in \cA_n^c(\delta_2)}
|\cos \vf|^t_A |J^uT^{-n+1}|^t_{T^{n-1}A} \, .
\]
Note that if $B \in \cM_{-n+1}^{0, \bH} \bigvee {\cH}$ with $\diam^u(B) < \delta_2/3$, then
any $A \in \cM_0^{n,\bH}$ for which $B = B_{n-1}(A)$ belongs to $\cA_n^c(\delta_2)$.
Also, 
\[
T^{n-1}A \in \cM_{-n+1}^{1, \bH} = \cM_{-n+1}^{0, \bH} \bigvee {\cH} \bigvee \cM_0^1
\]
so that for fixed $B$, the number of $A$ such that $B_{n-1}(A) = B$ is at most
$\# \cM_0^1$.  Moreover, since all such $A$ are by definition contained in
$T^{-n+1}(B_{n-1}(A)) = T^{-n+1}B$, and 
$T^{-n+1}B \in \cM_0^{n-1, \bH} \bigvee T^{-n+1} {\cH}$,
all $A$ corresponding to one $B$ are contained in the same homogeneity strip, so that
$|\cos \vf|_A$ is comparable on all such $A$.
In addition, since $\cM_{-n+1}^{0, \bH} \bigvee \cH$ is the partition into
connected components of $M \setminus \cup_{i=0}^n T^i \cS_0^{\bH}$,
it follows that $T^{-k}$ is smooth on $B$, and $T^{-k}B$ belongs to a single homogeneity strip
for each $0 \le k \le n-1$.  Then, applying the time reversal of Lemma~\ref{lem:comparable}
to unstable curves, we conclude that $|J^uT^{-n+1}|_{T^{n-1}A}$ is comparable
to $|J^uT^{-n+1}|_B$ for each $A$ such that $B_{n-1}(A)=B$. 
Thus,
\begin{equation}
\label{eq:switch B short}
\sum_{A \in \cA_n^c(\delta_2)} |J^sT^n|^t_A = C^{\pm 1} \sum_{B \in \bS_{-n+1}^{\delta_2}}
|\cos \vf|^t_{T^{-n+1}B} |J^uT^{-n+1}|^t_{B}  \, ,
\end{equation}
and similarly,
\begin{equation}
\label{eq:switch B long}
\sum_{A \in \cA_n(\delta_2)} |J^sT^n|^t_A = C^{\pm 1} \sum_{B \in \bL_{-n+1}^{\delta_2}}
|\cos \vf|^t_{T^{-n+1}B} |J^uT^{-n+1}|^t_{B}  \, .
\end{equation}

Next, we group elements of $\bS_{-n+1}^{\delta_2}$ by most recent $u$-long ancestor
in $\cM_{-n+1+j}^{0, \bH}$, as 
described above.  By \eqref{eq:d1 comp}, there is no need to consider long ancestors
for $j < n_1$.
Note that if $B' \in \cM_{-n+1+j}^{0, \bH} \bigvee {\cH}$ and $B \in \bI_j^{\delta_2}(B')$,
then $T^{-n+1}B$ lies in the same homogeneity strip as $T^{-n+1+j}B'$, so that $\cos \vf$ is comparable
on each of these sets.
Thus, by   \eqref{eq:d1 comp}--\eqref{eq:short elements},
\begin{align}
\label{eq:group}
\sum_{B \in \mathbb{S}_{-n+1}^{\delta_2}}&|\cos \vf|^t_{T^{-n+1}B} |J^uT^{-n+1}|^t_{B} \\
\nonumber& = \sum_{j=n_1}^{n-n_2-1} \sum_{B' \in \bL_{-n+1+j}^{\delta_2}}
\sum_{B \in \bI_j^{\delta_2}(B')} |\cos \vf|^t_{T^{-n+1}B} |J^uT^{-j}|^t_B |J^uT^{-n+1+j}|^t_{B'} \\
\nonumber& \qquad + \sum_{B' \in \bS_{-n_2}^{\delta_2}} \sum_{B \in \bI_{n-n_2}^{\delta_2}(B')}
|\cos \vf|^t_{T^{-n+1}B} |J^uT^{-j}|^t_B |J^uT^{-n+1+j}|^t_{B'} \\
\nonumber& \le \sum_{j=n_1}^{n-n_2-1} \sum_{B' \in \bL_{-n+1+j}^{\delta_2}}
 C  \delta_2^{-1} \theta^{tj} |\cos \vf|^t_{T^{-n+1+j}B'} |J^uT^{-n+1+j}|^t_{B'} \\
\nonumber & \qquad + \sum_{B' \in \bS_{-n_2}^{\delta_2}} 
 C \theta^{t (n-n_2-1)} |\cos \vf|^t_{T^{-n_2}B'} |J^uT^{-n_2}|^t_{B'} \, ,
\end{align}
where the final sum over $B' \in \bS_{-n_2}^{\delta_2}$ represents those $B \in \bS_{-n+1}^{\delta_2}$
which have had no $u$-long ancestor since before time $n_2$.

To proceed, we will need the following sublemma, linking the contribution from 
$\bL_{-n+1+j}^{\delta_2}$ to the contribution from $\bL_{-n+1}^{\delta_2}$.

\begin{sublem}
\label{sub:long}
Let $t_0\in (0,1)$.
There exists $C>0$ such that for all
$t\in [t_0,1]$, each $n_1 \le j \le n-n_2-1$,  and all $B' \in \bL_{-n+1+j}^{\delta_2}$,
\[
|\cos \vf|^t_{T^{-n+1+j}B'} |J^uT^{-n+1+j}|^t_{B'}
\le C \delta_2^{-1} \Lambda^{j(t-1)} \sum_{B \in L_j(B')} |\cos \vf|^t_{T^{-n+1}B} |J^uT^{-n+1}|^t_B \, ,
\]
where $L_j(B')$ denotes the collection of elements $B \in G_j(B')$ with $\diam^u(B) \ge \delta_2/3$.
\end{sublem}

\begin{proof}
Since $B' \in \bL_{-n+1+j}^{\delta_2}$, there exists an unstable curve $U \subset B'$ with
$|U| \ge \delta_2/3$.  Let $G_j'(B')$ denote those elements $B \in G_j(B')$ such that
$T^jU \cap B \neq \emptyset$.  Letting $\cG_j^{\delta_2}(U)$ denote the $j$th generation of homogeneous
elements of $T^jU$, using the time reversed definition of $\cG_j^{\delta_2}(W)$ for stable curves from
Section~\ref{1step}.
If $U_i \in \cG_j^{\delta_2}(U)$ has $|U| \ge \delta_2/3$, and $B \cap U \neq \emptyset$ for
some $B \in G'_j(B')$, then necessarily, $\diam^u(B) \ge \delta_2/3$.  Let $L'_j(B') \subset L_j(B')$
denote this collection of long elements.  Then letting $L_j^{\delta_2}(U) \subset \cG_j^{\delta_2}(U)$ denote those
elements of $\cG_j^{\delta_2}(U)$ with length at least $\delta_2/3$, we estimate,
\begin{align}
\label{sub:est} \sum_{B \in L_j(B')} &|\cos \vf|^t_{T^{-n+1}B}  |J^uT^{-n+1}|^t_B\\
\nonumber &\ge C |\cos \vf|^t_{T^{-n+1+j}B'} |J^uT^{-n+1+j}|^t_{B'} \sum_{B \in L'_j(B')} |J^uT^{-j}|^t_B \\
\nonumber & \ge C' |\cos \vf|^t_{T^{-n+1+j}B'} |J^uT^{-n+1+j}|^t_{B'} \delta_2 \sum_{U_i \in L^{\delta_2}_j(U)} 
|J_{U_i}T^{-j}|^t_{C^0(U_i)} \\
\nonumber & \ge C' \delta_2 |\cos \vf|^t_{T^{-n+1+j}B'} |J^uT^{-n+1+j}|^t_{B'} 
\tfrac 34 \sum_{U_i \in \cG^{\delta_2}_j(U)} 
|J_{U_i}T^{-j}|^t_{C^0(U_i)} \\
\nonumber & \ge C'' \delta_2 |\cos \vf|^t_{T^{-n+1+j}B'} |J^uT^{-n+1+j}|^t_{B'} 
\Lambda^{j(1-t )} \sum_{U_i \in \cG^{\delta_2}_j(U)} 
|J_{U_i}T^{-j}|_{C^0(U_i)} \\
\nonumber& \ge C'' \delta_2 |\cos \vf|^t_{T^{-n+1+j}B'} |J^uT^{-n+1+j}|^t_{B'} 
\Lambda^{j(1-t )} \frac{|U|}{\delta_2} \, ,
\end{align}
where in the first inequality we have used the fact that $\cos \vf$ is comparable on $T^{-n+1}B$ and
$T^{-n+1+j}B'$, in the second inequality we have applied the time reversal of 
Lemma~\ref{lem:comparable}
and the factor $\delta_2$ appears since there may be up to
$\sim \delta_2^{-1}$ elements of $L^{\delta_2}_j(U)$ in each element $B \in L'_j(B')$ (due
to artificial subdivisions in the definition of $\cG_j^{\delta_2}(U)$), and in the third inequality
we have applied 
the time reversal of Lemma~\ref{lem:short small} and \eqref{eq:delta1} 
from Lemma~\ref{lem:short small}
since $\delta_2 \le \delta_1$ 
and $j \ge n_1$.  Since $|U| \ge \delta_2/3$, this completes the proof of the sublemma.
\end{proof}

Using the sublemma, we now estimate the right hand side of \eqref{eq:group}, summing
over $B' \in \bL_{-n+1+j}^{\delta_2}$ and noting that if $B \in L_j(B')$, then
$B \in \bL_{-n+1}^{\delta_2}$ and each such $B$ is associated with a unique $B'$:
\begin{align*}
\sum_{B \in \mathbb{S}_{-n+1}^{\delta_2}}
|\cos \vf|^t_{T^{-n+1}B}& |J^uT^{-n+1}|^t_{B}
 \; \le \sum_{B' \in \bS_{-n_2}^{\delta_2}} 
 C \theta^{t (n-n_2-1)} |\cos \vf|^t_{T^{-n_2}B'} |J^uT^{-n_2}|^t_{B'} 
\\
& \qquad + \sum_{j=n_1}^{n-n_2-1}  C \delta_2^{-2} \theta^{tj} \Lambda^{j(t-1)}
\sum_{B \in \bL_{-n+1}^{\delta_2}}
 |\cos \vf|^t_{T^{-n+1}B} |J^uT^{-n+1}|^t_{B} \\
& \le C_{n_2}(t_0) \theta^{tn}+ C \delta_2^{-2} \sum_{B \in \bL_{-n+1}^{\delta_2}}
 |\cos \vf|^t_{T^{-n+1}B} |J^uT^{-n+1}|^t_{B} \, ,
 \end{align*}
for some constant  $C_{n_2}(t_0) >0$ depending only on $n_2$ and $t_0$.  
To see the estimate on the first term, note that by \eqref{eq:relate}, since elements of $\bS_{-n_2}^{\delta_2}$
are connected components of $M \setminus \cup_{i=0}^{n_2+1} T^i \cS_0^{\bH}$, we have
\[
\sum_{B' \in \bS_{-n_2}^{\delta_2}} |\cos \vf|^t_{T^{-n_2}B'} |J^uT^{-n_2}|^t_{B'}
\le C \sum_{A \in \cM_0^{n_2+1, \bH}} |J^sT^{n_2+1}|^t_A \le C Q_{n_2+1}(t) 
\le C Q_{n_2+1}(t_0) \, .
\]

Next, note that the sum over $\bL_{-n+1}^{\delta_2}$ grows at a rate of at least $C \Lambda^{n(1-t)}$
by the proof of the sublemma.  Thus we may choose $n_3 \ge n_2$ large enough that 
$C_{n_2}(t_0) \theta^{tn} \le C\Lambda^{n(1-t)}$ for all $n \ge n_3$, which implies,
\[
\sum_{B \in \mathbb{S}_{-n+1}^{\delta_2}}
|\cos \vf|^t_{T^{-n+1}B} |J^uT^{-n+1}|^t_{B}
\le
C \delta_2^{-2} \sum_{B \in \bL_{-n+1}^{\delta_2}}
 |\cos \vf|^t_{T^{-n+1}B} |J^uT^{-n+1}|^t_{B} \, .
 \]
Using this estimate  with
\eqref{eq:switch B short}, \eqref{eq:switch B long} and the fact that
$\cA_n(\delta_2) \cup \cA_n^c(\delta_2) = \cM_0^{n, \bH}$ yields,
\begin{align*}
Q_n(t) & = \sum_{A \in \cA_n^c(\delta_2)} |J^sT^n|^t_A 
+ \sum_{A \in \cA_n(\delta_2)} |J^sT^n|^t_A \\
& \le C (\delta_2^{-2} + 1) \sum_{B \in \bL_{-n+1}^{\delta_2}}
 |\cos \vf|^t_{T^{-n+1}B} |J^uT^{-n+1}|^t_{B}  \le C (\delta_2^{-2}+1)  \sum_{A \in \cA_n(\delta_2)} |J^sT^n|^t_A \, ,
\end{align*}
completing the proof of the lemma for $n \ge n_3$
and $g=0$.  The statement for general $n$ (and $g=0$) follows,
possibly reducing the constant $c_0$,
since there are only finitely many $n$ to correct for.

If $g\ne0$, the proof remains as is until \eqref{eq:d1 comp} with the same choices
of $n_2$ and $\delta_2$ (these choices are independent of $g$), so that
\eqref{eq:d1 comp} holds with
$|e^{S_j g}|_{T^{-j}B'}$ inserted in the left-hand side and $e^{j |g|_{C^0}}$ in the right.
The analogous modification is made to \eqref{eq:short elements}.
Then  \eqref{eq:switch B short} becomes
\begin{equation}
\label{eq:switch B short up}
\sum_{A \in \cA_n^c(\delta_2)} |J^sT^n|^t_A |e^{S_n g}|_A
= C^{\pm 1} \sum_{B \in \bS_{-n+1}^{\delta_2}} |\cos \vf|^t_{T^{-n+1}B}
|J^uT^{-n+1}|^t_B |e^{ S_{n-1}g}|_{T^{-n+1}B} \, ,
\end{equation}
where $C$ depends on $|\nabla g|_{C^0}$ via  \eqref{eq:phi A},
with the analogous modification to \eqref{eq:switch B long}. Then \eqref{eq:group} is
modified in the obvious way for $n \ge n_2+1$,
\begin{align}
\label{eq:group up} 
& \sum_{B \in \mathbb{S}_{-n+1}^{\delta_2}} |\cos \vf|^t_{T^{-n+1}B} |J^uT^{-n+1}|^t_{B}
|e^{ S_{n-1}g}|_{T^{-n+1}B} \\
\nonumber & \le \sum_{j=n_1}^{n-n_2-1} \sum_{B' \in \bL_{-n+1+j}^{\delta_2}}
 C \delta_1^{-1} \theta^{tj} e^{ j |g|_{C^0}} |\cos \vf|^t_{T^{-n+1+j}B'} 
 |J^uT^{-n+1+j}|^t_{B'} 
 |e^{ S_{n-1-j} g}|_{T^{-n+1+j}B'} \\
\nonumber&  + \sum_{B' \in \bS_{-n_2}^{\delta_2}} 
 C \theta^{t (n-n_2 - 1)} e^{ (n-n_2 - 1) |g|_{C^0}}  |\cos \vf|^t_{T^{-n_2}B'} |J^uT^{-n_2}|^t_{B'}
 |e^{ S_{n_2} g}|_{T^{-n_2} B'} \, . 
\end{align} 
A suitable analogue of Sublemma~\ref{sub:long} yields  $C>0$ such that for each
$n_1 \le j \le n - n_2-1$ and $B' \in \mathbb{L}^{\delta_2}_{-n+1+j}$,
\begin{align*}
|\cos \vf|^t_{T^{-n+1+j}B'} & |J^uT^{-n+1+j}|^t_{B'} |e^{ S_{n-1-j}g}|_{T^{-n+1+j}B'} \\
&\le C \delta_2^{-1} \Lambda^{j(t-1)} e^{ j |g|_{C^0}} 
\sum_{B \in L_j(B')} 
|\cos \vf|^t_{T^{-n+1}B} |J^uT^{-n+1}|^t_{B} |e^{ S_{n-1}g}|_{T^{-n+1}B} \, ,
\end{align*}
where we have used the lower bound \eqref{eq:grow up} rather than \eqref{eq:grow}
in \eqref{sub:est}.  This provides the contraction required to complete the proof of 
Lemma~\ref{lem:most long} since $\theta^t e^{| g|_{C^0}} < 1$ by \eqref{eq:up one}.
\end{proof}

 
\subsection{Defining $s_1>1$. Growth Lemmas for $t\in (1, s_1)$}
\label{sec:t>1}

In this section, we bootstrap from our results for $t \le 1$ to conclude a parallel set 
of results for  $t \in (1, s_1)$, for
 $s_1 >1$ from Definition~\ref{def:t_*} below.  
 To do this, we will apply Propositions~\ref{prop:lower bound}
and \ref{prop:exact} from Section~\ref{sec:complexity}
for $t \le 1$ whose proofs rely only on the lemmas in Section~\ref{sec:t<1}.
In Section~\ref{sec:boot}, we show how to extend this  to all $t < t_*$.

The easy lemma below will be crucial  to define $s_1$:

\begin{lemma}
\label{lem:t_*}
We have $P_*(1)=0$. Moreover,
the limit $\displaystyle  \chi_1:= \lim_{s \to 1^-} \frac{P_*(s)}{1-s}$ exists and $\chi_1 \ge \log \Lambda >0$.
\end{lemma}

In fact,  $\chi_1 =\int_M \log J^uT \, d\musrb$, which follows from Theorems~\ref{strconv} and
\ref{thm:geo var}. 

\begin{proof}
Proposition~\ref{prop:lower bound} for $t=1$ together with 
\cite[Lemma~3.2]{dz1} prove that $Q_n(1)$ is uniformly bounded
 for all $n$, so that
$P_*(1) \le 0$. Since Proposition~\ref{prop:pressure} gives
$P_*(1) \ge P(1) = 0$,
we have established that $P_*(1)=0$.
Next, the convexity of $P_*(t)$ (Proposition~\ref{prop:Ptilde}) on $(0, \infty)$
implies that left (and right) derivatives exist at every $t>0$.
Thus, since $P_*(1) = 0$, the limit below  exists
\begin{equation}\label{leftder}
 \lim_{s \to 1^-} \frac{P_*(s)}{1-s} = \lim_{s \to 1^-} \frac{P_*(s) - P_*(1)}{1-s} \, .
\end{equation}
The proof that $P(t)$ is strictly decreasing in Proposition~\ref{prop:Ptilde} implies
$\chi_1 \ge \log \Lambda > 0$.
\end{proof}

\medskip

\begin{defin}\label{def:t_*}
Recalling  $\theta(t_1) \in (\Lambda^{-1}, \Lambda^{-1/2})$ 
from Definition~\ref{def:k0d0}, we  define
$\displaystyle  s_1 := \frac{\chi_1}{\chi_1 + \log \theta}>1$.
\end{defin}

Note that  $s_1$ is just the intersection point between the tangent line to $P_*(t)$ at $t=1$
(which is the largest $t$ where we have established
the lower bound \eqref{eq:grow} on the sum over $\cG_n(W)$) and the line $y = t \log \theta$.  If $t<s_1$
then $\theta^t < e^{P_*(t)}$, which
can be viewed as a pressure gap condition. 
Note finally that establishing Theorem~\ref{thm:geo var}. 
in a neighbourhood of $t=1$ will give $s_1\le t_*$.
\medskip

A key to many
results for $0<t\le 1$ is the lower bound on the rate of growth given by
\eqref{eq:grow} in the proof of Lemma~\ref{lem:short small}.   Our
next lemma obtains this lower bound
for $t\ge 1$, interpolating via a H\"older inequality. 

\begin{lemma}
\label{lem:low t>1}
Let $t_0\in (0,1)$  and $t_1 \in (1, t_*)$. Let  $\bar t_1 \ge 1$.
For any $\kappa>0$, there exist $C_\kappa>0$, $\eta_\kappa > 0$ such that for
all $g \in C^0$ and $\delta>0$, and all $W \in \hW^s$  with  $|W| \ge \delta/3$,
\begin{equation}
\label{eq:low t>1 up}
\sum_{W_i \in \cG_n^\delta(W)} |J_{W_i}T^n|^t_{C^0(W_i)} |e^{ S_ng}|_{C^0(W_i)}
\ge C_\kappa  \delta^{\frac{1}{\eta_\kappa}-1}  e^{-n(\chi_1 + \kappa)(t-1) - n  |g|_{C^0}} \, ,\,
\, \forall n \ge 1 \,,\,\,\forall t \in [1, \bar t_1] \, .
\end{equation}
Moreover, if $|W| \ge \delta_0/3$, then 
\[
\sum_{W_i \in \cG_n^\delta(W)} |J_{W_i}T^n|^t_{C^0(W_i)} |e^{ S_ng}|_{C^0(W_i)}
\ge C_\kappa  \delta^{- 1}  e^{-n(\chi_1 + \kappa)(t-1) - n |g|_{C^0}} \, ,\,
\, \forall n \ge 1 \,,\,\,\forall t \in [1, \bar t_1] \, .
\]
\end{lemma}

\begin{proof}
Assume first $g=0$.
For $t \ge 1$, we have for any $s \in (0,1)$,   taking $\eta(s)\in (0,1]$ such that 
$\eta t + (1-\eta)s = 1$, that 
$
\sum_i a_i = \sum_i a_i^{\eta t + (1-\eta)s} 
\le \bigl(\sum_i a_i^t \bigr)^\eta \bigl( \sum_i a_i^s \bigr)^{1-\eta} 
$ for any positive numbers $a_i$.
It follows that for any $W \in \hW^s$ with  $|W| \ge \delta/3$
and all $n \ge 1$,
\begin{equation}
\label{eq:interpol}
\sum_{W_i \in \cG_n^\delta(W)} |J_{W_i}T^n|^t
\ge \frac{\left( \sum_{W_i \in \cG_n^\delta(W)} |J_{W_i}T^n| \right)^{1/\eta}}
{\left( \sum_{W_i \in \cG_n^\delta(W)} |J_{W_i}T^n|^s \right)^{(1-\eta)/\eta}}
\ge \biggl (\frac {C_1 }3\biggr )^{1/\eta}  \left(  C_2[0]   \tfrac{\delta_0}{\delta}  \tfrac{2}{c_{2}} e^{nP_*(s)} \right)^{(\eta-1)/\eta} \, ,
\end{equation}
where we have used \eqref{eq:grow} for the lower bound in the numerator, and
Lemma~\ref{lem:extra growth} with $\varsigma=0$  and $\tilde t_1 =1$
combined with Proposition~\ref{prop:exact} for the upper bound in the denominator.
The factor $\delta_0/\delta$ comes from the fact that here we use $\cG_n^\delta(W)$, while
Lemma~\ref{lem:extra growth} uses $\cG_n^{\delta_0}(W)$. 
Since $\eta = (1-s)/(t-s)$, 
$e^{nP_*(s)(\eta-1)/\eta}=e^{-n(t-1)P_*(s)/(1-s)}$.  
 For fixed $\kappa>0$, Lemma~\ref{lem:t_*}
allows us to  choose $s=s(\kappa)\in (0, 1)$  (and hence $\eta_\kappa = \eta(s) >0$) 
such that $P_*(s)/(1-s) \le \chi_1 + \kappa$, completing the proof
for $g=0$ since $\eta(s)>(1-s)/\bar t_1$. 
For $g \ne 0$, \eqref{eq:low t>1 up} follows since
$|e^{S_ng}|_{C^0(W_i)} \ge e^{-n |g|_{C^0}}$ for each $W_i$. 

For the second inequality of the lemma when $|W| \ge \delta_0/3$, notice that (3.9) gives a lower bound
of $\frac{C_1 \delta_0}{3 \delta}$ in this case.  The rest of the estimate is the same (up to $C_\kappa$ changing
by a power of $\delta_0$).
\end{proof}

By 
definition, $\theta^t e^{\chi_1 (t-1)} < 1$ if $t <  s_1$.  Thus for $\bar t_1\in (1,s_1)$
there exists $ \kappa_1 = \kappa(\bar t_1) >0$  such that
\begin{equation}
\label{nofoot} \theta^{\bar t_1} e^{(\chi_1 + \kappa_1)(\bar t_1-1)}   < 1\, ,\,\,
\mbox{ and thus }\, 
\theta^t e^{(\chi_1 + \kappa_1)(t-1)} < 1\, ,\,\,
\forall t\le \bar t_1\, . 
\end{equation}
Our next lemma is the analogue of  Lemma~\ref{lem:short small} for $t>1$.

\begin{lemma}
\label{lem:short small t>1}
 Let $t_0\in (0,1)$,  $t_1 \in (1, t_*)$ and  $\bar t_1 \in (1, s_1)$. 
 Let $\kappa_1=\kappa(\bar t_1)$ satisfy \eqref{nofoot}.
Then for any $\ve>0$ there exist $\delta_1>0$ and $n_1 \ge 1$, 
such that\footnote{We take $\delta_1 <\delta_1(\ve)$ and $n_1\ge n_1(\ve)$  with
$\delta_1(\ve)$ and $n_1(\ve)$ from Lemma~\ref{lem:short small}.} for all $W \in \hW^s$ with $|W| \ge  \delta_1/3$, 
 \[
  \sum_{\substack{W_i \in \cG_n^{ \delta_1}(W) \\ |W_i| <  \delta_1/3}}
|J_{W_i}T^n|^t_{C^0(W_i)} |e^{ S_ng}|_{C^0(W_i)}
\le \ve \sum_{W_i \in \cG_n^{ \delta_1}(W) }
|J_{W_i}T^n|^t_{C^0(W_i)} |e^{ S_ng}|_{C^0(W_i)}
\, , \, \forall  t \in [1, \bar t_1 ]\, ,
\forall n \ge n_1 \, , 
  \]
for all $g\in C^1$
satisfying \eqref{eq:up one} and such that, in addition,
\begin{equation}
\label{eq:up two}
2|g|_{C^0}
<  - \bar t_1 \log \theta -(\chi_1+\kappa_1) (\bar t_1  -1) \,, \,
\mbox{ i.e. }\,\,
\theta^{ \bar t_1} e^{ (\chi_1+ \kappa_1)(\bar t_1 -1) + 2 | g|_{C^0}} < 1 \, .
\end{equation}
\end{lemma}

\medskip
Let $[t_0, \bar t_1]\subset (0,s_1)$.
For all $g\in C^1$ satisfying \eqref{eq:up one} and \eqref{eq:up two},
Lemma~\ref{lem:short small} and
Lemma~\ref{lem:short small t>1} for $\ve = 1/4$ give $n_1\ge 1$  and
$\delta_1>0$  such that
for  all  $n\ge n_1$ and all  $W \in \hW^s$ with $|W| \ge \delta_1/3$,
\begin{equation}
\label{eq:delta11}
\sum_{W_i \in L_n^{\delta_1}(W)} |J_{W_i}T^n|^t_{C^0(W_i)} 
|e^{S_ng}|_{C^0(W_i)}
\ge \tfrac 34 \sum_{W_i \in \cG_n^{\delta_1}(W)} |J_{W_i}T^n|^t _{C^0(W_i)}|e^{S_ng}|_{C^0(W_i)} \, ,
\, \forall t \in [t_0, \bar t_1] \, .
\end{equation}

\begin{proof}[Proof of Lemma~\ref{lem:short small t>1}]
Assume first that $g=0$.
Define $\rho = \theta^{\bar t_1} e^{(\chi_1 + \kappa_1)(\bar t_1 -1)} < 1$ (with ${\kappa_1} (t_0, t_1, \bar t_1)$ from \eqref{nofoot}).
For $\ve >0$, choose 
$\bar \ve >0$ such that 
$\frac{6 \delta_0 C_0^2}{C_{\kappa_1}(1-\rho)} \frac{\bar \ve}{1-\bar \ve} < \frac{\ve}{2}$ 
and choose $m_1$ such that $\rho^{m_1} \le \bar \ve$.
 Next, choose $ \delta_1>0$ such that \eqref{eq:n1 short} holds for all $n \le 2  m_1$.
 
For $W \in \hW^s$ with $|W| \ge \delta_1/3$, let $S_n^{\delta_1}(W)$ denote the elements $W_i \in \cG_n^{\delta_1}(W)$ such that $|W_i|< \delta_1/3$,
and let $L_n^{\delta_1}(W) = \cG_n^{\delta_1}(W) \setminus S_n^{\delta_1}(W)$.
For $n \ge m_1$, write $n = \ell m_1 + r$ for some $0 \le r < m_1$.  As in the proof of Lemma~\ref{lem:short small},
we group elements $W_i \in S_n^{\delta_1}(W)$ according to the largest $k$ such that
$T^{(\ell - k)m_1 + r}W_i \subset V_j \in L_{k m_1}^{\delta_1}(W)$.  
Such a $k \in [0, \ell-1]$ must exist since $|W| \ge \delta_1/3$ while $|W_i| < \delta_1/3$.
By choice of $\delta_1$ and $m_1$,
we may apply \eqref{eq:n1 short}
for $(\ell - k)m_1 + r$ iterates to obtain
\[
\sum_{W_i \in S_n^{\delta_1}(W)} |J_{W_i}T^n|^t_{C^0(W_i)} 
\le \sum_{k=0}^{\ell-1} \sum_{V_j \in L^{\delta_1}_{k m_1}(W)} |J_{V_j}T^{km_1}|^t_{C^0(V_j)} C_0 \theta^{t ((\ell - k)m_1 +r)} \, .
\]

Next, for each $k$, we consider each $V_j \in L^{\delta_1}_{k m_1}(W)$ as being contained in an element
$U_i \in \cG_{k m_1}^{\delta_0}(W)$.  Since $|V_j| \ge \delta_1/3$, there are at most $3 \delta_0/\delta_1$ 
$V_j$ corresponding to each such $U_i$.  Then we group each $U_i \in \cG_{km_1}^{\delta_0}(W)$ according to
its most recent long ancestor $W_a \in L_j^{\delta_0}(W)$ for some $j \in [0, k m_1]$.  Note that
$j=0$ is possible if $|W| \ge \delta_0/3$.  However, if $|W|  < \delta_0/3$ and no such time $j$ exists for $U_i$, then we 
associate such $U_i$ with index $j=0$ in any case.  In either case, $U_i \in \cI_{km_1}^{\delta_0}(W)$,
where $\cI_{km_1}^{\delta_0}(W) = \cI_{k m_1}(W)$ as defined in Lemma~\ref{lem:short growth}.
With these groupings, we estimate,
\[
\begin{split}
\sum_{V_j \in L^{\delta_1}_{km_1}(W)} |J_{V_j}T^{km_1}|^t_{C^0(W_i)}
& \le \frac{3 \delta_0}{\delta_1} \left( \sum_{U_i \in \cI_{km_1}^{\delta_0}(W) } |J_{U_i}T^{km_1}|^t_{C^0(U_i)}  \right. \\
& \qquad \left. + 
\sum_{j=1}^{km_1} \sum_{W_a \in L_j^{\delta_0}(W)}  |J_{W_a}T^j|^t \sum_{U_i \in \cI_{km_1 -j}^{\delta_0}(W_a)}
|J_{U_i}T^{km_1-j}|^t_{C^0(U_i)} \right) \\
& \le \frac{3 \delta_0}{\delta_1} \left( C_0 \theta^{t km_1 } + \sum_{j=1}^{km_1} \sum_{W_a \in L_j^{\delta_0}(W)}  |J_{W_a}T^j|^t C_0 \theta^{t (km_1 - j)} \right) \, ,
\end{split}
\]
where we have applied Lemma~\ref{lem:short growth}
to each collection $\cI_{km_1 - j}^{\delta_0}(W)$.
Collecting these estimates yields,
\[
\sum_{W_i \in S_n^{\delta_1}(W)} |J_{W_i}T^n|^t_{C^0(W_i)} 
\le \frac{3 \delta_0}{\delta_1} \left( C_0^2 \frac{n}{m_1} \theta^{tn}  + \sum_{k=1}^{\ell-1}  \sum_{j=1}^{km_1}
C_0^2 \theta^{t (n-j)} \sum_{W_a \in L_j^{\delta_0}(W)} |J_{W_a}T^j|^t_{C^0(W_a)}  \right) \, .
\]

For fixed $k>0$ and each $j$ for which $L_j^{\delta_0}(W)$ is not empty, we group elements $W_i \in \cG_n^{\delta_1}(W)$ according to which $W_a \in \cG_j^{\delta_0}(W)$
they descend from.  Then we use Lemma~\ref{lem:low t>1} and the distortion bound \eqref{eq:ratio} to estimate
a lower bound,
\[
\begin{split}
\sum_{W_i \in \cG_n^{\delta_1}(W)} |J_{W_i}T^n|^t_{C^0(W_i)} 
& \ge 
\sum_{W_a \in L_j^{\delta_0}(W)} \tfrac 12 |J_{W_a}T^j|^t_{C^0(W_a)} \sum_{W_i \in \cG_{n-j}^{\delta_1}(W_a)}
|J_{W_i}T^{n-j}|^t_{C^0(W_i)} \\
& \ge \frac{C_{\kappa_1}}{2\delta_1} e^{-(n-j) (\chi_1 + \kappa_1) (t-1)} \sum_{W_a \in L_j^{\delta_0}(W)} |J_{W_a}T^j|^t_{C^0(W_a)} \, .
\end{split}
\]
Combining upper and lower bounds yields (using \eqref{eq:low t>1 up} for the term\footnote{A better estimate is possible in the case
$|W| \ge \delta_0/3$, but we will not need this.} corresponding to $j=0$),
\[
\begin{split}
\frac{\sum_{W_i \in S_n^{\delta_1}(W)} |J_{W_i}T^n|^t_{C^0(W_i)}}{\sum_{W_i \in \cG_n^{\delta_1}(W)} |J_{W_i}T^n|^t_{C^0(W_i)} }
& \le
\frac{ \frac{3 \delta_0 n}{\delta_1 m_1} C_0^2 \theta^{tn}}{C_{\kappa_1} \delta_1^{\frac{1}{\eta_{\kappa_1}}-1} e^{-n(\chi_1+\kappa_1)(t-1)}} \\
& \qquad
 + \sum_{k=1}^{\ell-1} \sum_{j=1}^{km_1} 
 \frac{ \frac{3\delta_0}{\delta_1} C_0^2 \theta^{t (n-j)} \sum_{W_a \in L_j^{\delta_0}(W)} |J_{W_a}T^j|^t_{C^0(W_a)} }
 {\frac{C_{\kappa_1}}{2\delta_1} e^{-(n-j) (\chi_1 + \kappa_1)(t-1)} \sum_{W_a \in L_j^{\delta_0}(W)} |J_{W_a}T^j|^t_{C^0(W_a)} } \\
 & \le
 \frac{3 \delta_0 n C_0^2}{C_{\kappa_1} m_1} \delta_1^{-\frac{1}{\eta_{\kappa_1}}} \rho^n
 + \sum_{k=1}^{\ell-1} \sum_{j=1}^{km_1} \frac{6 \delta_0 C_0^2}{C_{\kappa_1}} \rho^{n-j} \\
 & \le \frac{3 \delta_0 n C_0^2}{C_{\kappa_1} m_1} \delta_1^{-\frac{1}{\eta_{\kappa_1}}} \rho^n
 +  \frac{6 \delta_0 C_0^2}{C_{\kappa_1}(1-\rho)} \sum_{i=1}^{\ell-1} \bar \ve^i
\le \frac{3 \delta_0 n C_0^2}{C_{\kappa_1} m_1} \delta_1^{-\frac{1}{\eta_{\kappa_1}}} \rho^n
 + \frac{\ve}{2} \, ,
   \end{split}
\]
by choice of $\bar \ve$, where we have used the fact that $\sum_{k=1}^{\ell-1} \sum_{j=1}^{km_1} \rho^{n-j}
\le \sum_{k=1}^{\ell-1} \frac{\rho^{n-k m_1}}{1-\rho} \le \sum_{i=1}^{\ell-1} \frac{\bar \ve ^i}{1-\rho}$.
Finally, we choose $n_1 \ge m_1$ sufficiently large that the first term is less than
$\frac{\ve}{2}$ for all $n \ge n_1$, completing the proof in the case $g=0$. 
 
If $g\ne 0$, define $\rho_g = \theta^{\bar t_1} e^{ (\chi_1 + \kappa_1)(\bar t_1-1) + 2 |g|_{C^0}} < 1$.
For $\ve >0$, pick $\bar \ve >0$ such that
$\frac{12 \delta_0 C_0^2}{C_{\kappa_1}(1-\rho_g)} \frac{\bar \ve}{1- \bar \ve} < \frac{\ve}{2}$.
Then choose $m_1$ and $\delta_1$ as in the case $g=0$, but with $\rho_g$ in place of $\rho$.
(The choices $ m_1$ and $ \delta_1$ are uniform for $g$ satisfying \eqref{eq:up one} 
and \eqref{eq:up two}.)
The argument then follows precisely as above with the inclusion of $g$ 
as in \eqref{eq:short bound up} in the proof of Lemma~\ref{lem:short small}.
The fact that the distortion constant for $g$ is at most $2$ is used for the lower bound for each $W_a \in L^{\delta_0}_j(W)$ appearing in the
denominator:
\[
\begin{split}
\sum_{W_i \in \cG_{n-j}^{\delta_1}(W_a) } & |J_{W_i}T^n|^t_{C^0(W_i)}  |e^{S_ng}|_{C^0(W_i)} \\
& \ge \tfrac 14 |J_{W_a}T^j|^t_{C^0(W_a)} |e^{S_jg}|_{C^0(W_a)}
\sum_{W_i \in \cG_{n-j}^{\delta_1}(W_a)} |J_{W_i}T^{n-j}|^t_{C^0(W_i)}
|e^{S_{n-j}g}|_{C^0(W_i)} \\
& \ge  \tfrac 14  |J_{W_a}T^j|^t_{C^0(W_a)} |e^{S_jg}|_{C^0(W_a)}
C_{\kappa_1} \delta_1^{-1} e^{-(n-j)(\chi_1+\kappa_1)(t-1) - (n-j) |g|_{C^0}}\, .
\end{split}
\]
\end{proof}

Our final lemma of this section is the analogue of 
Lemma~\ref{lem:most long} for $t > 1$.  Define $\cA_n(\delta)$ as in
\eqref{eq:An delta}.

\begin{lemma}
\label{lem:most long t>1}
Let $t_0\in (0,1)$,   $t_1 \in (1, t_*)$ and  $\bar t_1\in (1,s_1)$.
 Let $\delta_2 >0$
be as in Lemma~\ref{lem:most long}.   There exists a decreasing function $c_0 : [0, \infty) \to \mathbb{R}^+$
such that for any $\upsilon \ge 0$ and  any $g\in C^1$ satisfying \eqref{eq:up one}, \eqref{eq:up two},
and $|\nabla g|_{C^0}\le \upsilon$, we have
  \[
  \sum_{A \in \cA_n(\delta_2) }\sup_{x \in A \cap M'} |J^sT^n(x)|^t e^{ S_ng(x)}
\ge  c_0(\upsilon)  Q_n(t, g) \, , \, \forall n \in \mathbb{N}\, ,\,\, \forall t \in [1, \bar t_1 ]\, .
  \]
\end{lemma}

\begin{proof}
The calculations in the proof of Lemma~\ref{lem:most long}
for $g=0$ are valid for all
$t>0$ up through \eqref{eq:group}.  To proceed, we replace Sublemma~\ref{sub:long}
by the following.

\begin{sublem}
\label{sub:long t>1}
Let $t_0\in (0,1)$,   $t_1 \in (1, t_*)$ and  $\bar t_1\in (1,s_1)$.
There exists $C>0$ such that for 
all $t\in [1, \bar t_1]$, each $n_1 \le j \le n-n_2-1$ and all $B' \in \bL_{-n+1+j}^{\delta_2}$,
\[
|\cos \vf|^t_{T^{-n+1+j}B'} |J^uT^{-n+1+j}|^t_{B'}
\le C \delta_2^{-1/\eta_{\kappa_1}} e^{j( \chi_1 + \kappa_1  )(t-1)} \sum_{B \in L_j(B')} |\cos \vf|^t_{T^{-n+1}B} |J^uT^{-n+1}|^t_B \, ,
\]
where $L_j(B')$ denotes the collection of elements $B \in G_j(B')$ with $\diam^u(B) \ge \delta_2/3$.
\end{sublem}

\begin{proof}
The proof of this sublemma only requires one adjustment to the estimate in \eqref{sub:est}.
Using the same notation as in Sublemma~\ref{sub:long}, we have
\begin{align*}
\sum_{B \in L_j(B')} |\cos \vf|^t_{T^{-n+1}B} & |J^uT^{-n+1}|^t_B
\ge C |\cos \vf|^t_{T^{-n+1+j}B'} |J^uT^{-n+1+j}|^t_{B'} \sum_{B \in L'_j(B')} |J^uT^{-j}|^t_B \\
& \ge C' |\cos \vf|^t_{T^{-n+1+j}B'} |J^uT^{-n+1+j}|^t_{B'} \delta_2 \sum_{U_i \in L^{\delta_2}_j(U)} 
|J_{U_i}T^{-j}|^t_{C^0(U_i)} \\
& \ge C' \delta_2 |\cos \vf|^t_{T^{-n+1+j}B'} |J^uT^{-n+1+j}|^t_{B'} 
\tfrac 34 \sum_{U_i \in \cG^{\delta_2}_j(U)} 
|J_{U_i}T^{-j}|^t_{C^0(U_i)} \\
& \ge C'' \delta_2^{1/\eta_{\kappa_1}} |\cos \vf|^t_{T^{-n+1+j}B'} |J^uT^{-n+1+j}|^t_{B'} 
C_{\kappa_1} e^{-j(\chi_1 + \kappa_1)(t-1)} \, ,
\end{align*}
where the only new justifications are that we use the time reversal of  \eqref{eq:delta11}
in the third inequality since $\delta_2 \le \delta_1$ and $|U| \ge \delta_2/3$, 
and in the fourth inequality, we apply the time reversal of  Lemma~\ref{lem:low t>1} 
since $j \ge n_1$ with $\kappa_1$ from \eqref{nofoot} .
\end{proof}

Using Sublemma~\ref{sub:long t>1}, we estimate the right hand side of \eqref{eq:group}
as in the proof of Lemma~\ref{lem:most long}, summing
over $B' \in \bL_{-n+1+j}^{\delta_2}$
and recalling that if $B \in L_j(B')$, then
$B \in \bL_{-n+1}^{\delta_2}$ and each such $B$ is associated with a unique $B'$:
\begin{align}
\label{???}\sum_{B \in \mathbb{S}_{-n+1}^{\delta_2}}&
|\cos \vf|^t_{T^{-n+1}B} |J^uT^{-n+1}|^t_{B}
 \; \le \sum_{B' \in \bS_{-n_2}^{\delta_2}} 
 C \theta^{t (n-n_2-1)} |\cos \vf|^t_{T^{-n_2}B'} |J^uT^{-n_2}|^t_{B'}   \\
\nonumber& \qquad + \sum_{j=n_1}^{n-n_2-1}  C \delta_2^{-1-1/\eta_{\kappa_1}} \theta^{tj} e^{j( \chi_1+\kappa_1  )(t-1)}
\sum_{B \in \bL_{-n+1}^{\delta_2}}
 |\cos \vf|^t_{T^{-n+1}B} |J^uT^{-n+1}|^t_{B} \\
& \le C_{n_2} \theta^{tn}+ C \delta_2^{-1-1/\eta_{\kappa_1}} \sum_{B \in \bL_{-n+1}^{\delta_2}}
 |\cos \vf|^t_{T^{-n+1}B} |J^uT^{-n+1}|^t_{B} \, ,
\nonumber \end{align}
for some constant $C_{n_2} >0$ depending only on $n_2$, where we have used the fact
that $\theta^t e^{(\chi_1 + \kappa_1)(t-1)} < 1$ to sum over $j$.

The sum over $B \in \bL_{-n+1}^{\delta_2}$ shrinks at a rate bounded below by 
$C \delta_2^{1/\eta_{\kappa_1}} e^{-n( \chi_1 + \kappa_1  )(1-t)}$
by the proof of Sublemma~\ref{sub:long t>1}.  Thus we may choose $n_3 \ge n_2$ large enough that 
$C_{n_2} \theta^{tn} \le C  \delta_2^{-1 }  e^{-n(\chi_1 + \kappa_1)( t-1)}$ for all $n \ge n_3$, which implies,
\[
\sum_{B \in \mathbb{S}_{-n+1}^{\delta_2}}
|\cos \vf|^t_{T^{-n+1}B} |J^uT^{-n+1}|^t_{B}
\le
C \delta_2^{-1-1/\eta_{\kappa_1}} \sum_{B \in \bL_{-n+1}^{\delta_2}}
 |\cos \vf|^t_{T^{-n+1}B} |J^uT^{-n+1}|^t_{B} \, .
\]
The proof of the Lemma~\ref{lem:most long} proceeds without further changes from this point, ending the proof of Lemma~\ref{lem:most long t>1} if $g=0$.

If $g\ne 0$, choosing $\delta_2$  as in the proof of Lemma~\ref{lem:most long} 
implies that \eqref{eq:switch B short up} and \eqref{eq:group up} remain as written.  The only
change required in the proof is to use the lower bound \eqref{eq:low t>1 up}  with $\kappa = \kappa_1$ to
prove the analogue of Sublemma~\ref{sub:long}:  There exists $C>0$ such that for all
$n_1 \le j \le n-n_2-1$ and $B' \in \mathbb{L}^{\delta_2}_{-n+1+j}$,
\begin{align*}
&|\cos \vf|^t_{T^{-n+1+j}B'}  |J^uT^{-n+1+j}|^t_{B'} |e^{ S_{n-1-j}g}|_{T^{-n+1+j}B'} \\
&\qquad\qquad\le C \delta_2^{-1/\eta_{\kappa_1}} e^{j(\chi_1 + \kappa_1)(t-1)} e^{ j |g|_{C^0}} 
\sum_{B \in L_j(B')} 
|\cos \vf|^t_{T^{-n+1}B} |J^uT^{-n+1}|^t_{B} |e^{ S_{n-1}g}|_{T^{-n+1}B} \, .
\end{align*}
We then proceed as in \eqref{???}  
using 
the contraction provided by 
\eqref{eq:up two} to sum over $j$.
\end{proof}


\subsection{Lower Bounds on Complexity. Exact Exponential Growth of $Q_n(t,g)$}
\label{sec:complexity}

In order to conclude that the spectral radius of $\cL_t$ on $\cB$ is 
 $e^{P_*(t)}$ and to control the peripheral spectrum of $\cL_t$, we shall
establish the exact exponential growth of $Q_n(t)$. 

The lower bound on the spectral radius of $\cL_t$ is a consequence of the following lemma, 
guaranteeing that the weighted complexity of long elements of $\hW^s$ grows at 
the rate $Q_n(t,g)$.

\begin{proposition}
\label{prop:lower bound}
Let $t_0\in (0,1)$,   $t_1 \in (1, t_*)$ and  $\bar t_1\in (1,s_1)$.
 There exists a decreasing function $c_1 : [0, \infty) \to \mathbb{R}^+$
such that for any $\upsilon \ge 0$ and 
any
$W \in \hW^s$ with $|W| \ge \delta_1/3$, 
\begin{equation}
\label{eq:lower bound up}
\sum_{W_i \in \cG_n(W)} |J_{W_i}T^n|^t_{C^0(W_i)} |e^{ S_n g}|_{C^0(W_i)}
\ge  c_1(\upsilon)  Q_n(t, g) \, , \, \forall n \ge 1 \, , \,\, \forall t \in [t_0, \bar t_1 ]\, ,
\end{equation}
for any $g\in C^1$ with $|\nabla g|_{C^0}\le \upsilon$ and 
such that \eqref{eq:up one}  and \eqref{eq:up two} hold.
\end{proposition}

\begin{proof}
As usual we first consider $g=0$.
The main idea of the proof is to show that for each curve 
$W \in \cW^s$ with $|W| \ge \delta_1/3$,
the image $T^{-n}W$ intersects a positive fraction of elements of $\cM_n^{0, \bH}$, weighted
by $|J^sT^n|^t$, for $n$ large enough.
The mixing property of $\mu_{\tiny{\mbox{SRB}}}$ is instrumental here.

To do this, we recall the construction of locally maximal homogeneous Cantor rectangles from \cite[Section~7.12]{chernov book} 
(and similar to those used in
\cite[Section~5.3]{max} where we worked\footnote{The construction in
\cite[Section~7.12]{chernov book}  uses $\cW^s_{\mathbb {H}}$, but since each $V$  in 
$\cW^s$ are unions of manifolds $W_i$ in $\cW^s_{\mathbb {H}}$, if the $W_i$ cross properly, so does $V$.} with $\cW^s$ instead
of $\cW^s_{\mathbb {H}}$).  We call $D \subset M$ a {\em solid rectangle} if $D$ is a closed, simply connected region whose
boundary consists of two homogeneous unstable and two stable manifolds.  Given such a rectangle $D$, the maximal
Cantor rectangle $R(D)$ in $D$ is the union of all points in $D$ whose homogeneous stable and unstable manifolds completely cross
$D$.  Note that $R(D)$ is closed and contains the boundary of $D$ \cite[Section~7.11]{chernov book}, but is not simply connected due to the effect of the singularities, which create, for any $\ve>0$,
a dense set of points with stable and unstable manifolds shorter than $\ve$.

In what follows, we restrict to Cantor rectangles with sufficiently high density, i.e.,
\begin{equation}
\label{eq:dense R}
\inf_{x \in R} \frac{m_{W^u}(W^u(x) \cap R)}{m_{W^u}(W^u(x) \cap D(R))} \ge 0.99 \, ,
\end{equation}
where $m_{W^u}$ denotes arclength measure along an unstable manifold.
We say that a homogeneous stable curve $W \in   \hW^s_H$ {\em properly crosses} a maximal homogeneous Cantor rectangle
$R = R(D)$ satisfying \eqref{eq:dense R} if 
 $W$ crosses both unstable sides of $D$,
 and, in addition, for every $x \in R$, the point $W \cap W^u(x)$ divides the curve $W^u(x) \cap D(R)$ in a ratio between
  $0.1$ and $0.9$, and on either side of $W \cap W^u(x)$, the density of $R$ in $W^u(x) \cap D(R)$ is at least $0.9$.  
Reversing the roles of stable and unstable manifolds, we obtain the analogous definition of an unstable curve 
properly crossing a Cantor rectangle.

By \cite[Lemma~7.87]{chernov book}, we choose a finite number of locally maximal homogeneous Cantor
rectangles $\cR({\delta_2}) = \{ R_1, \ldots, R_k\}$ satisfying \eqref{eq:dense R} and its analogue along stable manifolds, with the property 
that any homogeneous stable or unstable curve of length at least
{ $\delta_2/3$} properly crosses at least one of them.  Let {$\delta_2'$} be the minimum diameter of the rectangles in
$\cR(\delta_2)$ and note that $\delta_2'$ is a function only of $\delta_2$.

Now fix $n \ge 1$ and let $\cA_n^i \subset \cA_n(\delta_2)$ denote those
elements $A \in \cA_n(\delta_2)$ such that $B_{n-1}(A)$ contains an homogeneous unstable curve of length at least
$\delta_2/3$ that properly crosses $R_i$.  Due to Lemma~\ref{lem:most long} 
 for $t \le 1$ and Lemma~\ref{lem:most long t>1} for $t>1$, there exists $i^*$ such that
\begin{equation}
\label{eq:k part}
\sum_{A \in \cA_n^{i^*}} \sup_{x \in A \cap M'} |J^sT^n(x)|^t \ge \frac{c_0}{k} Q_n(t) \, .
\end{equation}

Fix an arbitrary homogeneous $W \in \hW^s$ with $|W| \ge \delta_1/3$, and let $R_j \in \cR(\delta_2)$ denote the Cantor rectangle that is
properly crossed by $W$ { (recalling that $\delta_2 \le \delta_1$)}.  By the mixing property of $\musrb$ and \cite[Lemma~7.90]{chernov book}, there
exists $N_1 = N_1(\delta_2) \ge 1$ such that $T^{-N_1}R_i$ has a homogeneous
connected component that properly crosses $R_{i^*}$, for all $i = 1, \ldots, k$.  In particular, 
$T^{-N_1}R_j$ properly crosses $R_{i^*}$,  so an element of $\cG_{N_1}(W)$ properly crosses $R_{i^*}$.

Let $W_1 \in \cG_{N_1}(W)$ denote the component of $T^{-N_1}W$ that properly crosses $R_{i^*}$ and note that $W_1$
crosses $B_{n-1}(A)$ for all $A \in \cA_n^{i^*}$.  Since $W_1$ is homogeneous and
$N_1 \ge 1$,  $W_1$ cannot cross a singularity line in $T^{-1}\cS_0$ (since then
the curve would have been subdivided at time $N_1-1$), and so for each such $A$, $W_1$ crosses 
an element $B'_A \in \cM_{-n+1}^{1, \mathbb{H}}$, $B'_A \subset B_{n-1}(A)$.  
Let $V'_A = W_1 \cap B'_A$ and let 
$V_A = T^{-n+1}V'_A$.
Then $V_A$ is a homogeneous component belonging to an element of $\cG_{n-1+N_1}(W)$.  
By Lemma~\ref{lem:comparable}, 
recalling the notation
$|J^sT^k|^t_{A'} =  \sup_{x \in A' \cap M'} |J^sT^k(x)|^t$ from \eqref{forbrevity},
\[
|J_{V_A}T^{n-1}|^t_{C^0(V_A)} = e^{\pm tC} |J^sT^{n-1}|^t_{T^{-n+1}B_{n-1}(A)} \, ,
\]
since $T^{-n+1}B_{n-1}(A) \in \cM_0^{n-1, \mathbb{H}} \bigvee T^{-n+1}{\cH}$. 
By definition, $T^{-n+1}B_{n-1}(A)$ contains $A$.  Thus,
\begin{equation}
\label{eq:first compare}
|J^sT^{n-1}|^t_A \le e^{tC} |J_{V_A}T^{n-1}|^t_{C^0(V_A)} \, .
\end{equation}

Next, we wish to compare $J^sT$ on $T^{n-1}A$ with $J_{V'_A}T$.  Since $V'_A \subset W_1 \subset T^{-N_1}W$, we have that
$TV'_A$ is a stable curve, and so is $TW_1$, so that $J_{V'_A}T = e^{\pm C_d} J_{W_1}T = e^{\pm C_d} k^{-q}$, where 
$k$ is the index of the homogeneity strip containing $W_1$.  But since  $|W_1| \ge \delta_2'$ 
(since $W_1$ properly crosses $R_{i^*}$), we have
$k \le (\delta_2')^{ -1/(q+1)}$ and so $J_{W_1}T \ge C (\delta_2')^{ q/(q+1)}$.  Since $J^sT \le e^{C_d}$, we have, using
\eqref{eq:first compare}, that
$|J^sT^n|^t_A \le C (\delta_2')^{ -tq/(q+1)} |J_{V_A}T^n|^t_{C^0(V_A)}$.
Then summing over $A \in \cA_n^{i^*}$, we obtain,
\begin{equation}
\label{eq:i* compare}
\sum_{A \in \cA_n^{i^*}} |J^sT^n|^t_A \le C (\delta_2')^{-t q/(q+1)} \sum_{V_i \in \cG_n(TW_1)} |J_{V_i}T^n|^t_{C^0(V_i)} \, .
\end{equation}

Next, we express the sum over $\cG_{n+ N_1-1}(W)$ in two  ways.  On the one hand, by Lemma~\ref{lem:extra growth},
\begin{equation}
\label{eq:upper}
\begin{split}
\sum_{V_j \in \cG_{n+N_1-1}(W)} |J_{V_j}T^{n+N_1-1}|^t_{C^0(V_j)}
& \le \sum_{W_i \in \cG_n(W)} |J_{W_i}T^n|^t_{C^0(W_i)} \sum_{V_j \in \cG_{N_1-1}(W_i)} |J_{V_j}T^{N_1-1}|^t_{C^0(V_j)} \\
& \le C_2[0] Q_{N_1-1}(t) \sum_{W_i \in \cG_n(W)} |J_{W_i}T^n|^t_{C^0(W_i)} \, .
\end{split}
\end{equation}
On the other hand, letting $W_1'$ be the element of $\cG_{N_1-1}(W)$ containing $TW_1$,
\begin{equation}
\label{eq:lower}
\begin{split}
\sum_{V_j \in \cG_{n+N_1-1}(W)} |J_{V_j}T^{n+N_1-1}|^t_{C^0(V_j)}
& \ge e^{-t C_d} |J_{W_1'}T^{N_1-1}|^t_{C^0(W_1')} \sum_{V_i \in \cG_n(W_1')} |J_{V_i}T^n|^t_{C^0(V_i)} \\
& \ge e^{-t C_d} C' (\delta_2')^{t  \big(\frac{2q+1}{q+1} \big)^{N_1-1} }  \sum_{V_i \in \cG_n(W_1')} |J_{V_i}T^n|^t_{C^0(V_i)} \, ,
\end{split} 
\end{equation}
where the lower bound on $|J_{W_1'}T^{N_1-1}|^t_{C^0(W_1')}$ comes from the fact that $|W_1'| \ge \delta_2'$ and
for a stable curve $V$ { such that $V$ and $T^{-1}V$ are both homogeneous}, $|T^{-1}V| \le C |V|^{ \frac{q+1}{2q+1}}$, and this bound can be iterated $N_1-1$ times
as in \cite[eq. (5.3)]{max}.

Combining \eqref{eq:i* compare}, \eqref{eq:upper} and \eqref{eq:lower}, and recalling \eqref{eq:k part} yields,
\[
\sum_{W_i \in \cG_n(W)} |J_{W_i}T^n|^t_{C^0(W_i)} \ge (C_2[0] )^{-1} Q_{N_1-1}(t)^{-1} C'' (\delta_2')^{t  \big( \frac{2q+1}{q+1} \big)^{N_1} } 
\frac{c_0}{k} Q_n(t) \, ,
\]
which completes the proof of the proposition if $g=0$.

\smallskip
If $g\ne 0$, starting  as above, we choose the finite family
of Cantor rectangles
$\cR(\delta_2)$ in the same way, and find an index $i^*$ such that
the analogue of \eqref{eq:k part}
\[
\sum_{A \in \cA_n^{i^*}} \sup_{x \in A \cap M'} |J^sT^n(x)|^t e^{ S_n g(x)}
\ge \frac{c_0}{k} Q_n(t, g) \, ,
\]
holds, using Lemma~\ref{lem:most long} if $t\le 1$
and Lemma~\ref{lem:most long t>1} if $t> 1$.
Fixing $W \in \cW^s$, choosing $N_1$ as above,
and using the same  notation introduced there, we obtain the 
 modification of \eqref{eq:first compare},
\[
| |J^sT^{n-1}|^t e^{S_{n-1} g}|_A \le e^{tC} (1+\bar C C_*  |\nabla g |_{C^0}) ||J_{V_A}T^{n-1}|^t   e^{S_{n-1} g}|_{C^0(V_A)} \, ,
\]
applying \eqref{eq:phi A}.  Next, \eqref{eq:i* compare} needs only the multiplication 
by $e^{ S_n g}$ to each term on both sides, up to replacing the constant
$C$ by $C (1+ \bar C\, C_* \cdot |\nabla g |_{C^0})$.  The upper bound \eqref{eq:upper} requires only a change of constant
to $C_2[0] Q_{N_1-1}(t, g)$, using Lemma~\ref{lem:extra growth} with $\varsigma=0$,
while the lower bound \eqref{eq:lower} requires the added factor
$e^{- (N_1-1)  |g|_{C^0}}$ on the right hand side.  Since $N_1$ is fixed (depending only on
$\cR(\delta_2)$), these bounds are combined as in the 
case $g=0$ to complete the proof of the proposition.
\end{proof}

The following important consequence of Proposition~\ref{prop:lower bound}  will be used
to characterize the peripheral spectrum of $\cL_t$.

\begin{proposition}[Exact Exponential Growth of $Q_n(t,g)$]
\label{prop:exact}
Let $t_0\in (0,1)$,  $t_1 \in (1, t_*)$ and  $\bar t_1\in (1,s_1)$.
There exists a decreasing function $c_2 : [0, \infty) \to \mathbb{R}^+$
such that for any $\upsilon \ge 0$ and 
any $g\in C^1$ with $|\nabla g |_{C^0}\le \upsilon$ and 
such that \eqref{eq:up one}  and \eqref{eq:up two} hold, we have 
\begin{equation}
\label{eq:exact up}
e^{n P_*(t, g)} \le Q_n(t, g) \le \frac{2}{c_2(\upsilon)} e^{n P_*(t, g)}\, ,\,
\forall t \in [t_0,  \bar t_1 ]\, ,\,\, 
\forall n \ge 1 \, .
\end{equation}
\end{proposition}

\begin{proof}
The lower bound follows immediately from submultiplicativity of $Q_n(t,g)$
(obtained in the proof of Proposition~\ref{prop:Ptilde} for any $t>0$ and
$g\in C^1$) since then
$P_*(t,g) = \inf_n \frac 1n \log Q_n(t,g)$.  

To obtain the upper bound
for $g=0$, we first prove the following supermultiplicative
property:  There exists $c_2>0$ such that for all $t\in [t_0, \bar t_1]$ and for any $j,n \ge 1$,
\begin{equation}
\label{eq:super}
Q_{n+j}(t) \ge c_2 Q_n(t) Q_j(t) \, .
\end{equation}
Let $W \in \hW^s$ with $|W| \ge \delta_1/3$.  For $n$, $j \ge 1$, by 
Lemma~\ref{lem:extra growth} 
with $\varsigma = 0$,
\[
\sum_{W_i \in \cG_{n+j}^{\delta_1}(W)} |J_{W_i}T^{n+j}|^t_{C^0(W_i)} \le C_2[0]  \,  Q_{n+j}(t) \, .
\]
On the other hand, if $n \ge n_1$, then using Lemma~\ref{lem:distortion},
\begin{align*}
\sum_{W_i \in \cG_{n+j}^{\delta_1}(W)}  |J_{W_i}T^{n+j}|^t_{C^0(W_i)}
& \ge C  \sum_{V_k \in \cG_n^{\delta_1}(W)}  |J_{V_k}T^n|^t_{C^0(V_k)}
\sum_{W_i \in \cG_j^{\delta_1}(V_k)} |J_{W_i}T^j|^t_{C^0(W_i)} \\
& \ge C \sum_{V_k \in L_n^{\delta_1}(W)} |J_{V_k}T^n|^t_{C^0(V_k)}
\sum_{W_i \in \cG_j^{\delta_1}(V_k)} |J_{W_i}T^j|^t_{C^0(W_i)} \\
& \ge C \sum_{V_k \in L_n^{\delta_1}(W)} |J_{V_k}T^n|^t_{C^0(V_k)} c_1 Q_j(t) \\
& \ge C c_1 Q_j(t) \tfrac 34 \sum_{V_k \in \cG_n^{\delta_1}(W)} |J_{V_k}T^n|^t_{C^0(V_k)} 
\; \ge \; C' c_1^2 Q_j(t) Q_n(t) \, , 
\end{align*}
where in the third and fifth inequalities, we have used Proposition~\ref{prop:lower bound} and in the
fourth inequality we have applied \eqref{eq:delta11}.
This proves  \eqref{eq:super} for $n \ge n_1$, and the case
 $n \le n_1$ follows by adjusting the constant $c_2$.
(Note that $c_2$ is uniform in $t$.)
The proof of the upper bound on $Q_n(t)$ then proceeds precisely as in the proof of
\cite[Proposition~4.6]{max}.
The case of nonzero $g$ is identical.
\end{proof}


\subsection{Growth Lemmas and Exact Exponential Growth for $t \in (s_1,t_*)$}
\label{sec:boot}

The main result of this section is Proposition~\ref{prop:summary}
which extends Propositions~\ref{prop:lower bound} and \ref{prop:exact} to all $t<t_*$.
The constant
$t_*>1$ is defined by \eqref{eq:t* def}, while $s_1>1$ is introduced  in Definition~\ref{def:t_*}.
What we have proved up to now suffices
     to establish all the results of  Sections~\ref{sec:spec}--\ref{analytic}  for $t\in (0,s_1)$. In particular  Theorem~\ref{thm:geo var} holds in a neighbourhood of
     $t=1$, so  we know that $s_1\le t_*$.

Recall that $t_1 \in (1, t_*)$ is fixed in Definition~\ref{def:k0d0}, determining 
$\theta \in (\Lambda^{-1}, \Lambda^{-1/2})$ and
 our main statements are for $t\in [t_0,t_1]$.
If $s_1 \ge t_1$, there is nothing to do. 
Otherwise,  $\theta^{s_1} < e^{P(s_1)} \le e^{P_*(s_1)}$ by 
Proposition~\ref{prop:pressure}.  Since $P_*(t)$ is convex and decreasing, the 
left-hand slopes are lower semi-continuous, so
we may choose $\bar t_1 \in(1,s_1)$ so that  the intersection point $s_2(\bar t_1)$ 
between the tangent line to $P_*(t)$ (from the left) at
$t = \bar t_1$ and the line $t \log \theta$ satisfies $s_2 > s_1$.  
Indeed, we have
\begin{equation}
\label{eq:s_2}
s_2 =s_2(\bar t_1):= \frac{P_*(\bar t_1) + \chi_2 \bar t_1}{\chi_2+ \log \theta} 
 \, ,
\quad \mbox{ where } \chi_2 =\chi_2(\bar t_1):= \lim_{s \to \bar t_1^-} \frac{P_*(s) - P_*(\bar t_1)}{\bar t_1-s} \ge \log \Lambda \, ,
\end{equation}
where (by convexity of $P_*(t)$)  the limit defining $\chi_2$ exists and $P_*(t)$ lies above its tangents,
so that $\theta^t < e^{P_*(t)}$ for all $t < s_2$.

  Our next lemma
is an analogue of Lemma~\ref{lem:low t>1},  interpolating now from $\bar t_1$ to $s_2$.

\begin{lemma}
\label{lem:low t>s_1}
Fix $t_0\in (0,1)$  and $t_1 \in (1, t_*)$,  and let $\bar t_1 \in (1, s_1)$ and $s_2(\bar t_1)>s_1$
be as above.  For any  $\bar t_2 \in (s_1, s_2)$ and any $\kappa >0$, there exist
$C_\kappa>0$,  $\eta_\kappa >0$ such that 
for all $g \in C^0$,
all $\delta>0$, all $W \in \hW^s$  with  $|W| \ge \delta/3$, and all $n \ge 1$,
\begin{equation}
\label{eq:low s_1}
\sum_{W_i \in \cG_n^\delta(W)} |J_{W_i}T^n|^t_{C^0(W_i)} |e^{ S_ng}|_{C^0(W_i)}
\ge C_\kappa \delta^{\frac{1}{\eta_\kappa}-1} e^{-n(\chi_2 + \kappa)(t-\bar t_1) + n P_*(\bar t_1) - n  |g|_{C^0}} \, ,\,
\, \forall t \in [\bar t_1,\bar t_2] \, .
\end{equation}
Moreover, if $|W| \ge \delta_0/3$, then 
\[
\sum_{W_i \in \cG_n^\delta(W)} |J_{W_i}T^n|^t_{C^0(W_i)} |e^{ S_ng}|_{C^0(W_i)}
\ge C_\kappa  \delta^{- 1}  e^{-n(\chi_2 + \kappa)(t-1) - n |g|_{C^0}} \, ,\,
\, \forall n \ge 1 \,,\,\,\forall t \in [\bar t_1, \bar t_2] \, .
\]
\end{lemma}

\begin{proof}
We adapt the proof of Lemma~\ref{lem:low t>1}.  First assume $g=0$.
For $t \ge s_1$, let $s \in (1, \bar t_1)$,  and $\eta(s)\in (0,1]$ such that 
$\eta t + (1-\eta)s = \bar t_1$.  Then again using the H\"older inequality,
for any $W \in \hW^s$ with  $|W| \ge \delta/3$
and all $n \ge 1$,
\begin{equation}
\label{eq:interpol s_2}
\sum_{W_i \in \cG_n^\delta(W)} |J_{W_i}T^n|^t
\ge \frac{\left( \sum_{W_i \in \cG_n^\delta(W)} |J_{W_i}T^n|^{\bar t_1} \right)^{1/\eta}}
{\left( \sum_{W_i \in \cG_n^\delta(W)} |J_{W_i}T^n|^s \right)^{(1-\eta)/\eta}}
\ge \frac{\big( c_1 e^{n P_*(\bar t_1)} \big)^{1/\eta}}{\big( C_2[0]  \frac{\delta_0}{\delta} \frac{2}{c_2}
e^{n P_*(s)} \big)^{(1-\eta)/\eta} }   \, ,
\end{equation}
where we have used Propositions~\ref{prop:lower bound} and \ref{prop:exact} 
for the lower bound in the numerator, and
Lemma~\ref{lem:extra growth} with $\varsigma=0$
and Proposition~\ref{prop:exact} for the upper bound in the denominator.
Since $\eta = (\bar t_1-s)/(t-s)$, 
\[
e^{-n(P_*(s) - P_*(\bar t_1))/\eta} e^{n P_*(s)} 
= e^{-n(t-s) \frac{P_*(s) - P_*(\bar t_1)}{\bar t_1 - s}} e^{n P_*(s)} \, .
\]
For fixed $\kappa>0$, by \eqref{eq:s_2}, we may
choose $s=s(\kappa)\in (1, \bar t_1)$ and $\eta_\kappa > 0$
such that $(t-s)\frac{P_*(s) - P_*(\bar t_1)}{\bar t_1 - s} \le (t- \bar t_1) (\chi_2 + \kappa)$, 
completing the proof
for $g=0$ since
$P_*(s) \ge P_*(\bar t_1)$. 
For $g \ne 0$, the lemma follows, again using the bound $|e^{S_ng}|_{C^0(W_i)} \ge e^{-n|g|_{C^0}}$.
\end{proof}

By 
definition, $\theta^t e^{\chi_2 (t- \bar t_1) - P_*(\bar t_1) } < 1$ if $t < s_2$.  Thus for $\bar t_2\in (s_1,s_2)$,
there exists $\kappa_2 = \kappa(\bar t_2) >0$  such that
\begin{equation}
\label{kappa2 choice}
\theta^{\bar t_2} e^{(\chi_2 + \kappa_2)(\bar t_2- \bar t_1) - P_*(\bar t_1) }   < 1\, ,\,\,
\mbox{ and thus }\, 
\theta^t e^{(\chi_2 + \kappa_2)(t- \bar t_1) - P_*(\bar t_1)} < 1\, ,\,\,
\forall t\le \bar t_2\, . 
\end{equation}
Our next lemma extends Lemma~\ref{lem:short small t>1} for $t \in [\bar t_1, \bar t_2]$.

\begin{lemma}
\label{lem:short small t_2}
Let $t_0\in (0,1)$  and $t_1 \in (1, t_*)$.
Let $\bar t_2 \in (s_1, s_2)$, and let $\kappa_2=\kappa(\bar t_2)$ satisfy \eqref{kappa2 choice}.
Then for any $\ve>0$ there exist $\delta_1>0$ and $n_1 \ge 1$, 
such that for all $W \in \hW^s$ with $|W| \ge  \delta_1/3$, and for all $n \ge  n_1$, 
 \[
  \sum_{\substack{W_i \in \cG_n^{ \delta_1}(W) \\ |W_i| <  \delta_1/3}}
|J_{W_i}T^n|^t_{C^0(W_i)} |e^{ S_ng}|_{C^0(W_i)}
\le \ve \sum_{W_i \in \cG_n^{ \delta_1}(W) }
|J_{W_i}T^n|^t_{C^0(W_i)} |e^{ S_ng}|_{C^0(W_i)}
\, , \, \forall  t \in [1, \bar t_2 ]\, ,
  \]
for all $g\in C^1$
satisfying \eqref{eq:up one}, \eqref{eq:up two} and such that, in addition,
\begin{equation}
\label{eq:up three}
2|g|_{C^0}
< - \bar t_2 \log \theta -(\chi_2+\kappa_2) (\bar t_2  - \bar t_1) + P_*(\bar t_1) \,, \,
\mbox{ i.e. }\,\,
\theta^{\bar t_2} e^{(\chi_2+ \kappa_2)(\bar t_2 - \bar t_1) - P_*(\bar t_1) + 2 | g|_{C^0}} < 1 \, .
\end{equation}
\end{lemma}

\begin{proof}
The proof of Lemma~\ref{lem:short small t_2} proceeds with the analogous modifications
used in the proof of Lemma~\ref{lem:short small t>1}, using the lower bound
\eqref{eq:low s_1} in place of \eqref{eq:low t>1 up}.  The proof goes through due
to the contraction provided by \eqref{kappa2 choice} and \eqref{eq:up three}.
\end{proof}

\medskip
Let $[t_0, \bar t_2]\subset (0,s_2)$.
For all $g\in C^1$ satisfying \eqref{eq:up one}, \eqref{eq:up two}, \eqref{eq:up three},
Lemmas~\ref{lem:short small},  
\ref{lem:short small t>1} and \ref{lem:short small t_2} for $\ve = 1/4$ give $n_1\ge 1$  and
$\delta_1>0$  such that
for  all  $n\ge n_1$ and all  $W \in \hW^s$ with $|W| \ge \delta_1/3$,
\begin{equation}
\label{eq:delta111}
\sum_{W_i \in L_n^{\delta_1}(W)} |J_{W_i}T^n|^t_{C^0(W_i)} 
|e^{S_ng}|_{C^0(W_i)}
\ge \tfrac 34 \sum_{W_i \in \cG_n^{\delta_1}(W)} |J_{W_i}T^n|^t _{C^0(W_i)}|e^{S_ng}|_{C^0(W_i)} \, ,
\, \forall t \in [t_0, \bar t_2] \, .
\end{equation}

At this point it is clear that Lemma~\ref{lem:most long t>1} (with the same constant $\delta_2>0$,
but possibly smaller $c_0>0$), and Propositions~\ref{prop:lower bound} and \ref{prop:exact}
(with possibly smaller constants $c_1, c_2 >0$)
hold with $\bar t_1$ replaced by $\bar t_2 \in (s_1, s_2)$.

The interpolation can now be continued inductively.  Suppose we have created a sequence
$1 < \bar t_1 < s_1 < \bar t_2 < s_2 < \ldots < \bar t_n < s_n < t_1 < t_*$ so that 
Propositions~\ref{prop:lower bound} and \ref{prop:exact} hold with $\bar t_1$ replaced by
$\bar t_n$.  Then since $s_n < t_1$, we have $\theta^{s_n} < e^{P_*(s_n)}$ and we
may define
\[
\chi_{n+1} = \lim_{s \to \bar t_n^-} \frac{P_*(s) - P_*(\bar t_n)}{\bar t_n - s} \ge \log \Lambda,
\quad  \mbox{and} \quad
s_{n+1}=  \frac{P_*(\bar t_n) + \chi_n \bar t_n}{\chi_n + \log \theta} > s_n \, ,
\]
where $s_{n+1} > s_n$ by choice of $\bar t_n$.  
Following the proof of Lemma~\ref{lem:low t>s_1} with $\bar t_1, \bar t_2, \chi_2$ replaced by 
$\bar t_n, \bar t_{n+1}, \chi_{n+1}$, it follows that the conclusion of the lemma holds for
all $t \in [\bar t_n, \bar t_{n+1}]$.  Analogous modifications to Lemma~\ref{lem:short small t_2}
imply that Lemma~\ref{lem:most long t>1} and the propositions of Section~\ref{sec:complexity}
hold with $\bar t_1$ replaced by $\bar t_{n+1} \in (s_n, s_{n+1})$.

Finally, the sequence $(s_n)$ cannot accumulate on any $s_\infty \le t_1$.  For if it does, then
by definition of $\theta$, it follows that $\theta^{s_\infty} < e^{P_*(s_\infty)}$, so we may repeat the construction
above, finding a point of intersection $s'>s_\infty$ between $t \log \theta$ and the left
hand tangent to $P_*(t)$ at some $\bar t_n < s_\infty$.  It follows that this sequence of interpolations
can be chosen so that $t_1 < \bar t_n < t_*$ for some $n \ge 1$.  At this point we stop, and
since we have made only finitely many choices of the required constants, we have extended
the analogues of Propositions~\ref{prop:lower bound} and \ref{prop:exact} to all $t_1<t_*$:

\begin{proposition}
\label{prop:summary}
Let $t_0\in (0,1)$ and $t_1\in (1,t_*)$.
 There exist decreasing functions $c_i : [0, \infty) \to \mathbb{R}^+$, $i=1,2$,
such that for any $\upsilon \ge 0$ and 
 any $g\in C^1$ with $|\nabla g|_{C^0}\le \upsilon$ and 
such that $|g|_{C^0}$ is sufficiently small (depending on the number of interpolations 
required to reach  $t_1$), 

\begin{itemize}
  \item[a)] for any $W \in \hW^s$ with $|W| \ge \delta_1/3$,   
\[
\sum_{W_i \in \cG_n(W)} |J_{W_i}T^n|^t_{C^0(W_i)} |e^{ S_n g}|_{C^0(W_i)}
\ge  c_1(\upsilon) Q_n(t, g) \, , \, \forall n \ge 1 \, , \,\, \forall t \in [t_0, t_1 ]\, ;
\]
  \item[b)]  for all $n \ge1$, we have $\displaystyle
e^{n P_*(t, g)} \le Q_n(t, g) \le \frac{2}{c_2(\upsilon)} e^{n P_*(t, g)}\, ,\,
\forall t \in [t_0, t_1 ]\, . $
\end{itemize}
\end{proposition}


\section{Spectral Properties of $\cL_t$ (Theorem~\ref{thm:spectral})}
\label{sec:spec}

\subsection{Definition of Norms and Spaces $\cB$ and $\cB_w$}
\label{sec:norms}

For fixed $t_0 > 0$  and $t_1 \in(\max\{t_0,1\}, t_*)$,  we choose  $\theta(t_1) \in (\Lambda^{-1}, \Lambda^{-1/2})$
satisfying $\theta^{t_1} < e^{P_*(t_1)}$,
$q > \min \{1,2/t_0\}$, $k_0=k_0(t_0,  t_1)$ (for the homogeneity strips \eqref{eq:H}), 
and $\delta_0=\delta_0(t_0,  t_1)$ from Definition~\ref{def:k0d0}.
These choices affect the definitions of $\cW^s$ and $\cW^s_{\bH}$, as well as  conditions \eqref{eq:all} and \eqref{eq:Cve} below
on the parameters $\alpha$, $\beta$, $\gamma$, $p$,  $\ve_0$, determining spaces 
$\cB=\cB(t_0,  t_1)$ and $\cB_w=\cB_w(t_0,  t_1)$ on which 
$\cL_t$ will be bounded  for all
$t \ge t_0$. 
An additional condition on the parameter $p$ depending on $t_1 < t_*$ will be needed to obtain
the Lasota--Yorke bound 
\eqref{eq:LY}  (see Lemma~ \ref{lem:t ok}) and thus the spectral gap of $\cL_t$ on  
 $\cB$ for all $t\in [t_0,t_1]$.

\medskip

First we define notions of distance\footnote{The triangle inequality is not satisfied,
but this is of no consequence for our purposes.} between stable curves and test functions as follows.

Since the slopes of stable curves are uniformly bounded away from the vertical, we view each $W \in \hW^s$ as the
graph of a function of the $r$-coordinate over an interval $I_W$,
\[
W = \{ G_W(r) : r \in I_W \} = \{ (r, \vf_W(r)) : r \in I_W \} \, .
\]
By the uniform bound on the curvature of $W \in \hW^s$, we have $B:=\sup_{W \in \hW^s} | \vf_W''| <\infty$.

Next, given $W_1, W_2 \in \hW^s$ with functions $\vf_{W_1}$, $\vf_{W_2}$, we define 
\[
d_{\cW^s}(W_1, W_2) = | I_{W_1} \bigtriangleup I_{W_2}| + |\vf_{W_1} - \vf_{W_2}|_{C^1(I_{W_1} \cap I_{W_2})} \, ,
\]
if $W_1$ and $W_2$ lie in the same homogeneity strip, and $d_{\cW^s}(W_1, W_2) = 3B + 1$ otherwise.

Finally, if $d_{\cW^s}(W_1, W_2) < 3B+1$, then for $\psi_1 \in C^0(W_1)$, $\psi_2 \in C^0(W_2)$, define
\[
d(\psi_1, \psi_2) = |\psi_1 \circ G_{W_1} - \psi_2 \circ G_{W_2} |_{C^0(I_{W_1} \cap I_{W_2})} \, ,
\]
while if $d_{\cW^s}(W_1, W_2) \ge 3B+1$
and  $\psi_1 \in C^0(W_1)$, $\psi_2 \in C^0(W_2)$, we set
$
d(\psi_1, \psi_2)  =\infty$.

\medskip

We next define the norms, introducing parameters $\alpha$, $\beta$, $\gamma$, $p$, and $\ve_0$.
Choose\footnote{The condition $\gamma \le \frac{1}{6q+7}$ is used in  Lemma~\ref{lem:image}.}
\begin{equation}
\label{eq:all}
\alpha \in \Big( 0,  \frac{1}{q+1} \Big] \, , \,\, \, \,\,
p > q+1, \, \, \, \,\, \beta \in \Big( \frac{1}{p}, \alpha \Big),\,\, \,\,\,
 \gamma \in \Big( 0,\min \Big\{ \frac{1}{p}, \alpha - \beta, \frac{1}{6q+7}\Big\} \Big)\, .
\end{equation}
(This implies $\alpha \le 1/3$, $\gamma <1/p$, and
$\min(\beta, t)> \frac{1}{p}$.)
Finally for  $C_{vert}<\infty$ to be determined in \eqref{eq:Cv}, let $\ve_0$ satisfy
\begin{equation}
\label{eq:Cve}
0< C_{vert} \,  \ve_0^{1/(q+1)} \le \frac 34 \, . 
\end{equation} 

\smallskip

For $f \in C^1(M)$, recalling $C^\eta(W)$   and $\cW^s_H$ from Section~\ref{defstuff}, define the weak norm of $f$ by\footnote{
Using weakly homogeneous curves implies that Lebesgue
measure belongs to $\cB$, see Remark~\ref{LebB}.}
\[
|f|_w = \sup_{W \in \cW^s_H} \sup_{|\psi|_{C^\alpha(W)} \le 1} \int_W f \, \psi \, dm_W \, ,
\]
define the stable norm of $f$ by\footnote{The weight $|W|^{-1/p}$ in \eqref{stab} is used both in the proof of
compactness (Proposition~\ref{embeds}) and to
control the estimate over unmatched pieces in the Lasota-Yorke inequality (Section~\ref{sec:unstable}).}
\begin{equation}\label{stab}
\|f\|_s = \sup_{W \in \cW^s_H} \sup_{|\psi|_{C^\beta(W)} \le |W|^{-1/p}} \int_W f \, \psi \, dm_W \, ,
\end{equation}
and the unstable norm of $f$ by
\[
\| f \|_u = \sup_{\ve \le \ve_0} \sup_{\substack{ W_1, W_2 \in \cW^s_H \\ d_{\cW^s}(W_1, W_2) \le \ve}}
\sup_{\substack{|\psi_i|_{C^\alpha(W_i)} \le 1 \\ d(\psi_1, \psi_2)=0}} \ve^{-\gamma}
\left| \int_{W_1} f\, \psi_1 \, dm_{W_1} - \int_{W_2} f \, \psi_2 \, dm_{W_2} \right| \, .
\]
Finally, define the strong norm of $f$ to be 
$$\| f \|_{\cB} = \| f \|_s + c_u \| f \|_u
$$ 
for a constant $c_u =c_u(\beta,\gamma, p)> 0$
(so that $c_u$ depends on $[t_0,t_1]$) to be chosen
in \eqref{choicec_u}.
Define $\cB$ to be the completion of $C^1(M)$ in the 
$\| \cdot \|_{\cB}$ norm, and $\cB_w$ to be the completion of 
$C^1(M)$ in the $|\cdot |_w$ norm. 


\subsection{Statement of the Spectral Result. Embeddings}

The  equilibrium measure in Theorem~\ref{thm:equil} 
and its properties will be obtained by letting the transfer operator
$\cL_t$ act on $\cB$.

\begin{theorem}[Spectrum of $\cL_t$ on $\cB$]
\label{thm:spectral}
For each $t_0\in (0,1)$ and $t_1\in (1,t_*)$  there exists a
Banach space $\cB= \cB(t_0,t_1)$  such that 
for each $t \in [t_0,t_1]$, the operator $\cL_t$ is bounded on $\cB$ 
with spectral radius equal to $e^{P_*(t)}$ and,
recalling $\theta(t_1)$ from Lemma~\ref{lem:one step},
  essential spectral radius not larger than
$$\max \{ \Lambda^{(-\beta + 1/p)}, \theta^{(t-1/p)}e^{- P_*(t)}, \Lambda^{-\gamma} \} e^{P_*(t)} < e^{P_*(t)}
\, .
$$
Moreover $\cL_t$ has a spectral gap: the only eigenvalue 
of modulus $e^{P_*(t)}$ is $e^{P_*(t)}$ and it is simple.

Let $\nu_t$ denote the unique element of $\cB$ 
with $\nu_t(1)=1$ satisfying $\cL_t \nu_t = e^{P_*(t)} \nu_t$, and let $\tilde \nu_t$ denote the
maximal eigenvector for the dual, $\cL_t^* \tilde \nu_t = e^{P_*(t)} \tilde \nu_t$.  
Then the distribution $\mu_t$ defined by 
$\mu_t(\psi) = \frac{\tilde \nu_t( \psi \nu_t)}{\tilde \nu_t(\nu_t)}$ is in fact  
a $T$-invariant probability measure.
This measure is mixing, correlations for $C^\alpha$ observables decay exponentially  
with rate $\upsilon$ for any
\begin{equation}\label{uups}
\upsilon>\upsilon_0(t) := \sup  \{ |\lambda|\mid \lambda \in \mbox{sp}(e^{-P_*(t)}\cL_t) \setminus \{1\}\}
\, ,
\end{equation}
and correlations for H\"older observables of arbitrary exponent decay exponentially.
\end{theorem}

Recall that $\theta<1/\sqrt \Lambda<1$. Note that since $qt_0>1$ while $\beta <1/(q+1)$ and $\gamma\le \min\{ 1/(q+1),1/(q+1)-\beta\}$, our bound on the essential spectral radius tends
to $e^{P_*(t_0)}$ as $t_0\to 0$.
Similarly, as $t_1\to t_*$ we need
to let $p\to \infty$ to ensure $\theta^{ t_1 -1/p}<e^{P_*(t_1)}$
(see Lemma~ \ref{lem:t ok}) and our bound on the  essential spectral radius tends
to $e^{P_*(t_*)}$ as $t_1\to t_*$.

\smallskip

As usual, Hennion's theorem is the key  to proving
the above theorem. It requires two ingredients: the compact embedding 
proposition below and the
Lasota--Yorke estimates in Proposition~\ref{prop:ly}.
 
\begin{proposition}[Embeddings]
\label{embeds}
For any $t_0\in (0,1)$ and $t_1 \in (1,t_*)$, the   continuous inclusions  
$$C^1(M)  
\subset \BB \subset \BB_w  \subset (C^1(M))^* \, $$
hold, so that
$C^1(M)\subset (\BB_w)^*\subset \BB^* \subset (C^1(M))^*$.
In addition,  the inclusions $C^1(M)  
\subset \BB$ and $\BB\subset \BB_w$ are injective, and  the embedding of the unit ball of $\cB$ 
in $\cB_w$ is compact.
\end{proposition}

The  embedding $\BB_w  \subset (C^1(M))^*$
is understood in the following sense: For $f \in C^1(M)$, 
we identify $f$ with the measure $f d\musrb\in  (C^1(M))^*$. Then,
for  $f \in \cB_w$ there exists $C_f<\infty$ such that, letting $f_n\in C^1(M)$ be a sequence 
 converging to $f$ in the $\cB_w$ norm,  
 for every $\psi\in C^1(M)$ the limit 
 $
 f(\psi):=\lim_{n \to \infty} \int f_n \psi \, d\musrb
$
exists  and satisfies $|f(\psi)|\le C_f |\psi |_{C^1(M)}$.
See Lemma~\ref{lem:distrib} for a strengthening of this embedding.

\begin{proof}[Proof of Proposition~\ref{embeds}]
The proof of the  claims in the first sentence is the same as the
proof of \cite[Prop. 4.2, Lemma 4.4]{max}. 
The injectivity of the first inclusion is obvious, while the injectivity of the second 
follows from our definition of $C^\beta(W)$:  
if $|f|_w = 0$ then $\| f\|_u=0$ since the class of test functions is the same, but
also $\| f \|_s =0$ since $C^1(W)$ is dense in $C^\beta(W)$, proving injectivity. 
The
proof of the compact embedding  follows exactly  the lines of that of \cite[Prop. 6.1]{max}, using $\hW^s$. The only differences are  that, in the unstable norm,
$|\psi|_{C^\beta(W)} \le|\log  |W||^{\tilde \gamma}$ there is replaced by
 $|\psi|_{C^\beta(W)} \le |W|^{-1/p}$, while
the logarithmic modulus of continuity
$|\log \epsilon|^{-\zeta}$ there is replaced
by a H\"older modulus of continuity $\epsilon ^\gamma$.
\end{proof}

To show that the transfer operator $\cL_t$ is bounded
on $\cB$, we require the following lemma.

\begin{lemma}
\label{lem:image}
For any $f \in C^1(M)$ and any $t  \ge  t_0$, the image 
$\cL_t f$ belongs to the closure of $C^1(M)$ in the strong norm $\| \cdot \|_{\cB}$,
for $\cB=\cB(t_0,  t_1)$.
\end{lemma}

We prove Lemma~\ref{lem:image} in Section~\ref{sec:approach}.

\begin{remark}[Lemma 4.9 in \cite{max}]\label{omit}
We remark that the proof of \cite[Lemma~4.9]{max}, which is the analogue of the
present Lemma~\ref{lem:image}, was omitted there.  The reference given there to 
\cite[Lemma~3.8]{dz1} is not correct since $J^sT$ is not piecewise H\"older.  However, its statement
is correct as the  proof  in Section~\ref{sec:approach} 
and Remark~\ref{t=0} demonstrate.
\end{remark}

\begin{remark}[Lebesgue measure belongs to $\cB$]
\label{LebB}
Since we identify $f \in C^1(M)$ with the measure
$f d\musrb$, Lebesgue measure is identified with the function
$f = 1/\cos \vf$, which is not in $C^1(M)$.  However, it follows from \cite[Lemma 3.5]{dz2},
that $1/\cos\vf$ can be approximated by $C^1$ functions in the $\cB$ norm, so that 
Lebesgue measure belongs to $\cB$.  (The proof requires that
our norms integrate on weakly homogeneous stable manifolds, rather than on all
$W \in \cW^s$ as was done in  \cite{max}.)
\end{remark}


\subsection{Lasota--Yorke Inequalities}

Using the exact bounds for $Q_n(t)$ from Proposition~ \ref{prop:exact},
we prove the following proposition (under more general conditions than 
Theorem~\ref{thm:spectral}).

\begin{proposition}
\label{prop:ly}
Fix $t_0\in (0,1)$ and $t_1\in (1, t_*)$ and let $\cB = \cB(t_0, t_1)$, $\theta=\theta (t_1)$.
Fix $ t_2 \in(t_0, \infty)$. 
For any $t\in [t_0, t_2]$, the operator $\cL_t$ extends continuously to $\cB_w$ and $\cB$, and setting $C_n = \max_{1 \le j \le n} Q_j(t)$,
there exists $C = C(t_0,  t_2) <\infty$ such that for every $n \ge 0$, 
\begin{align}
\label{eq:weak}
| \cL_t^n f |_w   &\le   C Q_n(t) |f|_w  \, ,\,  \,  \forall f \in \cB_w\, , \\
\label{eq:stable} \| \cL_t^n f \|_s  &\le   C Q_n(t) \left[ (\Lambda^{(-\beta + 1/p)n} + \theta^{(t-1/p)n} Q_n(t)^{-1} ) \right] \| f \|_s +  C C_n  |f|_w \, , \, \forall \, f \in \cB  \, ,  \\
\label{eq:unstable} \| \cL_t^n f \|_u   &\le   C Q_n(t) \left[ n \Lambda^{-\gamma n} \| f\|_u + Q_n(t-1/p) Q_n(t)^{-1} \| f\|_s \right] \, , \,  \forall f \in \cB \, . 
\end{align}
Moreover,  if   $t_2 = t_1$,  then, up to taking 
$p$ large enough, for any 
$$ 
\sigma \in ( \max \{ \Lambda^{-\beta + 1/p}, \Lambda^{-\gamma}, \theta^{t - 1/p} e^{- P_*(t)} \}, 1)\, ,
$$  
there exists $c_u=c_u(t_0,t_1) > 0$,  and $\bar C_n >0$,
such that, for  all $f \in \cB$, 
\begin{equation}
\label{eq:LY}
 \| \cL_t^n f \|_{\cB} 
\le  C e^{P_*(t) n}   \left [\sigma^n \| f \|_{\cB} +  \bar C_n |f|_w \right ]  \, , \, \forall \, n  \ge 1  \, . 
\end{equation}
\end{proposition}

\smallskip
\noindent
Proving \eqref{eq:LY} will use the following lemma:

\begin{lemma}
\label{lem:t ok}
For any  $t_1\in (1,t_*)$
 there exists   $p> 1$ such that
$\theta^{t-1/p} < e^{P_*(t)}$ for
all $t \in (1/p, t_1]$.
\end{lemma}

\begin{proof}
If $t\in(0, 1]$ then $P_*(t) \ge 0$ so that $\theta^{t-1/p} < 1 \le e^{P_*(t)}$
for  all  $\theta<1$, all $p>1$ and all $t\in(1/p,1]$.
For $t \in (1, t_1]$,  since the slopes of $P_*(t)$ are at most
$- \log \Lambda$ by the proof of Proposition~\ref{prop:Ptilde}, we have $\frac{P_*(t_1)- P_*(t)}{t_1-t} <-\log \Lambda< \log \theta$, so that $\theta^t e^{- P_*(t)} < \theta^{t_1} e^{-P_*(t_1)}$.  
The choice of $\theta=\theta(t_1)$ in   Definition~\ref{def:k0d0}   gives $\theta^{t_1} e^{-P_*(t_1)}<1$.
 Choosing $p>1$ such that
$\theta^{t_1- 1/p} e^{-P_*(t_1)} < 1$ ends the  proof.
\end{proof}

\begin{proof}[Proof of Proposition~\ref{prop:ly}]
We first show that \eqref{eq:weak},  \eqref{eq:stable},
and  \eqref{eq:unstable} imply that 
if  $ t_2 =  t_1 < t_*$ and $p$ is large enough,
then  $\cL_t$ satisfies the  Lasota--Yorke inequality
\eqref{eq:LY} for $f\in \cB(t_0,t_1)$:
Choosing $p$ according to Lemma~\ref{lem:t ok}, observe that $\theta^{t-1/p} e^{- P_*(t)} < 1$
 implies $\theta^{t-1/p} Q_n(t)^{-1} \le \theta^{(t-1/p) n} e^{-P_*(t) n}< 1$ for all $n\ge 1$, since $Q_n(t) \ge e^{P_*(t) n}$  by 
Proposition~\ref{prop:exact}.
Next, recalling that $P_*(t)$ is  convex and  strictly decreasing by Proposition~\ref{prop:Ptilde}, and  fixing 
$$\ve_1:=P_*(t-1/p) - P_*(t) \in (0,  P_*(t_0-1/p)- P_*(t_0)  ) \, ,$$
we find, using  both the lower and upper bounds  from
Proposition~\ref{prop:summary}(b),
\[
Q_n(t-1/p) Q_n(t)^{-1} \le \frac{2}{c_2} e^{P_*(t-1/p) n} e^{-P_*(t)n} 
\le \frac{2}{c_2}  e^{\ve_1 n}\, ,
\,\forall n \ge 1 
\, .  
\]
Next,  fix $  1 > \sigma >  \max \{ \Lambda^{-\beta + 1/p}, \Lambda^{-\gamma}, \theta^{t - 1/p} e^{- P_*(t)} \}$ and choose  $N \ge 1$ such that
\[
\frac{2 C }{c_2} \max \{ N \Lambda^{-\gamma N} , 2\big( \Lambda^{-(\beta - 1/p)N} + \theta^{(t-1/p)N} e^{-P_*(t)N}\big)  \}
\le \sigma^N \, . 
\]
Choosing $c_u>0$ to satisfy
\begin{equation}
\label{choicec_u}
c_u \le 
\frac{c_2^2 \sigma^N}{8 C e^{2(P_*(t_0-1/p) -  P_*(t_0)  )N} }\, ,
\end{equation}
 we estimate, using once more the upper bound for $Q_n(t)$ from
Proposition~\ref{prop:summary}(b),
\begin{align*}
\| \cL_t^N f \|_{\cB} & = \| \cL_t^N f \|_s + c_u \| \cL_t^N f \|_u \\
& \le e^{P_*(t)N} \left[ \frac{\sigma^N}{2} \| f \|_s +c_u \sigma^N \| f \|_u +  \frac{4 c_u}{c_2^2}  e^{\ve_1 N} \| f \|_s \right]  +  C C_N   | f |_w \\
& \le  e^{P_*(t)N}  
\left[ \sigma^N \| f \|_{\cB} + e^{-P_*(t)N}  C C_{N}   | f |_w \right]\, .
\end{align*}
Iterating this equation and using the first claim of \eqref{eq:weak} (recalling one more time  the upper bound for $Q_n(t)$ from
Proposition~\ref{prop:summary}(b)) yields \eqref{eq:LY} for $n=\ell N$, with $\ell \ge 1$. 
The general case  follows since \eqref{eq:stable} and \eqref{eq:unstable} imply
$\| \cL_t^k f\|_{\cB} \le \bar C \| f\|_{\cB}$ for $k \le N$.

By Lemma~\ref{lem:image}, it suffices
to prove the bounds
\eqref{eq:weak},  \eqref{eq:stable},
and  \eqref{eq:unstable}  for $f \in C^1(M)$,
and they also imply that $\cL_t$ extends to a bounded operator on $\cB$ and $\cB_w$.
This is similar  to the proof of \cite[Proposition~2.3]{dz1}
and  is the content of Sections~\ref{sec:weak norm}--\ref{sec:unstable}.
\end{proof}


\subsubsection{Proof of Weak Norm Bound \eqref{eq:weak}}
\label{sec:weak norm}

Let $f \in C^1(M)$, $W \in \cW^s$ and $\psi \in C^\alpha(W)$ such that $|\psi|_{C^\alpha(W)} \le 1$.  Then for $n \ge 0$,
we have
\begin{equation}
\label{eq:weak start}
\begin{split}
\int_W \cL_t^n f \, \psi \, dm_W & = \sum_{W_i \in \cG_n(W)} \int_{W_i} f \psi \circ T^n |J^sT^n|^t \, dm_{W_i} \\
& \le \sum_{W_i \in \cG_n(W)} |f|_w |\psi \circ T^n|_{C^\alpha(W)} ||J^sT^n|^t|_{C^\alpha(W)} \, .
\end{split}
\end{equation}
The contraction along stable manifolds implies for $x, y \in W_i \in \cG_n(W)$, recalling \eqref{eq:holder def},
\begin{equation}
\label{eq:test contract}
|\psi(T^nx) - \psi(T^ny)| \le H^\alpha_W(\psi) d(T^nx, T^ny)^\alpha \le H^\alpha_W(\psi) |J^sT^n|_{C^0(W_i)}^\alpha d(x,y)^\alpha \, .
\end{equation} 
This implies  $H^\alpha_{W_i}(\psi \circ T^n) \le |J^sT^n|_{C^0(W_i)}^\alpha H^\alpha_W(\psi)$ and 
$|\psi \circ T^n|_{C^\alpha(W_i)} \le C_1^{-1} |\psi|_{C^\alpha(W)}$, with
$C_1$ from \eqref{eq:hyp}.

Moreover, since $\alpha \le 1/(q+1)$, the distortion bound of Lemma~\ref{lem:distortion} implies
\begin{equation}
\label{eq:JC}
||J^sT^n|^t|_{C^\alpha(W_i)} \le (1+ 2^tC_d) |J^sT^n|^t_{C^0(W_i)} \, , \, \,
 \forall\,  W_i \in \cG_n(W)\, . 
\end{equation}

Using \eqref{eq:test contract} and \eqref{eq:JC} in \eqref{eq:weak start}, we obtain,
\[
\int_W \cL^n_t f \, \psi \, dm_W \le \sum_{W_i \in \cG_n(W)} |f|_w C_1^{-1} (1+2^tC_d) |J^sT^n|^t_{C^0(W_i)}
\le C |f|_w Q_n(t) \, ,
\]
where in the last inequality, we have used Lemma~\ref{lem:extra growth} with $\varsigma = 0$.  Taking the suprema over
$\psi \in C^\alpha(W)$ with $|\psi|_{C^\alpha(W)} \le 1$ and $W \in \cW^s$  yields 
\eqref{eq:weak}.


\subsubsection{Proof of Stable Norm Bound \eqref{eq:stable}}
\label{sec:stable norm}

Let $f \in C^1(M)$, $W \in \cW^s$, and $\psi \in C^\beta(W)$ be such that $|\psi|_{C^\beta(W)} \le |W|^{-1/p}$.  For
$n \ge 0$ and
$W_i \in \cG_n(W)$, we define the average $\bpsi_i = |W_i|^{-1} \int_{W_i} \psi \circ T^n \, dm_{W_i}$.  Then as in
\eqref{eq:weak start}, we write,
\begin{align}
\label{eq:stable split}
\int_W \cL_t^n f \, \psi \, dm_W &= \sum_{W_i \in \cG_n(W)} \int_{W_i} f \, (\psi \circ T^n - \bpsi_i) |J^sT^n|^t \, dm_{W_i}\\
\nonumber &\qquad\qquad\qquad\qquad + \sum_{W_i \in \cG_n(W)} \bpsi_i \int_{W_i} f \, |J^sT^n|^t \, dm_{W_i} \, .
\end{align}
Note that by \eqref{eq:test contract}, 
\[
|\psi \circ T^n - \bpsi_i|_{C^\beta(W_i)} \le 2 |J^sT^n|^\beta_{C^0(W_i)} |\psi|_{C^\beta(W)}
\le 2 |J^sT^n|^\beta_{C^0(W_i)} |W|^{-1/p} \, .
\]
Therefore, replacing $\alpha$ by $\beta$ in  \eqref{eq:JC}, the definition of the strong stable norm gives
\begin{equation}
\label{eq:first stable}
\begin{split}
\sum_{W_i \in \cG_n(W)} \int_{W_i} f \, (\psi \circ T^n &- \bpsi_i) |J^sT^n|^t \, dm_{W_i}\\
& \le \sum_{W_i \in \cG_n(W)} 2(1+2^tC_d) \| f\|_s \frac{|W_i|^{1/p}}{|W|^{1/p}} |J^sT^n|^{t + \beta}_{C^0(W_i)} \\
& \le 2(1+2^tC_d) C_1^{-1} \Lambda^{-n(\beta - 1/p)} \| f \|_s C_2[0] \, Q_n(t) \, ,
\end{split}
\end{equation}
where in the second inequality we have used Lemma~\ref{lem:extra growth} with $\varsigma = 1/p$ (recall $\beta > 1/p$).

For the second sum in \eqref{eq:stable split}, note that $|\bpsi_i| \le |W|^{-1/p}$.  If $|W| \ge \delta_0/3$, then we simply estimate
\[
\sum_{W_i \in \cG_n(W)} \bpsi_i \int_{W_i} f \, |J^sT^n|^t \, dm_{W_i} \le \frac{3} {\delta_0^{1/p}} |f|_w (1+2^t C_d) \sum_{W_i \in \cG_n(W)}
|J^sT^n|^t_{C^0(W_i)} \le C |f|_w Q_n(t) \, ,
\] 
by Lemma~\ref{lem:extra growth} with $\varsigma=0$.

If $|W| < \delta_0/3$, we handle the estimate differently, splitting the sum
into two parts as follows.  We decompose the elements
of $\cG_n(W)$ by {\em first long ancestor} as follows:
Recalling the sets $\cI_n(W)$ defined 
in \S\ref{1step}, we call $V_j \in \cG_k(W)$ the first long ancestor of $W_i \in \cG_n(W)$ if 
\begin{equation}\label{FLA}
T^{n-k}W_i \subset V_j\, ,\quad
|V_j| \ge \delta_0/3\, , \quad \mbox{and $TV_j$ is contained in an element of } \cI_{k-1}(W)\, .
\end{equation}
We denote by $P_k(W)$ the set of such $V_j \in \cG_k(W)$ that are long for the first time at time $k$.  Note that $W_i$ has no long ancestor
if and only if $W_i \in \cI_n(W)$.

Grouping the terms in the second sum in \eqref{eq:stable split} by whether they belong to $\cI_n(W)$ or not, we apply the
weak norm to those elements that have a first long ancestor, and the strong stable norm to those that do not.  Thus,
\begin{equation}
\label{eq:stable short sum}
\begin{split}
& \sum_{W_i \in \cI_n(W)} \bpsi_i \int_{W_i} f \, |J^sT^n|^t \, dm_{W_i}
\le \sum_{W_i \in \cI_n(W)} |\bpsi_i| \| f\|_s |W_i|^{1/p} |J^sT^n|^t_{C^\beta(W_i)} \\
& \quad \le (1+2^t C_d) \| f\|_s \sum_{W_i \in \cI_n(W)} \frac{|W_i|^{1/p}}{|W|^{1/p}} |J^sT^n|^t_{C^0(W_i)}
\le (1+2^t C_d) \|f \|_s C_0\theta^{n(t-1/p)} \, ,
\end{split}
\end{equation}
where in the last estimate we applied Lemma~\ref{lem:short growth} with $\varsigma = 1/p$ since\footnote{This bound
holds since $p > q+1$ in the definition of the norms, yet
 $q \ge 2/t$ from \eqref{eq:H}.} $1/p \le \min \{1/2, t/2\}$.

For the terms that have a first long ancestor in $P_k(W)$, we again apply Lemma~\ref{lem:short growth}
from time $0$ (since $|W| < \delta_0/3$) to time $k$,
setting $\cG_0(V)=\{V\}$,
\begin{align*}
\sum_{k=1}^n & \sum_{V_j \in P_k(W)} \sum_{W_i \in \cG_{n-k}(V_j)} \bpsi_i \int_{W_i} f \, |J^sT^n|^t \, dm_{W_i} \\
& \le \sum_{k=1}^n \sum_{V_j \in P_k(W)} |f|_w |V_j|^{-1/p} (1+2^t C_d) \frac{|V_j|^{1/p}}{|W|^{1/p}} |J^sT^k|^t_{C^0(V_j)} 
\sum_{W_i \in \cG_{n-k}(V_j)}  |J^sT^{n-k}|^t_{C^0(W_i)} \\
& \le \sum_{k=1}^n \sum_{V_j \in P_k(W)} |f|_w 3 \delta_0^{-1/p} (1+2^t C_d) \frac{|V_j|^{1/p}}{|W|^{1/p}}|J^sT^k|^t_{C^0(V_j)} 
CC_2[0] \,  Q_{n-k}(t)  \\
& \le \sum_{k=1}^n |f|_w 3 \delta_0^{-1/p} (1+2^t C_d) CC_2[0] \,  C_0 \theta^{k(t-1/p)} Q_{n-k}(t) \, ,
\end{align*}
applying  Lemma~\ref{lem:extra growth} 
for $\varsigma=0$ in the second inequality and 
Lemma~\ref{lem:short growth} (for $\varsigma=1/p$) in the third.

Putting these estimates together with \eqref{eq:first stable} in \eqref{eq:stable split} yields\footnote{It is in fact possible to show $\max_{0\le j \le n} Q_{j}(t)\le \max \{ Q_n(t), Q_n(1)\}$,
but we shall not use this.}
\[
\int_W \cL_t^n f \, \psi \, dm_W \le  C Q_n(t) \big( \Lambda^{-n(\beta - 1/p)} + \theta^{n(t-1/p)} Q_n(t)^{-1} \big) \| f \|_s + C \max_{0\le j \le n} Q_{j}(t) |f|_w   \, ,
\]
and taking the appropriate suprema proves \eqref{eq:stable}
($C_n$ depends on $t$ only through $[t_0,t_1]$).


\subsubsection{Proof of Unstable Norm Bound \eqref{eq:unstable}}
\label{sec:unstable}

Let $\ve < \ve_0$ and let $W^1, W^2 \in \cW^s$ with $d_{\cW^s}(W^1, W^2) \le \ve$.  For $n \ge1$ and $\ell=1,2$, we partition
$T^{-n}W^\ell$ into matched pieces $U^\ell_j$ and unmatched pieces $V^\ell_i$ like
in \cite{dz1} as follows.

To each homogeneous connected component $V$
of $T^{-n}W^1$, we associate a family of vertical segments $\{ \gamma_x \}_{x \in V}$ of length at most 
$C_1^{-1} \Lambda^{-n}\ve$ such that if $\gamma_x$ is not cut by an element of $\cS_n^{\bH}$,  its image $T^n\gamma_x$
will have length $C\ve$ and will intersect $W^2$.
According to \cite[Sect.~4.4]{chernov book}, for such a segment, $T^i\gamma_x$ will be an unstable curve for $i=1, \ldots, n$ and so will remain uniformly transverse to the stable cone and
undergo the minimum expansion given by \eqref{eq:hyp}. 

Repeating this procedure for each connected component of $T^{-n}W^1$, we obtain a partition of $W^1$ into subintervals for
which $T^n\gamma_x$ is not cut and intersects $W^2$ and subintervals for which this is not the case.  This also defines
an analogous partition on $W^2$ and on the images $T^{-n}W^1$ and $T^{-n}W^2$.  We call two curves in
$T^{-n}W^1$ and $T^{-n}W^2$ {\em matched} if they are connected by  the foliation $\gamma_x$ and their images under
$T^n$ are connected by $T^n\gamma_x$.  We further subdivide the matched pieces if necessary to ensure that they 
have length $\le \delta_0$ and that they remain homogeneous stable curves.  Thus there are at most two matched pieces
$U^\ell_j$ corresponding to each element of $\cG_n(W^\ell)$.  The rest of the connected components of $T^{-n}W^\ell$ we call
{\em unmatched} and denote them by $V^\ell_i$.  Once again, there are at most two unmatched pieces $V^\ell_i$ 
corresponding to each element of $\cG_n(W^\ell)$.

Recalling the notation of Section~\ref{sec:norms}, we have constructed a pairing on matched pieces $U^\ell_j$ defined
over a common $r$-interval $I_j$ such that for each $j$,
\begin{equation}
\label{eq:match}
U^\ell_j = G_{U^\ell_j}(I_j) = \{ (r, \vf_{U^\ell_j}(r)) : r \in I_j \} \, , \; \; \; \ell = 1,2 .
\end{equation}

Now let $\psi_\ell \in C^\alpha(W^\ell)$ with $|\psi_\ell|_{C^\alpha(W^\ell)} \le 1$ and $d(\psi_1, \psi_2)=0$.  Decomposing
$W^1$ and $W^2$ into matched and unmatched pieces as above, we write,
\begin{align}
\label{eq:unstable decomp}
\left| \int_{W^1} \cL_t^n f \, \psi_1 - \int_{W^2} \cL_t^n f \, \psi_2 \right|
& \le \sum_j \left| \int_{U^1_j} f \, \psi_1 \circ T^n \, |J^sT^n|^t - \int_{U^2_j} f \, \psi_2 \circ T^n \, |J^sT^n|^t \right| \\
\nonumber & \qquad + \sum_{\ell, i} \left| \int_{V^\ell_i} f \, \psi_\ell \circ T^n \, |J^sT^n|^t \right| \, .
\end{align}

We estimate the unmatched pieces first.  For this we use the fact that unmatched pieces $V^\ell_i$ occur either
because $T^nV^\ell_i$ is near the endpoints of $W^\ell$ or because a vertical segment $T^n\gamma_x$ intersects
$\cS_{-n}^\bH$. In either case, due to the uniform transversality of the stable and unstable cones, we have
$|T^nV^\ell_i| \le C \ve$ for some uniform constant $C>0$, independent of $n$ and $W^\ell$, since $d_{\cW^s}(W^1,W^2) \le \ve$.
Thus, we estimate the sum over unmatched pieces using the strong stable norm,
\begin{equation}
\label{eq:unmatched}
\begin{split}
\sum_{\ell, i} \left| \int_{V^\ell_i} f \, \psi_\ell \circ T^n \, |J^sT^n|^t \right|
& \le \sum_{\ell,i} \| f\|_s |V^\ell_i|^{1/p} |\psi_\ell \circ T^n|_{C^\beta(V^\ell_i)} (1+2^t C_d) |J^sT^n|^t_{C^0(V^\ell)_i} \\
& \le \| f \|_s C_1^{-1} (1+2^t C_d) \sum_{\ell,i} |T^nV^\ell_i|^{1/p} |J^sT^n|^{t-1/p}_{C^0(V^\ell)_i} \\
& \le 4 C_2[0] \, C_1^{-1} (1+ 2^t C_d) \| f \|_s \ve^{1/p} Q_n(t-1/p) \, ,
\end{split}
\end{equation}
where $C_1$ is from \eqref{eq:hyp} and
we have used \eqref{eq:JC} in the first inequality, \eqref{eq:test contract}
 in the second, and Lemma~\ref{lem:extra growth} (for $\varsigma=0$)
in the third since there are at most two unmatched pieces corresponding to each element of $\cG_n(W^\ell)$.

To perform the estimate over matched pieces in \eqref{eq:unstable decomp},  
we  will change variables to define the relevant test functions on the same curve.  To this end,
recalling \eqref{eq:match}, define on each $U^1_j$,
\[
\tpsi_2 = \psi_2 \circ T^n \circ G_{U^2_j} \circ G_{U^1_j}^{-1}, \quad \mbox{and} \quad
\tJ^sT^n = J^sT^n \circ G_{U^2_j} \circ G_{U^1_j}^{-1} \, .
\]

\begin{sublem}
\label{lem:matching}
There exists $C>0$, independent of $t$, $n$, $W^1$, and $W^2$ such that
\begin{itemize}
  \item[a)] $\displaystyle d_{\cW^s}(U^1_j, U^2_j) \le C n \Lambda^{-n} \ve 
=: \ve_1$,
for all $j$;
  \item[b)] $\displaystyle |\psi_1 \circ T^n |J^sT^n|^t - \tpsi_2 |\tJ^sT^n|^t |_{C^\beta(U^1_j)} \le C 2^t |J^sT^n|^t_{C^0(U^1_j)} \ve^{\alpha-\beta}$, for all $j$.
\end{itemize}
\end{sublem}

\begin{proof}
Part (a) of the sublemma is \cite[Lemma~4.2]{dz1}.  To prove part (b), 
note that  due to the uniform bound on slopes of stable curves, it follows
\begin{equation}
\label{eq:JG}
1 \le JG_W(r) := \sqrt{1 + (\vf_W'(r))^2} \le \sqrt{1 + (\cK_{\max} + \tau_{\min}^{-1})^2} =: C_g < \infty \, .
\end{equation}
Therefore
$1 \le |JG_{U^\ell_j}|_{C^0(I_j)} \le C_g$, and we have
\begin{align*}
|\psi_1 \circ T^n |J^sT^n|^t & - \tpsi_2 |\tJ^sT^n|^t |_{C^\beta(U^1_j)}\\
&\le C_g |(\psi_1 \circ T^n |J^sT^n|^t) \circ G_{U^1_j} - (\psi_2 |\tJ^sT^n|^t) \circ G_{U^2_j} |_{C^\beta(I_j)} \\
& \le C_g |\psi_2 \circ T^n|_{C^\beta(U^2_j)} | |J^sT^n|^t \circ G_{U^1_j} - |J^sT^n|^t \circ G_{U^2_j} |_{C^\beta(I_j)} \\
& \qquad + C_g ||J^sT^n|^t|_{C^\beta(U^1_j)} |\psi_1 \circ T^n \circ G_{U^1_j} - \psi_2 \circ T^n \circ G_{U^2_j}|_{C^\beta(I_j)} \\
& \le C_gC_1^{-1} | |J^sT^n|^t \circ G_{U^1_j} - |J^sT^n|^t \circ G_{U^2_j} |_{C^\beta(I_j)} \\
& \qquad + C_g (1+2^tC_d) ||J^sT^n|^t|_{C^0(U^1_j)} |\psi_1 \circ T^n \circ G_{U^1_j} - \psi_2 \circ T^n \circ G_{U^2_j}|_{C^\beta(I_j)} \, ,
\end{align*}
where we have used \eqref{eq:test contract} and \eqref{eq:JC} for the final inequality.  We first observe that
\[
|\psi_1 \circ T^n \circ G_{U^1_j} - \psi_2 \circ T^n \circ G_{U^2_j}|_{C^\beta(I_j)} \le C \ve^{\alpha-\beta} \, ,
\]
by \cite[Lemma~4.4]{dz1}.
For brevity, set $J_\ell = J^sT^n\circ G_{U^\ell_j}$.  By\footnote{The case $q=2$ is treated there, the general case is similar.} \cite[eq. (4.16)]{dz1}, we have
\begin{equation}
\label{eq:Cv}
\left| 1 - \frac{J_1(r)}{J_2(r)} \right| \le C_{vert} \,  \ve^{1/(q+1)} \, ,
\,\, \forall r \in I_j\, , 
\end{equation}
for some constant $C_{vert} >0$ depending only on the uniform angle between the vertical direction and the stable and unstable
cones.  
Thus, since $\ve_0>0$ satisfies \eqref{eq:Cve}
and $\ve \le \ve_0$, this implies that $\frac 14 \le \frac{J_1(r)}{J_2(r)} \le \frac 74$.
Then, estimating as in Lemma~\ref{lem:distortion}, we have
\[
| J_1^t(r) - J_2^t(r) | \le 2^t |J_1^t|_{C^0(I_j)} C_{vert} \,  \ve^{1/(q+1)} \, .
\]
Following \cite[eq. (4.17) and (4.18)]{dz1}, yields,
\[
H^\beta(J_1^t - J_2^t) \le C 2^t |J_1^t|_{C^0(I_j)} \sup_{r,s \in I_j} \min \{ \ve^{1/(q+1)} |r-s|^{-\beta}, |r-s|^{1/(q+1) - \beta} \} \, ,
\]
where $H^\beta(\cdot)$ is the H\"older constant with exponent $\beta$ on $I_j$.
This bound is maximized when $\ve = |r-s|$, which yields 
$H^\beta(J_1^t - J_2^t) \le C 2^t |J_1^t|_{C^0(I_j)} \ve^{1/(q+1)-\beta}$.  Putting these estimates together yields,
\[
| |J^sT^n|^t \circ G_{U^1_j} - |J^sT^n|^t \circ G_{U^2_j} |_{C^\beta(I_j)} \le C 2^t |J^sT^n|^t_{C^0(U^1_j)} \ve^{1/(q+1) - \beta} \, .
\]
Together with the previous estimate on $\psi_\ell$, this
completes the proof of the sublemma since $\alpha \le 1/(q+1)$.
\end{proof}

Returning to \eqref{eq:unstable decomp}, we split the estimate over matched pieces as follows. 

\begin{align}
\nonumber
\left| \int_{U^1_j} f \, \psi_1 \circ T^n \, |J^sT^n|^t \right. & - \left. \int_{U^2_j} f \, \psi_2 \circ T^n \, |J^sT^n|^t \right|
 \le \left| \int_{U^1_j} f \, (\psi_1 \circ T^n \, |J^sT^n|^t - \tpsi_2 |\tJ^sT^n|^t )\right| \\
\label{eq:match split}& \qquad + \left|   \int_{U^1_j} f \, \tpsi_2 |\tJ^sT^n|^t
-   \int_{U^2_j} f \, \psi_2 \circ T^n \, |J^sT^n|^t  \right| \, .
\end{align}
We estimate the first term on the right side using the strong stable norm and Lemma~\ref{lem:matching}(b),
\[
\left| \int_{U^1_j} f \, (\psi_1 \circ T^n \, |J^sT^n|^t - \tpsi_2 |\tJ^sT^n|^t )\right|
 \le \| f \|_s \delta_0^{1/p} C 2^t |J^sT^n|^t_{C^0(U^1_j)}  \ve^{\alpha - \beta} \, .
\]
Then, noting that $d(\psi_1 \circ T^n \, |J^sT^n|^t, \tpsi_2 |\tJ^sT^n|^t) = 0$ by definition, and the $C^\alpha$ norm of each
test function is bounded by $C 2^t |J^sT^n|^t_{C^0(I_j)}$, using \eqref{eq:test contract} and \eqref{eq:JC}, we estimate the
second term on the right side of \eqref{eq:match split} using the strong unstable norm:
\begin{equation}
\label{eq:f B}
\begin{split}
{ \left| \int_{U^1_j} f \, \tpsi_2 |\tJ^sT^n|^t  - \int_{U^2_j} f \, \psi_2 \circ T^n |J^sT^n|^t \right| }
& \le \| f \|_u d_{\cW^s}(U^1_j, U^2_j)^\gamma  C 2^t |J^sT^n|^t_{C^0(U^1_j)} \\
& \le C' \| f \|_u n^\gamma \Lambda^{-n \gamma} \ve^\gamma |J^sT^n|^t_{C^0(U^1_j)} \, ,
\end{split}
\end{equation}
where we used Lemma~\ref{lem:matching}(a) in the second inequality.  Putting these estimates into \eqref{eq:match split},
then combining with \eqref{eq:unmatched} in \eqref{eq:unstable decomp}, and summing over $j$ (since there are at most
two matched pieces corresponding to each element of $\cG_n(W^1)$), yields,
\begin{equation}
\label{eq:final unstable}
\begin{split}
&\left| \int_{W^1} \cL_t^n f \, \psi_1 - \int_{W^2} \cL_t^n f \, \psi_2 \right|\\
&\qquad\qquad\qquad\le C \left( \| f \|_u n^\gamma \Lambda^{-n \gamma} \ve^\gamma Q_n(t) + 
\| f \|_s (\ve^{1/p} Q_n(t-1/p) + \ve^{\alpha- \beta} Q_n(t)) \right) \, .
\end{split}
\end{equation}
Dividing through by $\ve^\gamma$ and taking the appropriate suprema over $W^\ell$ and $\psi_\ell$ proves
\eqref{eq:unstable} since
$\gamma \le \min \{ 1/p, \alpha - \beta \}$.


\subsection{Proof of Lemma~\ref{lem:image}. ($\cL_t(C^1)\subset \cB$)}
\label{sec:approach}

 We assume $0\le t<1$.  The proof for $t \ge 1$ is similar, but simpler,
since $\cL_t f$ is bounded when $t \ge1$.
 Without loss of generality, we also assume that
$t_0 \le 1/2$, so that, by Definition~\ref{def:k0d0}, $q \ge 8$ and $p>9$.

 We  introduce a mollification in order 
to approximate $\cL_t f$ by functions in $C^1(M)$:
Let $\rho: \mathbb{R}^2 \to \mathbb{R}$ be a $C^\infty$ nonnegative, rotationally symmetric 
function supported on the unit disk with 
$\int_{\mathbb{R}^2} \rho \, d^2z = 1$ and 
$|\rho|_{C^1} \le 2$.
For $f \in C^1(M)$ and $\eta>0$, define
\[
g_\eta(x) = \int_{B_\eta(x)} \eta^{-2} \rho\left( \frac{d(x,z)}{\eta} \right) \, \cL_tf(z) d^2z \, ,
\]
where $B_\eta(x)$ is the ball of radius $\eta$ centered at $x$.  
Viewing $M$ as a subset of $\mathbb{R}^2$, we set $\cL_tf \equiv 0$ outside $M$
so that the integral is well-defined even when $B_\eta(x) \not\subset M$.
We first develop bounds on
$|g_\eta|_{C^0(M)}$ and $|g_\eta|_{C^1(M)}$, for any $t\ge 0$.

Since $t<1$, the operator
$\cL_tf$ is unbounded in neighborhoods of $T\cS_0$, so the bounds on $g_\eta$ will
be greatest in such neighborhoods.  Suppose $x$ and $\eta$ are such that $B_\eta(x) \cap T\cS_0 \neq \emptyset$ and note that there can be at most $\tau_{\max}/\tau_{\min} + 1$ connected components of  $B_\eta(x) \setminus T\cS_0$.  Fix one such component with boundary $S \in T\cS_0$ such that $S$ is the accumulation of the sequence of sets,
$B_\eta(x) \cap T\bH_k$, $k \ge k_0$.  On each such set,
$|J^sT|^{1-t} = C^{\pm 1} k^{-q(1-t)}$.  Also, due to the uniform transversality of $T\cS_0^{\bH}$
with the stable cone, we have  $\diam^s(B_\eta(x) \cap T\bH_k) \le C k^{-2q-1}$,
and  $\diam^u(B_\eta(x) \cap T\bH_k) \le C\eta$.
Moreover, since the boundary of $T\bH_k$ has distance approximately $k^{-2q}$ from $S$,
we have $B_\eta(x) \cap T\bH_k = \emptyset$ unless $k \ge C\eta^{-1/(2q)}$.
Assembling these facts, we estimate,
\[
|g_\eta| \le C \sum_{k \ge C\eta^{-1/(2q)}} \int_{B_\eta(x) \cap T\cS_0} \eta^{-2} \cL_t f \, d^2 z
\le C|f|_\infty \sum_{k \ge C\eta^{-1/(2q)}} \eta^{-1} k^{-q-1-{ qt} } \, .
\]
We conclude that, for any { $0 \le t < 1$},\footnote{For $t=0$, any choice of $q>1$ gives the same bound.}
\begin{equation}
\label{eq:geta bound}
|g_\eta|_{C^0(M)} \le C |f|_\infty \eta^{{ \frac{t}{2}} - \frac 12},
\quad \mbox{and similarly,}
\quad
|g_\eta|_{C^1(M)} \le C |f|_\infty \eta^{{ \frac{t}{2}} - \frac 32 } \, .
\end{equation}

To prove Lemma~\ref{lem:image},  we must control  $g_\eta - \cL_t f$ integrated along stable manifolds.
To this end, we will need the following two lemmas.
(The first one is classical and the second uses
bounds on the auxiliary foliation constructed in \cite[Section~6]{bdl}.)

\begin{lemma}
\label{lem:lots}
Let $W \in \cW^s$ be weakly homogeneous and for $\eta > 0$ let $W_u(\eta) \subset W$ 
denote the set of points in $W$ 
whose unstable manifold extends at least length $\eta$ on both sides of $W$.  Then
$m_W(W \setminus W_u(\eta)) \le C\eta$ for some constant $C>0$ independent of
$W$ and $\eta$.
\end{lemma}

\begin{proof}
This is precisely \cite[Theorem~5.66]{chernov book}.  See also the corrected proof in
\cite{baldestot}.
\end{proof}

\begin{lemma}
\label{lem:smooth}
There exist constants $C, C_s>0$ such that for any
weakly homogeneous unstable curve $U$ and any $\varrho > 0$, 
there exists a set $U' \subset U$
with $m_U(U \setminus U') \le C \varrho$ such that 
\[
 \left| \frac{J^sT(x)}{J^sT(y)} - 1 \right| \le C_s \left( \varrho^{- \frac{q}{q+1}} k_U^{-q} d(x,y) + d(x,y)^{1/(q+1)} \right) \, , \quad \forall x, y \in U'\, , 
\]
where $k_U$ is the index of the homogeneity strip containing $U$.
\end{lemma}

\begin{proof}
Fixing a length $\varrho < k_U^{-q-1}$, we define a foliation of stable curves transverse
to $U$, following the procedure\footnote{\cite{bdl} constructs this as a foliation of unstable curves
transverse to a stable curve. By the time reversal property of the billiard, the same
construction holds with stable and unstable directions exchanged.} 
in \cite[Sections~6.1, 6.2]{bdl}:  Choose $n \in \mathbb{N}$ arbitrarily
large and define a smooth ``seeding'' foliation of homogeneous 
stable curves transverse to connected components
of $T^nU$; elements of the seeding foliation are then pulled back under $T^{-n}$ and those that
are not cut form a foliation of homogeneous 
stable curves of length at least $\varrho$ and transverse to  $U$. Letting
$U'_n \subset U$ denote the set covered by this surviving foliation, we have  
$m_U(U \setminus U'_n) \le C \varrho$, for some $C>0$ independent of $n$
\cite[Section~6.1]{bdl}.
Moreover,
expressing the foliation in local coordinates $(s,u)$ adapted to the stable and unstable directions
defines a function $G(s,u)$ such that each stable curve can be expressed as
$\{ (s, G(s,u)) \}_{s \in [-\varrho, \varrho]}$, and $G(0,u) = u$, so that 
the unstable manifold $U$ corresponds to the vertical segment $\{ (0,u) \}_{u \in [0,|U|]}$.
It follows that the slope $\cV(u)$ of
the tangent vector to the foliation at $(0,u)$ 
is just
 $\partial_sG(0,u)$.
By \cite[Lemma~6.5]{bdl}, $\partial_u\partial_sG \in C^0$ with 
$|\partial_u\partial_sG|_\infty \le C\varrho^{-q/(q+1)} k_U^{-q}$ (where we have adapted
the exponent according to the spacing of our homogeneity strips).  

Note that the foliation of stable curves constructed in this way has tangent
vectors in $DT^{-n}C^s$.  
Since the bounds on $m_U(U \setminus U'_n)$
and $|\partial_u \partial_s G|_\infty$
are independent of $n$, we conclude there exists a set $U' \subset U$ with
$m_U(U \setminus U') \le C\varrho$ such that the stable manifolds
passing through $U'$ have length at least $\varrho$ and 
satisfy $|\partial_u\partial_sG|_\infty \le C\varrho^{-q/(q+1)} k_U^{-q}$
(see also \cite[Remark~1.1]{bdl}).
 
Finally,  for $u,v \in U'$ we estimate as in \eqref{eq:jac} (with $n=1$), 
using \eqref{eq:cos} for $\log \frac{\cos \vf(u)}{\cos \vf(v)}$ and 
$|\cV(u) - \cV(v)| \le C\varrho^{-q/(q+1)} k_U^{-q} d(u,v)$ from the construction in \cite{bdl}.
Putting these estimates together proves the lemma.
\end{proof}

We record for future use that for any measurable set $V \subseteq W  \in \cW^s$,
\begin{equation}
\label{eq:size bound}
\begin{split}
\int_V \cL_t f \, \psi \, dm_V & = \int_{T^{-1}V} f |J^sT|^t \psi \circ T \, dm_{T^{-1}V} \\
& \le |f|_\infty |\psi|_\infty |T^{-1}V|^{1-t}|V|^t \le C |f|_\infty |\psi|_\infty |V|^{(1+t)/2} \, ,
\end{split}
\end{equation}
where $|V|$ denotes the arc length measure of $V$, and we have used the H\"older inequality for the first inequality and the bound
$|T^{-1}V| \le C|V|^{1/2}$ in the second.

\medskip
\noindent
{\em Approximating the strong stable norm.}
Fix $\eta >0$, and let $W \in \cW^s$ and $\psi \in C^\beta(W)$ with 
$|\psi|_{C^\beta(W)} \le |W|^{-1/p}$.
If $|W| \le \eta$, then using \eqref{eq:geta bound} and  \eqref{eq:size bound}, we write, simply, 
using $p>9$,
\begin{equation}
\label{eq:short W}
\int_W (\cL_tf - g_\eta) \psi \, dm_W 
\le C|f|_\infty |W|^{-1/p} \big( |W|^{\frac{1+t}{2}} + |W| \eta^{\frac{t}{2} - \frac 12} \big)
\le C|f|_\infty \eta^{\frac{t}{2} + \frac{1}{3}} \, .
\end{equation}
In what follows, we assume $|W| > \eta$.  Let $W_\eta^-$ denote the curve $W$
minus the $\eta$-neighborhood of its boundary.
Treating the integral over the two components of $W \setminus W_\eta^-$ in the same way
as \eqref{eq:short W}, we estimate,
using that $m_W(W \setminus W_\eta^-) \le 2\eta$,
\begin{equation}
\label{eq:ends}
\int_{W\setminus W_\eta^-} (\cL_t f - g_\eta) \psi \, dm_W \le C |f|_\infty \eta^{\frac{t}{2} + \frac 13} 
\, .
\end{equation} 
Next,
since
$W$ intersects at most $N = \tau_{\max}/\tau_{\min}+1$ elements of $T\cS_0$, 
the  set
$W \cap \big( \cup_{k \ge \eta^{-1/(2q+1)}} T\bH_k \big)$ comprises at most $N$ intervals of length $C\eta^{2q/(2q+1)}$.
We estimate as in \eqref{eq:short W} using 
$V = W \cap \big( \cup_{k \ge \eta^{-1/(2q+1)}} T\bH_k \big)$ in \eqref{eq:size bound}, and that  $p > q+1\ge 9$
\begin{equation}
\label{eq:high index}
\int_{W \cap \big( \cup_{k \ge \eta^{-1/(2q+1)}} T\bH_k \big)}
(\cL_t f - g_\eta) \psi \, dm_W \le C|f|_\infty \eta^{\frac{t}{2} + \frac{3}{10}} \, .
\end{equation}
Finally, we estimate $\cL_tf - g_\eta$ on those portions of $W_\eta^-$ that intersect
$T\bH_k$ for $k \le \eta^{-1/(2q+1)}$.  Let $x$ be such a point in $W_\eta^-$.  Due to the
restriction on $k$, the ball $B_\eta(x)$ lies in a bounded number of homogeneity strips, so
we may use bounded distortion in conjunction with Lemma~\ref{lem:smooth} to
bound the difference in each such interval.  Let $S_\eta = W \setminus W_u(\eta)$ 
denote the exceptional set of points in Lemma~\ref{lem:lots}.
We write $A_\eta(x)$ for the subset
of $B_\eta(x)$ foliated by unstable manifolds of length at least $2\eta$, and let $E_\eta(x) = B_\eta(x) \setminus A_\eta(x)$.
Then,
\begin{align}
\label{eq:good bad}
 \cL_tf(x) - g_\eta(x)  &= \int_{B_\eta(x)} \eta^{-2} \rho(\tfrac{d(x,z)}{\eta}) (\cL_tf(x) - \cL_tf(z)) d^2z \\
\nonumber & = \int_{A_\eta(x)} \eta^{-2} \rho(\tfrac{d(x,z)}{\eta}) (\cL_tf(x) - \cL_tf(z)) d^2z\\
\nonumber &\qquad\qquad\quad
+ \int_{E_\eta(x)} \eta^{-2} \rho(\tfrac{d(x,z)}{\eta}) (\cL_tf(x) - \cL_tf(z)) d^2z \, .
\end{align}
We first estimate the integral over $E_\eta(x)$ using the bound 
$\cL_tf(z) \le C \cL_tf(x)$ for $z \in B_\eta(x)$, since
$B_\eta(x)$ lies in a bounded number of homogeneity strips.  Then, using the fact that
the unstable foliation is absolutely continuous, we disintegrate as follows,
\begin{equation}
\label{eq:E part}
\int_{E_\eta(x)} \eta^{-2} \rho(\tfrac{d(x,z)}{\eta}) (\cL_tf(x) - \cL_tf(z)) d^2z \\
\le C \cL_tf(x) \eta^{-1} |S_\eta \cap B_\eta(x)| \, .
\end{equation}

Next, we estimate the integral over $A_\eta(x)$.  Since each point $y \in A_\eta(x) \cap W_\eta^-$ 
has an unstable manifold $U_y$ extending a length at least $\eta$ on either side of $W$, 
we set $\varrho = \eta^{1+\frac{1}{2q}}$ and
denote by $A_\eta'(x)$ those points contained in sets $U'_y \subset U_y$ satisfying 
Lemma~\ref{lem:smooth}.  It follows from that lemma and the absolute continuity of the unstable
foliation that 
\begin{equation}
\label{eq:bad A}
\int_{A_\eta(x) \setminus A_\eta'(x)} \eta^{-2} \rho(\tfrac{d(x,z)}{\eta}) (\cL_tf(x) - \cL_tf(z)) d^2z
\le C \cL_tf(x) \eta^{\frac{1}{2q}} \, ,
\end{equation}
where we have again used the bound $\cL_tf(z) \le C \cL_tf(x)$ on $B_\eta(x)$. 
 
For $z \in A_\eta'(x)$, we bound the difference $\cL_tf(x) - \cL_tf(z)$ as follows.
Let $y = [x,z]$ denote the point of intersection between the stable manifold of $x$ (which is $W$)
and the unstable manifold of $z$, which is $U_y$.  By definition, $z \in U_y'$ and it is always the
case that $y \in U_y'$ since the stable manifold of $y$, $W$, has length at least $\eta > \varrho$.
Thus,
\begin{align*}
&|\cL_tf(x)  - \cL_tf(z)|  \\
&\qquad\le \left| \frac{f ( T^{-1}x)}{|J^sT|^{1-t} ( T^{-1}x)}
- \frac{f ( T^{-1}y)}{|J^sT|^{1-t} ( T^{-1}y)} \right|
+ \left| \frac{f ( T^{-1}y)}{|J^sT|^{1-t} ( T^{-1}y)}
- \frac{f ( T^{-1}z)}{|J^sT|^{1-t} ( T^{-1}z)} \right| \\
& \qquad \le  \cL_t 1(x)  [  |f|_{C^1} d(T^{-1}x,T^{-1}y) + |f|_{C^0} C d(T^{-1}x, T^{-1}y)^{1/(q+1)}  \\
& \qquad \quad+  |f|_{C^1} d(T^{-1}y, T^{-1}z) + |f|_{C^0} C_s (\eta^{-\frac{2q+1}{2q+2}} d(T^{-1}y, T^{-1}z)
+ d(T^{-1}y, T^{-1}z)^{1/(q+1)})  ] \, ,
\end{align*}
where we have used Lemma~\ref{lem:distortion} along $W$ and Lemma~\ref{lem:smooth}
along $U_y$ with $\varrho = \eta^{\frac{2q+1}{2q}}$.  Next, $d(T^{-1}y, T^{-1}z) \le C d(y,z)
\le C \eta$, while for $x \in T\bH_k$,
\[
d(T^{-1}x, T^{-1}y) \le C k^q d(x,y) \le C \eta^{\frac{q+1}{2q+1}} \,  
\]
since $k \le \eta^{-1/(2q+1)}$.  Putting these estimates together we obtain,
\[
|\cL_tf(x) - \cL_tf(z)| \le |f|_{C^1} \cL_t1(x) C \eta^{\frac{1}{2q+2}} \qquad \mbox{for $z \in A_\eta'(x)$},
\]
and combining this with \eqref{eq:E part} and \eqref{eq:bad A} in \eqref{eq:good bad} yields,
\begin{equation}
\label{eq:A and E}
|\cL_tf(x) - g_\eta(x)| 
\le C |f|_{C^1} \cL_t1(x) \eta^{\frac{1}{2q+2}} + C |f|_{C^0} \cL_t1(x) \eta^{-1} |S_\eta \cap B_\eta(x)| \, .
\end{equation}

We must integrate this bound over 
$W_\eta^- \cap (\cup_{k \le \eta^{-1/(2q+1)}}T\bH_k)$.  We estimate the integral of the first term 
in \eqref{eq:A and E} simply using \eqref{eq:size bound},
\begin{equation}
\label{eq:good one}
C|f|_{C^1} \eta^{\frac{1}{2q+2}} \int_{W_\eta^- \cap (\cup_{k \le \eta^{-1/(2q+1)}}T\bH_k)}
\cL_t1 \, \psi \, dm_W \le C |f|_{C^1} \eta^{\frac{1}{2q+2}} \, .
\end{equation}
Finally, to bound the second term in \eqref{eq:A and E}, we write $I_\eta(x) = B_\eta(x) \cap W$ and 
\[
|S_\eta \cap B_\eta(x)| = \int_{I_\eta(x)} 1_{S_\eta}(z) dm_W(z)
= \int_{-\eta}^\eta 1_{S_\eta}(x+G_W(x;r)) JG_W(x;r) dr \, ,
\]
where $G_W(x;r)$ denotes the (local) graph of the function defining $W$ in a neighborhood of $x$, 
as in \eqref{eq:match}, and we
have 
centered the local $r$-interval at $r=0$.
Then,
\begin{align}
\nonumber
 &\int_{W_\eta^- \cap (\cup_{k \le \eta^{-1/(2q+1)}}T\bH_k)}
\cL_tf(x) \frac{\psi(x) }{\eta} \int_{-\eta}^\eta 1_{S_\eta}(x+G_W(x;r)) JG_W(x;r) \, dr \, dm_W(x) \\
\nonumber &  \qquad\qquad\le |f|_{C^0} \frac{|W|^{-1/p} }{\eta} \int_{-\eta}^\eta \int_{W_\eta^-} \cL_t1(x) 1_{S_\eta}(x+G_W(x;r)) JG_W(x;r) \, dm_W(x) \, dr \\
& \label{eq:bad one} \qquad\qquad\le C|f|_{C^0} \frac{|W|^{-1/p}}{ \eta} \int_{-\eta}^\eta |S_\eta|^{(1+t)/2} dr
\; \le \; C|f|_{C^0} \eta^{\frac{1}{3} + \frac t2} \, ,
\end{align}
where we have used \eqref{eq:JG} and the fact that translations of $W_\eta^-$ up to length 
$\eta$ are 
subsets of $W$ in order to apply \eqref{eq:size bound} for the second inequality,
and Lemma~\ref{lem:lots}, with $|W| \ge \eta$ and $p>9$ for the final inequality.

Finally, using \eqref{eq:good one} and \eqref{eq:bad one} in \eqref{eq:A and E},
and adding the contributions from \eqref{eq:ends} and \eqref{eq:high index} in addition
to \eqref{eq:short W} yields,
\begin{equation}
\label{eq:final approx}
\int_W (\cL_tf - g_\eta) \, \psi \, dm_W \le C|f|_{C^1} \eta^{\frac{1}{2q+2}} \, ,
\end{equation}
for some $C>0$ independent of $W$, 
since $\min \{ \frac{t}{2}+\frac{3}{10}, \frac{1}{2q+2} \} = \frac{1}{2q+2}$ whenever
$q>1$ and $t>0$.  
Taking the appropriate
suprema over $\psi$ and $W$ yields the required estimate
$\| \cL_tf - g_\eta \|_s \le C|f|_{C^1} \eta^{\frac{1}{2q+2}}$.

\medskip
\noindent
{\em Approximating the unstable norm.}
Let $\ve \le \ve_0$ and $W_1, W_2 \in \cW^s$ with $d_{\cW^s}(W_1, W_2) \le \ve$.
Let $\psi_i \in C^\alpha(W_i)$ with $|\psi_i|_{C^\alpha(W_i)} \le 1$, $i=1,2$, and
$d(\psi_1, \psi_2) = 0$.  We must estimate,
\[
\int_{W_1} ( \cL_t f - g_\eta) \psi_1\, dm_{W_1} - \int_{W_2} (\cL_tf - g_\eta) \psi_2 \, dm_{W_2} \, .
\]
We consider two cases.

\smallskip
\noindent
{\em Case 1: $\eta^{\frac{1}{2q+2}} < \ve^{2\gamma}$.}  We apply \eqref{eq:final approx} to each term separately
and obtain
\[
\ve^{-\gamma} \left| \int_{W_1} ( \cL_t f - g_\eta) \psi_1\, dm_{W_1} - \int_{W_2} (\cL_tf - g_\eta) \psi_2 \, dm_{W_2} 
\right| \le C|f|_{C^1} \eta^{\frac{1}{4q+4}} \, .
\]

\smallskip
\noindent
{\em Case 2: $\eta^{\frac{1}{2q+2}} \ge \ve^{2\gamma}$.}  In this case, we write
\begin{equation}
\label{eq:split fg}
\begin{split}
\int_{W_1} & ( \cL_t f - g_\eta) \psi_1\, dm_{W_1} - \int_{W_2} (\cL_tf - g_\eta) \psi_2 \, dm_{W_2} \\
& = \int_{W_1} \cL_t f \, \psi_1\, dm_{W_1} - \int_{W_2} \cL_tf \, \psi_2 \, dm_{W_2} 
 + \int_{W_2} g_\eta \, \psi_2\, dm_{W_2} - \int_{W_1} g_\eta \, \psi_1 \, dm_{W_1} \, .
\end{split}
\end{equation}
We estimate the difference involving $\cL_tf$ using the estimates in Section~\ref{sec:unstable},
but using the fact that $f \in C^1(M)$ to obtain stronger bounds.  In particular,
the integral over unmatched pieces from \eqref{eq:unmatched} is bounded
by $C|f|_{C^0} \ve$.  The bound on the first term of \eqref{eq:match split} remains the same, but
the bound on the second term from \eqref{eq:f B} is improved to $C |f|_{C^1}  \ve$.
Putting these estimates together as in \eqref{eq:final unstable} and 
dividing\footnote{We use here the strict inequality $\gamma<\alpha-\beta$.} by
$\ve^\gamma$ implies,
\begin{equation}
\label{eq:f approx}
\ve^{-\gamma} \left| \int_{W_1} \cL_t f \, \psi_1\, dm_{W_1} - \int_{W_2} \cL_tf \, \psi_2 \, dm_{W_2}
\right| \le C|f|_{C^1} \ve^{\alpha-\beta-\gamma} 
\le C|f|_{C^1} \eta^{\frac{\alpha - \beta - \gamma}{4\gamma(q+1)}} \, .
\end{equation}
Next, we turn to the second difference in \eqref{eq:split fg}.  Using
the notation of Section~\ref{sec:unstable}, we split the integrals up into one integral over
the common $r$-interval $I_1 \cap I_2$ and at most two integrals over $I_1 \bigtriangleup I_2$.
The (at most two) curves $V^\ell_i \subset W_\ell$ corresponding to intervals in $I_1 \bigtriangleup I_2$ have length
bounded by $C\ve$ by definition of $d_{\cW^s}(W_1,W_2)$.  Thus using
\eqref{eq:geta bound}, we have
\begin{equation}\label{needl}
\int_{V^\ell_i} g_\eta \, \psi_i \, dm_{W_\ell} \le C|f|_{C^0} \eta^{\frac{t}{2} - \frac 12} \ve
\le C |f|_{C^0} \eta^{\frac{t}{2} - \frac 12 + \frac{1-\gamma}{4\gamma(q+1)}}  \ve^\gamma \, .
\end{equation}
On the curves $U_1, U_2$, which are the graphs of the functions $\vf_{U_1}, \vf_{U_2}$ over
$I_1 \cap I_2$, we have,
\begin{equation}\label{needl1}
\int_{U^1} g_\eta \, \psi_1 \, dm_{W_1} - \int_{U_2} g_\eta \, \psi_2 \, dm_{W_2}
\le | JG_{U_1} (g_\eta \psi_1) \circ G_{U_1} - JG_{U_2} (g_\eta \psi_2) \circ G_{U_2} |_{C^0(I_1 \cap I_2)} \, ,
\end{equation}
where $G_{U_\ell}(r) = (r, \vf_{U_1}(r))$.
Then estimating as in the proof of Sublemma~\ref{lem:matching}, we have
\begin{equation}\label{needl2}
| JG_{U_1} (g_\eta \psi_1) \circ G_{U_1} - JG_{U_2} (g_\eta \psi_2) \circ G_{U_2} |_{C^0(I_1 \cap I_2)}
\le C |g_\eta|_{C^1(M)} \ve \le C |f|_\infty \eta^{\frac{t}{2} - \frac 32 + \frac{1-\gamma}{4\gamma(q+1)}} \ve^\gamma \, ,
\end{equation}
where we have used the fact that $d(\psi_1, \psi_2) =0$ and $|\vf_{U_1}' - \vf_{U_2}'| \le \ve$.
Putting these estimates together with \eqref{eq:f approx} in \eqref{eq:split fg} yields,
\begin{align}
\label{lastd}
\ve^{-\gamma} & \left| \int_{W_1} ( \cL_t f - g_\eta) \psi_1\, dm_{W_1} - \int_{W_2} (\cL_tf - g_\eta) \psi_2 \, dm_{W_2} \right| \\
\nonumber & \qquad \qquad \le C|f|_{C^1} \eta^{\frac{\alpha - \beta - \gamma}{4\gamma(q+1)}} + C |f|_{C^0} \eta^{\frac{t}{2} - \frac 32 + \frac{1-\gamma}{4\gamma(q+1)}} \, ,
\end{align}
and we use $\gamma \le \frac{1}{6q+7}$ from \eqref{eq:all} to deduce that $- \frac 32 + \frac{1-\gamma}{4\gamma(q+1)} \ge 0$.
This completes Case~2, which, together with Case 1, implies the required bound
$\| \cL_tf - g_\eta\|_u \le |f|_{C^1(M)} \eta^\delta$, for some $\delta>0$, ending
the proof of Lemma~\ref{lem:image}.

\begin{remark}[Adapting the proof of Lemma~\ref{lem:image} to the case $t=0$]
\label{t=0} 
Homogeneity strips are not used in \cite{max}, so one requires a nonhomogeneous
version of Lemma~\ref{lem:lots}, but
it is not hard to show directly that there exists $C>0$ such that $m_W(W \setminus W_u(\eta)) \le C\sqrt{\eta}$ for  
any $W \in \cW^s$ and $\eta > 0$,
and this weaker bound suffices (see discussion of \eqref{eq:bad one} below). Lemma~\ref{lem:smooth} can be kept unchanged as it is only needed on unstable manifolds  contained in a single homogeneity strip.

We show how to adapt the proof of Lemma~\ref{lem:image}
 to the norm  from \cite[\S 4.1]{max} with $q=2$, and 
parameters  $\beta$, $\gamma$, and $\varsigma$:
Eq \eqref{eq:short W} and \eqref{eq:ends}  get better since the test function satisfies 
$|\psi| \le |\log |W||^\gamma$, so we find $\eta^{1/2} |\log \eta|^\gamma$.
Similarly, \eqref{eq:high index} has the bound $\eta^{3/10} |\log \eta|^\gamma$.
Eq \eqref{eq:good bad}--\eqref{eq:good one} remain as written.
Eq \eqref{eq:bad one} proceeds as above until the last line, at which point we use
$|S_\eta| \le C\sqrt{\eta}$, so that the final bound becomes
$C|f|_\infty \eta^{1/4} |\log \eta|^\gamma$.
Thus we arrive at  \eqref{eq:final approx} with a bound $C |f|_{C^1} \eta^{1/6}$.
The factor $|\log \eta|^\gamma$ can be absorbed by the various  exponents, all being greater than $1/6$.  So there is no extra restriction the parameter $\gamma$ from \cite{max}
from the stable norm estimate.

For the unstable norm estimate, one 
distinguishes between the case $\eta^{1/6} <  |\log \ve|^{-2 \varsigma}$, which yields a bound 
with $\eta^{1/12}$, and the case
 $\eta^{1/6} \ge |\log \ve|^{-2 \varsigma}$, which  implies that $\ve \le \exp(-\eta^{-\frac{\alpha-\beta}{12\varsigma}})$, which is superexponentially small in $\eta$, so that
 \eqref{eq:split fg} remains the same, while  \eqref{eq:f approx} is bounded by
$\ve^{\alpha - \beta} |\log \ve|^\varsigma \le \exp(-\eta^{-\frac{\alpha-\beta}{24\varsigma}})$.
Similarly, \eqref{needl}
 is  bounded by $|\log\ve|^{-\varsigma}$ times a factor superexponentially small in $\eta$.  (We have a power of $\ve$ which is factored into $|\log \ve|^{-\varsigma}$ times $\ve^{1-\delta}$, for any $\delta$.)  The same is true of 
\eqref{needl1}--\eqref{needl2}.
Finally, in \eqref{lastd}, we end up with 
$\exp(-\eta^{-\frac{\alpha-\beta}{24\varsigma}})$ plus 
$\eta^{-3/2}\exp(- \eta^{-\frac{1}{24 \varsigma}})$, and this goes to $0$ as $\eta$ goes to $0$, for any $\varsigma >0$ 
(in particular,  there is no extra condition on $\varsigma$ from this estimate).
\end{remark}

\subsection{Spectral Gap for $\cL_t$. Constructing $\mu_t$ (Proof of Theorem~\ref{thm:spectral})}
\label{4.4}

We harvest the results from the previous subsections to show
Theorem~\ref{thm:spectral} at the end of this section.
Our first result follows from Proposition~\ref{prop:ly}
and the exact growth
for $Q_n(t)$ (Propositions~\ref{prop:exact} and \ref{prop:summary}).

\begin{proposition}[Quasi-compactness]
\label{prop:radius}
Let $t_0\in (0,1)$ and $t_1\in (1,t_*)$. Then we can
choose  parameters for $\cB$ such that for any $t\in  [t_0,t_1]$,
the operator $\cL_t$ acting on $\cB$ is quasi-compact: its
spectral radius is $e^{P_*(t)}$ and its essential spectral radius is at most $\sigma e^{P_*(t)}$, where
$$\sigma  := \max \{ \Lambda^{-\beta + 1/p}, \theta^{t-1/p} e^{-P_*(t)}, \Lambda^{-\gamma} \} <1
\, . $$
Moreover, the peripheral spectrum of $\cL_t$ contains no Jordan blocks.
\end{proposition}

\begin{proof} Since $t_0>0$
and $t_1<t_*$,  we can choose $p>1$ 
such that $p> 2/t_0\ge 2/t$ and (by Lemma~\ref{lem:t ok}) $\theta^{(t-1/p)}e^{- P_*(t)} < 1$ 
for any $t\in [t_0, t_1]$.
Then \eqref{eq:weak} and Proposition~\ref{prop:summary}(b)  imply
that the spectral radius of $\cL_t$  on $\cB_w$ is at most $e^{P_*(t)}$. Combining \eqref{eq:LY} from Proposition~\ref{prop:ly} with
Hennion's theorem and compactness of the unit ball of 
$\cB$ in $\cB_w$ from Proposition~\ref{embeds}, the essential
spectral radius of $\cL_t$ on $\cB$ is at most $\sigma e^{P_*(t)} < e^{P_*(t)}$.
Hence the spectral radius of $\cL_t$ on $\cB$ is at most
 $e^{P_*(t)}$.

Next, notice that by Lemma~\ref{lem:distortion} and
our choice of $\delta_1$ in \eqref{eq:delta111},  
we have for $W \in \cW^s$ with $|W| \ge \delta_1/3$,
\begin{equation}
\label{eq:lower evol}
\begin{split}
\int_W \cL_t^n 1 \, dm_W & = \sum_{W_i \in \cG_n^{\delta_1}(W)} \int_{W_i} |J^sT^n|^t \, dm_{W_i}
\ge \sum_{W_i \in L_n^{\delta_1}(W)} \tfrac 13 \delta_1 2^{-t} |J^sT^n|^t_{C^0(W_i)} \\
& \ge \tfrac 14 \delta_1 2^{-t} \sum_{W_i \in \cG_n^{\delta_1}(W)} |J^sT^n|^t_{C^0(W_i)} 
\ge \tfrac 14 \delta_1 2^{-t} c_1 Q_n(t) \, ,
\end{split}
\end{equation}
where for the final inequality we have applied Proposition~\ref{prop:summary}(b).  Then,
since $Q_n(t) \ge e^{n P_*(t)}$ by the lower bound
in Proposition~\ref{prop:summary}(b), we conclude
\[
\| \cL_t^n \|_{\cB} \ge \| \cL_t^n 1\|_{\cB} (\| 1 \|_{\cB})^{-1} \ge (\| 1 \|_{\cB})^{-1}C \delta_1 e^{P_*(t) n} \quad \Longrightarrow \quad 
\lim_{n \to \infty} \| \cL_t^n \|_{\cB}^{1/n} \ge e^{P_*(t)} \, .
\]
Thus the spectral radius of $\cL_t$ on $\cB$ is in fact $e^{P_*(t)}$ and $\cL_t$ is quasi-compact on $\cB$.

 Finally, to prove there are no Jordan blocks in the peripheral spectrum, assume
to the contrary that there exist $f_0, f_1 \in \cB$, $f_0 \neq 0$, and $\lambda \in \mathbb{C}$, 
$|\lambda| = e^{P_*(t)}$, such that 
$\cL_t f_0 = \lambda f_0$ and $\cL_t f_1 = \lambda f_1 + f_0$.  Then
$\cL_t^n f_1 = \lambda^n f_1 + n \lambda^{n-1} f_0$, so that
\[
n | f_0 |_w \le e^{P_*(t)} | f_1 |_w + e^{-(n-1)P_*(t)} | \cL_t^n f_1 |_w  \, ,
\]
and dividing by $n$, letting $n \to \infty$ and applying \eqref{eq:weak} and
Proposition~\ref{prop:summary}(b) yields $|f_0|_w = 0$.  The injectivity of $\cB_w$ into $\cB$
given by Proposition~\ref{embeds} implies $f_0=0$ in $\cB$, a contradiction.
\end{proof}

For $\varpi \in [0, 1)$, let $\bV_\varpi$ denote the eigenspace of $\cL_t$ on $\cB$ corresponding
to the eigenvalue $e^{P_*(t)} e^{2\pi i \varpi}$.  Due to Proposition~\ref{prop:radius}, we
have the following decomposition of $\cL_t$  on $\cB$,
\begin{equation}
\label{eq:decomp}
\cL_t = \sum_{\varpi} e^{P_*(t) + 2\pi i \varpi} \Pi_\varpi + R_t \, ,
\end{equation}
where the sum is over finitely many $\varpi$ due to the quasi-compactness of $\cL_t$, and
$\Pi_\varpi^2 = \Pi_\varpi$, $\Pi_\varpi \Pi_{\varpi'} = R_t \Pi_\varpi = \Pi_\varpi R_t = 0$
for $\varpi \neq \varpi'$ (mod $2\pi$), and the spectral radius of $R_t$ is strictly less than
$e^{P_*(t)}$.

\begin{lemma}
\label{lem:peripheral}
Define $\displaystyle  \nu_t  = \lim_{n \to \infty} \frac 1n \sum_{k=0}^{n-1} e^{- P_*(t) k} \cL_t^k 1$.  
\begin{itemize}
  \item[a)]  Then $\nu_t \neq 0$ is a nonnegative Radon measure, and $e^{P_*(t)}$ is in the spectrum
  of $\cL_t$.
  \item[b)] All elements of $\bV = \oplus_\varpi \bV_\varpi$ are 
complex measures, absolutely
  continuous with respect to $\nu_t$.
\end{itemize}
\end{lemma}

 Lemma~\ref{lem:peripheral} is  standard, adapting what has been done  in the SRB case.  We give a proof for completeness.

\begin{proof}
(b) The lack of Jordan blocks enables us to define  spectral projectors by
\begin{equation}
\label{eq:spec}
\Pi_\varpi : \cB \to \bV_\varpi\, ,
\quad \Pi_\varpi f = \lim_{n \to \infty} \frac 1n \sum_{k=0}^{n-1} e^{- P_*(t) k - 2\pi i \varpi k} \cL_t^k f \, ,
\end{equation}
where convergence in the $\cB$ norm is guaranteed by Propositions~\ref{prop:ly}
and \ref{prop:summary}(b).  Moreover, since $C^1(M)$ is dense in $\cB$ and $\bV_\varpi$ is 
finite-dimensional, for each $\nu \in \bV_\varpi$, there exists $\bar f_\nu \in C^1(M)$ such that
$\Pi_\varpi \bar f_\nu = \nu$.

Taking $\nu \in \bV_\varpi$ and $\psi \in C^\alpha(M)$,  using \eqref{eq:spec} and recalling our
identification of $\bar f_\nu \in C^1$ with the measure $\bar f_\nu d\musrb$, we have
\[
|\nu(\psi)| \le \lim_{n \to \infty}
\frac 1n \sum_{k=0}^{n-1} e^{- P_*(t) k} |\cL_t^k \bar f_\nu(\psi)|
\le |\bar f_\nu|_\infty \Pi_0 1(|\psi|)
\le |\bar f_\nu|_\infty |\psi|_\infty \Pi_01(1) \, .
\] 
The last two inequalities show respectively that $\nu$ is a complex Radon measure, and 
is absolutely continuous with respect
to $\nu_t$, with density $f_\nu \in L^\infty(\nu_t)$.  It may be that
$f_\nu \neq \bar f_\nu$.

\medskip
\noindent
(a)  Item (b) implies also that $\nu_t$ is a nonnegative Radon measure
since $\bar f_{\nu_t} = 1$ and $\Pi_0$ is nonnegative.  Also, if $\nu_t=0$, then all elements
of $\bV_\varpi$ are $0$, contradicting the fact that the spectral radius of $\cL_t$ is $e^{P_*(t)}$.
Thus $\nu_t \ne 0$ and $e^{P_*(t)}$ is in the spectrum of $\cL_t$ since
$\cL_t \nu_t = e^{P_*(t)} \nu_t$.
\end{proof}

The dual operator $\cL_t^*$ acting on $\cB^*$ has the
same spectrum as $\cL_t$ on $\cB$.   Define 
\begin{equation}
\label{eq:dual}
\tnu_t:= \lim_{n \to \infty} \frac 1n \sum_{k=0}^{n-1} e^{-P_*(t) k} (\cL_t^*)^k d\musrb \, ,
\end{equation}
which converges in the dual norm due to the absence of Jordan blocks.  By 
the analogous
arguments to item (b) of Lemma~\ref{lem:peripheral}, the distribution $\tnu_t \neq 0$ is a nonnegative 
Radon measure, and every other
eigenvector corresponding to the peripheral spectrum is a Radon measure, absolutely 
continuous with respect to $\tnu_t$, with bounded density.

With $\tnu_t$, we will define our candidate $\mu_t$ for the equilibrium state 
in Proposition~\ref{prop:mut def}. For this  (and in \eqref{forCLT2}), we shall  use 
the following lemma (proved exactly as in \cite[Lemma~4.4]{max}) which gives 
more precise information about the inclusion
 $\cB_w\subset (C^1(M))^*$
in Proposition~\ref{embeds}.
Recalling \eqref{eq:holder def}, 
let $H^\alpha_{\cW^s_{\bH}}(\psi) = \sup_{W \in \cW^s_{\bH}} H^\alpha_W(\psi)$.

\begin{lemma}
\label{lem:distrib}
There exists $C>0$ such that for all $f \in \cB_w$ and $\psi \in C^\alpha(\cW^s_{\bH})$,
\[
|f(\psi)| \le C |f|_w \big( |\psi|_\infty + H^\alpha_{\cW^s_{\bH}}(\psi) \big) \, .
\]
\end{lemma}

In addition, we shall use the following extension of the above lemma, which uses that
$\musrb$ has smooth conditional measures on {\em stable} manifolds.
\begin{lemma}
\label{lem:Lt distr}
There exists $C>0$ such that for all $f \in \cB_w$, $\psi \in C^\alpha(\cW^s_{\bH})$, and $n \ge 0$,
\[
|\cL_t^n(f\psi)(1)| \le C Q_n(t) |f|_w |\psi|_{C^\alpha(\cW^s_{\bH})} \, .
\]
\end{lemma}

\begin{proof}
We proceed similarly to the proof of \cite[Lemma~4.4]{max}.  By density of $C^1$ in $\cB_w$, it suffices
to prove the inequality for $f \in C^1$.  Then according to our convention, for $\psi \in C^\alpha(\cW^s_{\bH})$,
\[
\cL_t^n(f\psi)(1) = \int_M \cL_t^n(f\psi) \, d\musrb \, .
\]
To estimate the integral, we disintegrate $\musrb$ into a family of conditional
probability measures on stable manifolds as follows.
Fix a foliation of $\cF = \{ W_\xi \}_{\xi \in \Xi} \subset \cW^s_{\bH}$ of maximal,  homogeneous local stable manifolds
belonging to $\cW^s_{\bH}$.  The conditional measures are defined by
$\ximusrb = |W_\xi|^{-1} \rho_\xi dm_{W_\xi}$, where $\rho_\xi$ satisfies
\cite[Cor 5.30]{chernov book},
\begin{equation}
\label{eq:rho}
0 < c_\rho \le \inf_{\xi \in \Xi} \inf_{W_\xi} \rho_\xi \le \sup_{\xi \in \Xi} |\rho_\xi|_{C^\alpha(W_\xi)}
\le C_\rho < \infty \, .
\end{equation}
We denote the factor measure on the index set $\Xi$ 
by $\hatmusrb$.
Then,
\[
\begin{split}
\left| \int_M \cL_t^n(f\psi) \, d\musrb \right| & = \left| \int_{\Xi} \int_{W_\xi} \cL_t^n(f\psi) \, \rho_\xi \, dm_{W_\xi} |W_\xi|^{-1} d\hatmusrb(\xi) \right| \\
& =  \left| \int_{\Xi} \sum_{W_i^\xi \in \cG_n(W_\xi)} \int_{W_i} f \psi \, \rho_\xi \circ T^n \, |J^sT^n|^t \, dm_{W_i^\xi} |W_\xi|^{-1} d\hatmusrb(\xi) \right| \\
& \le C C_\rho |f|_w |\psi|_{C^\alpha(\cW^s_{\bH})} \left| \int_{\Xi} \sum_{W_i^\xi \in \cG_n(W_\xi)} |J^sT^n|^t_{C^0(W_i^\xi)}  |W_\xi|^{-1} d\hatmusrb(\xi) \right| \\
& \le C C_\rho C_2 Q_n(t) |f|_w |\psi|_{C^\alpha(\cW^s_{\bH})} \int_{\Xi} |W_\xi|^{-1} d\hatmusrb(\xi) \, ,
\end{split}
\]
where in the last line we have used Lemma~\ref{lem:extra growth} with $\varsigma=0$.  The remaining integral
is finite by \cite[Exercise~7.22]{chernov book} since the family $(W_\xi, d\ximusrb, \hatmusrb)_{\xi \in \Xi}$ is
a standard family.
\end{proof}

\begin{proposition}[Constructing $\mu_t$]\label{prop:mut def}
For $\nu \in \cB$ and $\tilde \nu \in \cB^*$ we set $\langle \nu, \tilde \nu \rangle :=\tilde \nu(\nu)$. 
  \begin{itemize}
    \item[a)] The measure $\tnu_t \in \cB^*$ is in fact an element of $\cB_w^*$.
 \item[b)] We have $\langle \nu_t, \tnu_t \rangle \neq 0$, and the distribution $\mu_t$ defined 
  for $\psi \in C^\alpha(\cW^s_{\bH})$ by
$$\mu_t(\psi) := \frac{ \langle \psi \nu_t,  \tnu_t \rangle}{\langle \nu_t , \tnu_t \rangle}$$
is a $T$-invariant probability measure.
   \end{itemize}
\end{proposition}

\begin{proof}
a) Let $g_n = n^{-1} \sum_{k=0}^{n-1} e^{-P_*(t) k} (\cL_t^*)^k d\musrb$.  By definition,
$\| g_n - \tnu_t  \|_{\cB^*} \to 0 $ as $n \to \infty$.  Thus for $f \in \cB$, we have
\[
| \langle f , \tnu_t \rangle | \le | \langle f, \tnu_t - g_n \rangle | + | \langle f, g_n \rangle | 
\le | \langle f, \tnu_t - g_n \rangle |  + C|f|_w \, ,
\] 
where for the last inequality, we used the bound,
\[
|\langle f, (\cL_t^k)^* d\musrb \rangle| = |\langle \cL_t^k f, d\musrb \rangle|
\le C |\cL_t^k f|_w \le C' e^{P_*(t) k} |f|_w \, ,
\]
by Lemma~\ref{lem:distrib}
and Proposition~\ref{prop:ly}.
Taking $n \to \infty$ yields the bound $|\langle f , \tnu_t \rangle| \le C |f|_w$ for all
$f \in \cB$ and since $\cB$ is dense in $\cB_w$, the distribution $\tnu_t$ extends to a bounded
linear operator on $\cB_w$, as required.

\smallskip
\noindent b) 
First we show the expression  $\langle \psi \nu_t, \tnu_t \rangle$ is well-defined for 
$\psi \in C^\alpha(\cW^s_{\bH})$.   According to our convention, for $f \in C^1(M)$, 
we define for $n \ge 0$,
\[
\langle f, \psi (\cL_t^n)^*d\musrb \rangle = \int \cL_t^n(f\psi) \, d\musrb
\le C Q_n(t) |f|_w |\psi|_{C^\alpha(\cW^s_{\bH})} \, ,
\]
by Lemma~\ref{lem:Lt distr}. 
Thus $\psi (\cL_t^n)^*d\musrb$ extends to
a bounded linear functional on $\cB_w$.  Applying Proposition~\ref{prop:summary}(b)  and
\eqref{eq:dual}, we obtain
\begin{equation}
\label{eq:weak dual}
\psi \tnu_t \in \cB_w^* \; \mbox{ with } \; | \langle f, \psi \tnu_t \rangle | 
\le C' |f|_w |\psi|_{C^\alpha(\cW^s_{\bH})},
\; \; \forall f \in \cB_w \, .
\end{equation}
(We do not claim or need that $\psi f \in \cB_w$, i.e.
that $\psi f$ can be approached by a sequence of 
$C^1$ functions in the weak norm.) 
Thus $\langle \psi \nu_t , \tnu_t \rangle := \langle \nu_t, \psi \tnu_t \rangle$ is well-defined.
Remark that the above argument also shows that $\mu_t(f \psi) = \langle f \nu_t, \psi \tnu_t \rangle$
for all $f \in C^1(M)$, $\psi \in C^\alpha(\cW^s_{\bH})$.

 Next, suppose $\langle \nu_t, \tnu_t \rangle = 0$.  Then for any $f \in C^1(M)$, and $n \ge 1$,
using \eqref{eq:spec},
\begin{equation}
\label{eq:zero}
\begin{split}
\langle f, \tnu_t\rangle & = \frac 1n \sum_{k=0}^{n-1} e^{-P_*(t) k}  \langle f, (\cL_t^*)^k \tnu_t \rangle
= \frac 1n \sum_{k=0}^{n-1} e^{- P_*(t) k } \langle \cL_t^k f, \tnu_t \rangle \\
& \xrightarrow[n\to\infty]{}  \langle \Pi_0(f), \tnu_t \rangle 
= c_t(f) \langle \nu_t , \tnu_t \rangle = 0 \, .
\end{split}
\end{equation}
By density of $C^1(M)$ in $\cB$, this implies that $\tnu_t = 0$ as an element of $\cB^*$,
a contradiction.  Thus $\langle \nu_t, \tnu_t \rangle \neq 0$, and indeed
$c_t(f) = \frac{\langle f, \tnu_t \rangle}{\langle \nu_t, \tnu_t \rangle}$, so that 
$\mu_t$ is a well-defined
element of $\cB^* \subset (C^1(M))^*$.
It is then easy to see that $\mu_t$ is a nonnegative
distribution and thus a Radon measure. The fact that $\mu_t$ is $T$-invariant
is an  exercise, using that $\nu_t$ and $\tnu_t$ are eigenvectors of $\cL_t$
and $\cL_t^*$.
\end{proof}

Following \cite[Definition~7.5]{max}, we remark that elements of $\cB$ and $\cB_w$
can be viewed both as distributions on $M$, as well as families
of leafwise distributions on stable manifolds.  In particular, for $f \in C^1(M)$, 
$W \in \cW^s$, the map defined by 
\[
\cD_{W, f}(\psi) := \int_W f \, \psi \, dm_W \, , \quad \psi \in C^\alpha(W) \, ,
\]
can be viewed as a distribution of order $\alpha$ on $W$.  Since
$|\cD_{W, f}(\psi)| \le |f|_w |\psi|_{C^\alpha(W)}$, the map $\cD_{W, \cdot}$ can be extended
to all $f \in \cB_w$.  We will use the notation $\int_W \psi \, f$ for this extension and
call the associated family of distributions the leafwise distributions $(f, W)_{W \in \cW^s}$
corresponding to $f$.  If $f$ satisfies $\int \psi \, f \ge 0$ for all $\psi \ge 0$, then 
the leafwise distribution is a leafwise measure.

Recalling the disintegration of the measure $\musrb$ from the proof of Lemma~\ref{lem:Lt distr},
we state  an analogue of \cite[Lemma~7.7]{max}:

\begin{lemma}[$\nu_t$ as a leafwise measure]
\label{lem:equivalent}
Let $\nu_t^\xi$ and $\hat{\nu}_t$ denote the conditional measure on $W_\xi$ and the
factor measure on $\Xi$, respectively, obtained by disintegrating $\nu_t$ on the foliation
of stable manifolds $\cF$.  For all $\psi \in C^\alpha(M)$,
\[
\int_{W_\xi} \psi \, d\nu_t^\xi = \frac{\int_{W_\xi} \psi \, \rho_\xi \, \nu_t}{\int_{W_\xi} \rho_\xi \, \nu_t}
\quad \forall \xi \in \Xi, \mbox{ and } \quad d\hat\nu_t(\xi) = |W_\xi|^{-1} 
\left( \int_{W_\xi} \rho_\xi \, \nu_t \right) d\hatmusrb(\xi) \, .
\]
Moreover, viewed as a leafwise measure, $\nu_t(W)>0$ for all $W \in \cW^s_{\bH}$.
\end{lemma}

\begin{proof}
We begin by showing that $\nu_t(W)>0$ for all $W \in \cW^s_{\bH}$.
If $|W| \ge \delta_1/3$ (recalling that our choice of $\delta_1$ 
in \eqref{eq:delta11} is uniform for $t\in [t_0,t_1]$), then the positivity follows immediately from the uniform lower bound
\eqref{eq:lower evol}.  So assume $W \in \cW^s_{\bH}$ with $|W| < \delta_1/3$.

First, we claim that there exists $n_W = \mathcal{O}(\log |W|)$ such that 
at least one element of $\cG_{n_W}^{\delta_1}(W)$ has length at least $\delta_1/3$.
For any $n \ge 1$, if no elements $W_i \in \cG_n^{\delta_1}(W)$ have length at 
least $\delta_1/3$, then
$\cG_n^{\delta_1}(W) = \cI_n^{\delta_1}(W)$ so that by 
Lemma~\ref{lem:short growth} with $\varsigma=0$,
$\sum_{W_i \in \cG_n^{\delta_1}(W)} |J^sT^n|_{C^0(W_i)} \le C_0 \theta^n$, while by \eqref{eq:grow},
$\sum_{W_i \in \cG_n^{\delta_1}(W)} |J^sT^n|_{C^0(W_i)} \ge C_1 |W| \delta_1^{-1}$.  This can continue
only so long as
\[
C_1 |W| \delta_1^{-1} \le C_0 \theta^n \implies n \le \frac{ \log \big( \frac{C_1|W|}{C_0\delta_1} \big)}{\log \theta} =: n_W \, .
\]

Next, letting $V \in \cG_{n_W}^{\delta_1}(W)$ be such that $|V| \ge \delta_1/3$, we estimate as in 
\eqref{eq:lower}, using the fact that since $V$ and $T^{n_W}V$ are both homogeneous,
$|V| \le C |T^{n_W}V|^{\big( \frac{q+1}{2q+1} \big)^{n_W} }$ for some $C \ge 1$, so that
\[
|J^sT^{n_W}|^t_{C^0(V)} \ge e^{-t C_d} \frac{|T^{n_W}V|^t}{|V|^t} \ge e^{-t C_d}
(\delta_1/(3C))^{t \big( \frac{2q+1}{q+1} \big)^{n_W} } \, .
\]
Finally, recalling our choice of $n_1$ from \eqref{eq:delta111},  and using the fact that $|V| \ge \delta_1/3$, we estimate,\footnote{In fact, estimating more carefully for $t \le 1$, one 
can obtain the more precise lower bound $C' \delta_1^{1-t} |W|^{C'' P_*(t)} |W|^t$ for some $C', C'' >0$,
but we will not need this here.}
\begin{equation}
\label{eq:lower weight W}
\begin{split}
\frac 1n \sum_{k=0}^{n-1} e^{- P_*(t)k} &  \int_W  \cL_t^k 1 \, dm_W
 \ge e^{- P_*(t) n_W} \frac 1n \sum_{k = n_1 + n_W}^{n-1}  e^{- P_*(t) (k-n_W)} \int_V \cL_t^{k-n_W} 1 \, |J^sT^{n_W}|^t \, dm_V \\
& \ge e^{- P_*(t) n_W} e^{-t C_d}
\big( \tfrac{\delta_1}{3C} \big)^{t \big( \frac{2q+1}{q+1} \big)^{n_W} }
\frac 1n \sum_{k = n_1 + n_W}^{n-1}  e^{- P_*(t) (k-n_W)} \int_V \cL_t^{k-n_W} 1  \, dm_V \\
& \ge  e^{- P_*(t) n_W} e^{-t C_d}
\big( \tfrac{\delta_1}{3C} \big)^{t \big( \frac{2q+1}{q+1} \big)^{n_W} }
\tfrac{n-1-n_1-n_W}{n} \tfrac 14 \delta_1 2^{-t} c_1 \, ,
\end{split}
\end{equation}
where in the last line we have applied \eqref{eq:lower evol} and Proposition~\ref{prop:summary}(b).

These lower bounds  depend only on $|W|$ and carry over to
$\nu_t(W)$ since they are uniform in $n$.

With the lower bounds established, the remainder of the proof follows precisely as in
\cite[Lemma~7.7]{max}, disintegrating the measure 
$\left( \frac 1n \sum_{k=0}^{n-1} e^{-kP_*(t) } \cL_t^k 1 \right) d\musrb$ 
on the foliation of stable manifolds $\cF$ from \eqref{eq:rho},  using  that convergence in $\cB$ 
to $\nu_t$
implies convergence of the integral on each $W_\xi \in \cF$.  The
lower bounds on $\nu_t(W)$ imply that the ratio 
$\frac{\int_{W_\xi} \psi \, \rho_\xi \, \nu_t}{\int_{W_\xi} \rho_\xi \, \nu_t}$
is well-defined for each $W_\xi \in \cF$.
\end{proof}

In view of \eqref{eq:twist}
in the proof of Lemma~\ref{lem:gap} below (and also \eqref{forCLT4}), it is convenient to
define $\cL_t$ acting explicitly on distributions.
For any point $x \in M$ that has a stable manifold of zero length,
we define $W^s(x) = \{ x \}$, and extend $\cW^s$ to a larger collection $\wW^s$ 
including these singletons.  
For $\alpha \le 1$,
let 
\[
C^\alpha(\wW^s):=\{ 
\psi \mbox{ bounded and measurable }\mid | \psi |_{C^\alpha(\wW^s)} := \sup_{W \in \wW^s} |\psi|_{C^\alpha(W)} < \infty\}  \, .
\]
Let $C^\alpha_{\cos}(\wW^s)$ denote the
set of measurable functions $\psi$ such that $\psi \cos \vf \in C^\alpha(\wW^s)$.
It follows from the uniform hyperbolicity
of $T$ that if $\psi \in C^\alpha(\wW^s)$, then $\psi \circ T \in C^\alpha(\wW^s)$ 
(see \eqref{eq:test contract}).
Also, as in the proof of Lemma~\ref{lem:comparable}, by \cite[eq. (5.14)]{chernov book}, we have
$J^sT(x) \approx \cos \vf(x)$ for $x \in M'$.   We extend $J^sT$ to all $x \in M$
by defining it to be 1 on $M \setminus M'$.  Then using \eqref{eq:distortion}, we have
$\psi \circ T/|J^sT|^{1-t} \in C^\alpha_{\cos}(\wW^s)$
whenever $\psi \in C^\alpha(\wW^s)$ and $\alpha \le 1/(q+1)$.
Using these facts, for a distribution $\mu \in (C^\alpha_{\cos}(\wW^s))^*$, define 
$\cL_t : (C^\alpha_{\cos}(\wW^s))^* \to (C^\alpha(\wW^s))^*$ by
\begin{equation}\label{distr0}
\cL_t\mu(\psi) = \mu \left( \frac{\psi \circ T}{|J^sT|^{1-t}} \right) \, , \quad
\mbox{for all } \psi \in C^\alpha(\wW^s).
\end{equation}
To reconcile this definition with \eqref{eq:L}, for $f \in C^\alpha(\cW^s)$, we identify
$f$ with the measure $f d\musrb$. Such a measure belongs to 
$(C^\alpha_{\cos}(\wW^s))^*$ since $1/\cos \vf \in L^1(\musrb)$. 
With this convention, the measure $\cL_t f$ has density
with respect to $\musrb$ given by \eqref{eq:L}.
Finally, 
note that $\cB\subset (C^\alpha_{\cos}(\wW^s))^*$,
due to Lemma~\ref{lem:distrib} and Remark~\ref{LebB}.

\smallskip

 We are finally ready to prove that $\cL_t$ enjoys a spectral gap,
using Lemma \ref{lem:equivalent} (which exploited that $\musrb$ has smooth \emph{stable} conditional densities, a very nongeneric property
in the setting of hyperbolic dynamics).

\begin{lemma}[Spectral Gap]
\label{lem:gap}
$\cL_t$ has a spectral gap on $\cB$, i.e., $e^{P_*(t)}$ is a simple eigenvalue and all other 
eigenvalues of $\cL_t$ have modulus strictly less than $e^{P_*(t)}$.
\end{lemma}

\begin{proof}
 \emph{ Step 1: the spectrum of $e^{-P_*(t)} \cL_t$ consists of finitely many cyclic groups;} 
in particular, each $\varpi$ in \eqref{eq:decomp} is rational. 
To prove this, 
suppose $\nu \in \bV_\varpi$, $\nu \neq 0$, and $\psi \in C^\alpha(M)$.  
Then by  Lemma~\ref{lem:peripheral}(b) 
and viewing $\nu$ as a distribution in the sense of \eqref{distr0}
\begin{equation}
\label{eq:twist}
\begin{split}
\int_M \psi \, f_\nu \, d\nu_t 
& = \nu(\psi) = e^{- P_*(t) -2 \pi i \varpi} \cL_t \nu(\psi)
= e^{- P_*(t) - 2\pi i \varpi} \nu\left( \frac{\psi \circ T}{|J^sT|^{1-t}} \right) \\
& = e^{- P_*(t) - 2\pi i \varpi} \nu_t \left( f_\nu \frac{\psi \circ T}{|J^sT|^{1-t}} \right)
= e^{- P_*(t) - 2\pi i \varpi} \cL_t\nu_t \left(\psi \, f_\nu \circ T^{-1} \right) \\
& = e^{- 2\pi i \varpi} \nu_t \left(\psi \, f_\nu \circ T^{-1} \right) \, ,
\end{split}
\end{equation}
so that $f_\nu \circ T^{-1} = e^{2 \pi i \varpi} f_\nu$, $\nu_t$-almost everywhere.

Defining $\nu_{k,t} = (f_\nu)^k \nu_t$, for $k \in \mathbb{N}$,  we claim that
$e^{P_*(t) + 2\pi i \varpi k}$ belongs to the spectrum of $\cL_t$ and $\nu_{k,t} \in \bV_{\varpi k}$.  
The claim completes the proof
of  Step 1 since the peripheral spectrum is finite, forcing $\varpi k = 0$ (mod 1) for some $k \ge 1$,
so that $\varpi$ must be rational.

To prove the claim, set $f_\nu = 0$ outside the support of $\nu_t$, and define
the measure $\langle f_\nu \nu_t, \cdot \, \tnu_t \rangle = \langle \nu, \cdot \, \tnu_t \rangle$.
We claim that this measure is not identically zero.  If it were, 
then for any $\psi \in \cB^*$, making the dual argument to \eqref{eq:zero},
\[
\langle \nu, \psi \rangle = \langle \Pi_\varpi \nu, \psi \rangle
= \langle \nu, \Pi_\varpi^* \psi \rangle = \langle \nu , \tilde f_\varpi \tnu_t \rangle \tilde c_\varpi(\psi) =0 \, ,
\]
where we have used that every eigenvector corresponding to the peripheral spectrum
of $\cL_t^*$ is absolutely continuous with respect to $\tnu_t$, i.e.
$\tnu_\varpi = \tilde f_\varpi \tnu_t$, as explained after \eqref{eq:dual}. Thus
$\nu = 0$, a contradiction.

Since $\langle f_\nu \nu_t, \cdot \, \tnu_t \rangle$ is not identically zero, it follows that
$\langle (f_\nu)^k \nu_t, \cdot \, \tnu_t \rangle$ is not identically zero.  Thus there exists
$\psi \in C^\alpha(M)$ such that $\langle (f_\nu)^k \nu_t, \psi \tnu_t \rangle \neq 0$.

For $\ve >0$, choose $g \in C^1(M)$ such that $\mu_t(|g - (f_\nu)^k|) < \ve$.  Note that
$g \nu_t \in \cB$ { by \cite[Lemma~5.3]{dz2}.}  
We will show that $\Pi_{\varpi k}(g\nu_t) \neq 0$.  For 
$\psi \in C^\alpha(M)$ and each $j \ge 0$,
\begin{align*}
e^{-P_*(t) j - 2 \pi i \varpi k j} \langle \cL_t^j(g \nu_t), \psi \tnu_t \rangle
& = e^{-P_*(t) j - 2 \pi i \varpi k j} \langle g \nu_t, \psi \circ T^j (\cL_t^*)^j \tnu_t \rangle \\
& = e^{- 2 \pi i \varpi k j} \langle \nu_t, \tnu_t \rangle \, \mu_t(g \, \psi \circ T^j ) \, ,
\end{align*}
where we have used $(\cL_t^*)^j \tnu_t = e^{P_*(t) j} \tnu_t$.  Also, due to the invariance
of $\mu_t$,
\[
\langle (f_\nu)^k \nu_t, \psi \tnu_t \rangle
= e^{- 2\pi i \varpi k j} \langle (f_\nu)^k \circ T^{-j} \, \nu_t, \psi \tnu_t \rangle
= e^{- 2\pi i \varpi k j} \langle \nu_t, \tnu_t \rangle \, \mu_t( (f_\nu)^k \, \psi \circ T^j) \, .
\]
Putting these two expressions together, we estimate,
\begin{align*}
& \left| \lim_{n \to \infty} \frac 1n \sum_{j=0}^{n-1} e^{-P_*(t) j - 2 \pi i \varpi k j} \langle \cL_t^j(g \nu_t), \psi \tnu_t \rangle - \langle (f_\nu)^k \nu_t, \psi \tnu_t \rangle \right| \\
& \le \lim_{n \to \infty} \frac 1n \sum_{j=0}^{n-1}  \langle \nu_t, \tnu_t \rangle \,
 \mu_t(|g - (f_\nu)^k|) |\psi|_\infty \le \ve  \langle \nu_t, \tnu_t \rangle |\psi |_\infty \, .
\end{align*}
Since $\langle (f_\nu)^k \nu_t, \cdot \, \tnu_t \rangle \neq 0$ and $\ve>0$ was arbitrary, 
this estimate shows that
(i) $\Pi_{\varpi k} (g \nu_t) \neq 0$, so that $\bV_{\varpi k}$ is not empty, 
and (ii) $\nu_{k,t} = (f_\nu)^k \nu_t$ can be approximated by elements of $\bV_{\varpi k}$, and
so must belong to $\bV_{\varpi k}$, as claimed.

\medskip
\noindent
\emph{Step 2:  $\cL_t$ has a spectral gap.}
It suffices to show that the ergodicity of $(T, \musrb)$ implies that the positive eigenvalue 
$e^{P_*(t)}$ is simple.  
For then applying  Step 1, suppose
$\nu \in \bV_\varpi$ for $\varpi = a/b$.  Then both $\cL_t^b \nu = e^{P_*(t) b} \nu$ and
$\cL_t^b \nu_t = e^{P_*(t) b}\nu_t$, so that $\cL_t^b$ has eigenvalue 
$e^{P_*(t) b}$ of multiplicity 2, and this is also its spectral radius, contradicting the 
fact that $(T^b, \musrb)$ is also ergodic.

Now, suppose $\nu \in \bV_0$.  
By Lemma~\ref{lem:peripheral}(b), there exists $f_\nu \in L^\infty(\nu_t)$
such that $d\nu = f_\nu d\nu_t$. We will show that $f_\nu$ is $\nu_t$-a.e. a constant.

By \eqref{eq:twist} $f_\nu \circ T = f_\nu$, $\nu_t$-a.e. so that
setting
\[
S_nf_\nu = \sum_{j=0}^{n-1} f_\nu \circ T^j \, ,
\]
we see that $\frac 1n S_nf_\nu = f_\nu$ for all $n \ge 1$.  Thus $f_\nu$ is constant on stable
manifolds.   Next, since the factor measure $\hat{\nu}_t$ is equivalent to $\hatmusrb$ on the
index set $\Xi$ by Lemma~\ref{lem:equivalent}, it follows that $f_\nu = f_\nu \circ T$ on 
$\hatmusrb$-a.e. $W_\xi \in \cF$.  So $f_\nu = f_\nu \circ T$, $\musrb$-a.e.  
Since $\musrb$ is ergodic, $f_\nu$ is constant $\musrb$-a.e.  But 
 since $f_\nu$ is constant
on each stable manifold $W_\xi \in \cF$, it follows that there exists $c>0$ such that
$f_\nu = c$ for $\hatmusrb$-a.e. $\xi \in \Xi$, and 
once again using the equivalence of 
$\hatmusrb$ and $\hat{\nu}_t$, we conclude that $f_\nu$ is constant $\nu_t$-a.e.
\end{proof}

\begin{proof}[Proof of Theorem~\ref{thm:spectral}]
All claims except the last sentence of the theorem follow from Propositions~\ref{prop:radius} and ~\ref{prop:mut def},  and  Lemma~\ref{lem:gap}.
Exponential decay of correlations for $C^\alpha$ functions with rate $\upsilon$ satisfying
\eqref{uups} follows from the  classical
spectral decomposition 
$$\cL^k_tf=e^{kP_*(t)}[{ c_t(f) } \cdot \nu_t
+ \cR^k_t(f)]
\, , \, \mbox{where } \exists C <\infty \mbox{ s. t. }\|\cR^k_t f\| < C \upsilon^k \|f\|\, , 
\forall k\ge 0\, , \,\forall f \in\cB\, ,
$$
{ and $c_t(f) = \frac{\langle f, \tnu_t \rangle}{\langle \nu_t, \tnu_t \rangle}$.} 
Indeed,  by \cite[Lemma~5.3]{dz2} 
for $\psi \in C^\alpha(M)$,
\begin{equation}
\label{strong}
\psi \circ T^{ -j} f \in \cB \; \mbox{ and } \; \| \psi \circ T^{ -j} f \|_{\cB} \le C_j |\psi|_{C^\alpha}
\|f\|_{\cB} \quad \mbox{for all { $j \ge 1$}} \, ,
\end{equation}
we find  for $f_1, f_2 \in C^\alpha(M)$ { (using \eqref{strong} with $j=k$)}, 
\begin{align*}
\int (f_1\circ T^k)& f_2 d\mu_t = \frac{\langle  (f_1 \cdot f_2 \circ T^{-k}  \nu_t ,\tnu_t\rangle}{ \langle \nu_t, \tnu_t \rangle} =
e^{-k P_*(t)} \frac{\langle f_1 \cL^k_t(f_2  \nu_t), \tnu_t \rangle}{ \langle \nu_t, \tnu_t \rangle} \\
&=  { c_t(f_2\nu_t)} \frac{\langle f_1 \nu_t , \tnu_t \rangle}{ \langle \nu_t, \tnu_t \rangle}
+ \frac{\langle f_1 \cR^k_t(f_2  \nu_t), \tnu_t \rangle}{ \langle \nu_t, \tnu_t \rangle}
= \int f_1 d\mu_t \int f_2 d\mu_t +
\frac{\langle f_1 \cR^k_t(f_2  \nu_t), \tnu_t \rangle}{ \langle \nu_t, \tnu_t \rangle}\, ,
\end{align*}
and we have, using again \eqref{strong} (with { $j=0$}),
$$
\biggl |\langle f_1 \cR^k_t(f_2  \nu_t), \tnu_t \rangle \biggr |
\le |f_1|_{C^\alpha} \| \cR^k_t(f_2  \nu_t)\|
\le  C |f_1|_{C^\alpha}  \upsilon^k \|f_2  \nu_t\|\le  C  |f_1|_{C^\alpha} |f_2|_{C^\alpha} \upsilon^k \, .
$$
Exponential mixing for H\"older functions of  exponent smaller than $\alpha$ then follows from mollification  (a lower exponent may worsen the rate of mixing).
Finally, mixing is obtained by a standard argument: Since
$\mu$ is a Borel probability measure and $M$ is a compact metric space (and thus a normal topological space), any 
$f\in L^2(\mu)$ can be approximated by a sequence of continuous functions 
in the $L^2(\mu)$ norm, using Urysohn functions. So, by Cauchy--Schwartz, we may
reduce to proving mixing for continuous test functions.
Clearly, Lipschitz functions form a subalgebra of the Banach algebra of continuous functions,
the constant function $\equiv 1$ is Lipschitz, and  for any $x\ne y$ in $M$
there exists a Lipschitz function $\tilde f$ with $\tilde f(x)\ne \tilde f(y)$.
 Since
$M$ is a compact metric space   the Stone--Weierstrass theorem
 implies that any continuous function  on $M$ can be approached  in
the supremum norm
by a sequence of Lipschitz  functions on $M$. 
Since we have proved mixing for Lipschitz functions, the proof is concluded.
\end{proof}


\section{Final Properties of $\mu_t$ 
(Proof of Theorem~\ref{thm:equil} and Theorem~\ref{thm:geo var})}
\label{finall}

In this section we show
Proposition~\ref{nbhdprop}, Corollary~\ref{cor:max},
 Proposition~\ref{zerolim},  Lemma~\ref{lem:full}, and Proposition~\ref{prop:unique},  which, together
with Theorem~\ref{thm:spectral}, give  Theorem~\ref{thm:equil}.

\subsection{Measuring Neighbourhoods of Singularity Sets -- $\mu_t$ is $T$-adapted}

In this section, we show
Proposition~\ref{nbhdprop}, which gives in particular that $\mu_t$ is $T$-adapted.
For any $\ve >0$ and any  $A \subset M$, 
we set $\cN_\ve(A)=\{ x \in M \mid d(x,A) < \epsilon\}$.
The proof will be based on controlling the
measure of small neighbourhoods of singularity sets.

\begin{proposition}\label{nbhdprop}
Let $\mu_t$ be given by  Theorem~\ref{thm:spectral} for $t\in { [t_0,t_1]}$, with  $p>2$ the norm parameter.
\begin{itemize}
  \item[a)]  For any  $C^1$ curve $S$ uniformly transverse to the stable cone,  there exists $C>0$ such that 
$\mu_t(\cN_\ve(S)) \le C  \ve^{1/p}$ for all $\ve > 0$.
\item[b)]
The measure    $\mu_t$ has no atoms. We have $\mu_t(\cS_n)=0$
for any $n \in \mathbb{Z}$, and  $\mu_t(W)=0$
for any local stable or unstable manifold $W$.
\item[c)] The measure $\mu_t$ is adapted, i.e.,
$\int | \log d(x, \cS_{\pm 1})| \, d\mu_t < \infty
$.
\item[d)]
 For any $p'>2p$, $\mu_t$-almost every $x$ and each $n \in \mathbb{Z}$,  there exists $C>0$ such that 
\begin{equation}\label{lower}
d(T^jx, \cS_n) \ge C j^{-p'}\, , \forall j\ge 0\, .
\end{equation}
\item[e)] $\mu_t$-almost every $x \in M$ has stable and unstable
manifolds of positive lengths.
\end{itemize}
\end{proposition}

\begin{proof}
We proceed as in \cite[Corollary 7.4]{max}.
The key fact is that  for any $n \in \mathbb{N}$ there exists $C_n <\infty$
such that for all $\ve > 0$ 
\begin{equation}
\label{nbhd}
\mu_t(\cN_\epsilon(\cS_{-n})) <  C_n \ve^{1/p} 
\, , \quad \mu_t(\cN_\epsilon(\cS_{n})) <  C_n \ve^{1/(2p)}  \, .
\end{equation}
 Denoting by $1_{n ,\ve}$
 the indicator function of the set $\cN_{\ve}(\cS_{-n})$,
  Proposition~\ref{prop:mut def}(a) implies
 \[
 \mu_t(\cN_{\ve}(\cS_{-n})) = \langle 1_{n , \ve} \nu_t, \tnu_t \rangle \le C |1_{n, \ve} \nu_t |_w \, ,
 \]
for $n\ge 0$.  The bound
$|1_{n, \ve} f |_w \le A_n \|f\|_s |\ve|^{1/p}$ for all $f \in \cB$ follows
exactly as 
the proof of
\cite[Lemma 7.3]{max}, replacing
the logarithmic modulus of continuity $|\log \ve|^{-\gamma}$ 
in the strong stable norm there
by our H\"older modulus of continuity $\ve^{1/p}$, 
 and using the fact that $\cS_{-n}$ is uniformly transverse to the stable cone.
This proves the first inequality in \eqref{nbhd}.  The second follows from the
invariance of $\mu_t$, together with the fact that 
$T(\cN_\ve(\cS_n)) \subset \cN_{C\ve^{1/2}}(\cS_{-n})$.

Claim a) of the proposition follows  from the proof of \eqref{nbhd}, since the only property required
of $\cS_{-n}$ is that it comprises finitely many smooth curves uniformly transverse to the stable cone.
The bound \eqref{nbhd}  applied to arbitrary stable curves immediately implies that $\mu_t$ has no atoms,
and that $\mu_t(\cS_n)=0$
for any $n \in \mathbb{Z}$.   Next, if we had $\mu_t(W)  >0$ for a local stable manifold, then  $\mu_t(T^nW) >0$
for all $n > 0$.  Since $\mu_t$ is a probability measure and $T^n$ is continuous on stable manifolds, 
$\cup_{n \ge 0} T^nW$ must be the union of  finitely many smooth curves  (indeed, if $\cup_{n \ge 0} T^nW$ comprised
infinitely many smooth curves, then $\mu_t(M) = \infty$ by the invariance of $\mu_t$).  
Since $|T^nW| \to 0$, there is a subsequence $(n_j)$ such that $\cap_{j \ge 0} T^{n_j}W = \{ x \}$.  Thus $\mu_t(\{ x \}) >0$, a contradiction.
  For an unstable manifold $W$, use the fact that $T^{-n}$ is continuous on $W$.
So we have established b).

To show c), choose $p'>2p$.  Then by \eqref{nbhd}
\begin{align*}
\int_{M\setminus \cN_{1}(\cS_1)}  &|\log d(x, \cS_1)| \, d\mu_t 
 = \sum_{j \ge 1} \int_{\cN_{j^{-p'}}(\cS_1)
\setminus \cN_{(j+1)^{-p'}}(\cS_1)} |\log d(x, \cS_1)| \, d\mu_t \\
 & \le p'\sum_{j \ge 1} \log (j+1) \cdot \mu_t(\cN_{j^{-p'}}(\cS_1))
\le p' C_1 \sum_{j \ge 1} 
\log (j+1) \cdot j^{-p'/(2p)} < \infty .
\end{align*}
A similar estimate holds for $\int |\log d(x, \cS_{-1})| \, d\mu_t$.

Next, fix $\eta>0$, $p' > 2 p$ and $n \in \mathbb{Z}_+$.  Since
both sums
\begin{equation}
\label{eq:borel}
\sum_{j \ge 1} \mu_t(\cN_{\eta j^{-p'}}(\cS_{-n})) \le  \tilde C
{ C_{-n} \eta^{\frac 1p}} \sum_{j \ge 1}  j^{-\frac {p'} p}\, ,\, \,\,  
\sum_{j \ge 1} \mu_t(\cN_{ \eta j^{-p'}}(\cS_n)) \le  \tilde C
 { C_n \eta^{\frac {1}{2p}}} \sum_{j \ge 1}  { j^{-\frac {p'} {2p}}} \, , 
\end{equation}
are finite,
the Borel--Cantelli Lemma implies that $\mu_t$-almost every $x \in M$ visits  
$\cN_{\eta j^{-p'}}(\cS_n)$ only
finitely many times. 
This gives \eqref{lower} and thus claim d).
Finally, the existence of nontrivial stable and unstable manifolds 
claimed in e) follows from the Borel--Cantelli estimate
\eqref{eq:borel} by a standard argument, choosing $p'>2p$ and  $\eta \ge 1$
such that $\Lambda^j > \eta^{-1} j^{p'}$ for all $j$ (see \cite[Sect. 4.12]{chernov book}).
\end{proof}


\subsection{
$\mu_t$ is an Equilibrium State. Variational Principle for $P_*(t)$}
\label{BoBa}

For $\epsilon >0$, $x\in M$, and $n\ge 1$ denote by $B_n(x,\epsilon)$ the dynamical (Bowen) ball   for
 $T^{-1}$:
\begin{equation}\label{Bowen}
B_n(x,\epsilon)=\{ y \in M\mid d(T^{-j} (y), T^{-j} (x)) \le \epsilon \, , \, \, \forall \, 0 \le j \le n\} \, .
\end{equation}

\begin{proposition}[Upper Bounds on the Measure of Dynamical Balls]
\label{prop:max}  
Let $t_0\in (0,1)$ and $t_1\in (1,t_*)$.
There exists $A<\infty$ such that for all small enough
$\epsilon >0$, all
$x\in M$, and all $n \ge 1$, the measure 
$\mu_t$ constructed in Theorem~\ref{thm:spectral} for $t\in [t_0,t_1]$
satisfies
\begin{equation}
\label{eq:upper ball}
 \mu_t(\overline{B_n(x,\epsilon) })
\le A e^{-n P_*(t)+t \sum_{k=1}^{ n} \log J^sT( T^{-k} (x))} \, .
\end{equation}
\end{proposition}

\begin{cor}[Equilibrium State for $-t\log J^u$. Variational principle for $P_*(t)$.]
\label{cor:max}
The measure 
$\mu_t$ constructed in Theorem~\ref{thm:spectral} for $t\in (0,t_*)$ 
satisfies 
$P_{\mu_t}(-t\log J^uT) = P_*(t)=P(t)$. 
\end{cor}

\begin{proof}[Proof of Corollary~\ref{cor:max}]
By definition we have $P_{\mu_t}(- t \log J^uT) \le P(t)$, and
Proposition~\ref{prop:pressure} gives
$P(t) \le P_*(t)$, so it is enough to show $P_{\mu_t}(-t \log J^uT) \ge P_*(t)$.
We follow \cite[Cor. 7.17]{max}.
Since 
$\int |\log d(x, \cS_{\pm 1})| \, d\mu_t < \infty$ by Proposition~ \ref{nbhdprop}, and $\mu_t$
is ergodic, we may apply \cite[Prop.~3.1]{DWY} (a slight generalization of the Brin--Katok local theorem \cite{brin},
using \cite[Lemma~2]{M}, 
continuity of the map is not used) 
 to $T^{-1}$. This gives that for $\mu_t$-almost every $x \in M$,
$$ \lim_{\epsilon \to 0} \liminf_{n \to \infty}
  - \tfrac 1n \log \mu_t(B_n(x, \epsilon)) =  \lim_{\epsilon \to 0} \limsup_{n \to \infty}
  - \tfrac 1n \log \mu_t(B_n(x, \epsilon)) = h_{\mu_t}(T^{-1})=h_{\mu_t}(T)\, .
  $$
Using  \eqref{eq:upper ball}  it follows that  for any $\epsilon$ sufficiently small,
$$\limsup_{n \to \infty}  - \tfrac 1n \log \mu_t(B_n(x, \epsilon)) 
\ge P_*(t) - \lim_{n \to \infty} \tfrac tn \sum_{k=1}^{ n} \log J^sT( T^{-k}  (x))
\ge P_*(t) - t \int_M \log J^sT \, d\mu_t \, ,$$ 
for all $\mu_t$-typical $x$.
Thus  applying \eqref{eq:switch potential}, we get $P_{\mu_t}( -t \log J^uT)  \ge P_*(t)$.
\end{proof}

\begin{proof}[Proof of Proposition~\ref{prop:max}]
For $x \in M$ and $n \ge 0$, let $1^B_{n,\epsilon}$ denote the indicator function of 
$B_n(x,\epsilon)$.  
Since  $\nu_t$ is attained as the (averaged) limit of $e^{-n P_*(t)}\cL^n_t 1$ in the weak  (and strong) 
norm and  since we have
$\int_W(\cL^n_t 1)\, \psi dm_W \ge 0$ whenever $\psi \ge 0$, it follows that, viewing $\nu_t$
as a  leafwise distribution,
\begin{equation}
\label{eq:pos}
\int_W \psi \, \nu_t \ge 0, \quad \mbox{ for all $\psi \ge 0$.}
\end{equation}
Then the inequality $|\int_W \psi \,\nu_t| \le \int_W |\psi| \, \nu_t$ implies that the supremum
in the weak norm can be obtained by restricting to $\psi \ge 0$. 
{In addition, for each $n \ge 0$,
\begin{equation}
\label{eq:change var}
\begin{split}
\int_W \psi \, & \cL_t^n \nu_t  = \lim_k e^{-k P_*(t)} \int_W \psi \, \cL_t^n (\cL_t^k 1) \, dm_W \\
& = \lim_k e^{-k P_*(t)} \int_{T^{-n}W} \psi \circ T^n \, \cL_t^k 1 \, |J^sT^n|^t \, dm_{T^{-n}W}
=   \int_{T^{-n}W} \psi \circ T^n \, |J^sT^n|^t \, \nu_t \, ,
\end{split}
\end{equation}
for each $W \in \cW^s$ and $\psi \in C^\beta(W)$.
}

Let $W \in \cW^s$ be a curve intersecting $B_n(x, \epsilon)$, and let $\psi \in C^\alpha(W)$
satisfy $\psi \ge 0$ and $|\psi|_{C^\alpha(W)} \le 1$.  
Then, since $\cL_t \nu_t = e^{P_*(t)} \nu_t$, we have
\begin{equation}
\label{eq:unwrap}
\int_W \psi \, 1^B_{n , \epsilon} \, \nu_t = \int_W \psi \, 1^B_{n, \epsilon} \, e^{-n P_*(t)} \cL^n_t \nu_t 
 = e^{-n P_*(t)} \sum_{W_i \in \cG_n(W)} \int_{W_i} (\psi \circ T^n) \, (1^B_{n,\epsilon} \circ T^n) |J^sT^n|^t \, \nu_t \, .
\end{equation}
In the proof of \cite[Prop. 7.12]{max} we showed
that   $1^B_{n,\ve} f \in \cB_w$ (and $\cB$) for each $f \in \cB$ and $n \ge 0$.  
  In the proof of \cite[Lemma~3.4]{max}, we 
found (using our strong notion of
finite horizon) $\tilde \ve>0$ such that there if $x, y$ lie in different elements of
$\cM_0^n$, then  $ \max_{0 \le i \le n} d(T^ix, T^iy)\ge \tilde \ve$.  
Since $B_n(x, \epsilon)$ is defined with respect to $T^{-1}$, we will use
the time reversal counterpart of this property:  If $\epsilon < \tilde \ve$, we conclude that
$B_n(x,\epsilon)$ is contained in a single component of $\cM_{-n}^0$, i.e.,
$B_n(x, \epsilon) \cap \cS_{-n} = \emptyset$, so that $T^{-n}$ is a diffeomorphism of
$B_n(x, \epsilon)$ onto its image.  Note that $T^{-n}(B_n(x, \epsilon))$ is contained in a single component of $\cM_0^n$,  denoted $A_{n, \epsilon}$. Thus,  $W_i \cap A_{n, \epsilon} = W_i$ for each $W_i \in \cG_n(W)$. 
By \eqref{eq:pos}, 
$$\int_{W_i} (\psi \circ T^n )\, 1_{T^{-n}(B_n(x, \epsilon))}|J^sT^n|^t  \, \nu_t \le \int_{W_i} (\psi \circ T^n)|J^sT^n|^t  \, \nu_t\, .
$$
In the proof of \cite[Prop. 7.12]{max} we observed
that there are at most
two $W_i \in \cG_n(W)$  having nonempty intersection with $T^{-n}(B_n(x, \epsilon))$.
Using these facts together with \eqref{eq:test contract}
and \eqref{eq:JC}
(which implies $||J^sT^n|^t|_{C^\alpha(W_i)}\le C||J^sT^n|^t|_{C^0(W_i)}$), we sum over $W_i' \in \cG_n(W)$ such that
$W_i' \cap T^{-n}(B_n(x, \epsilon)) \neq 0$, to obtain
\[
\int_W \psi \, 1^B_{n,\epsilon} \, \nu_t \le e^{-n P_*(t)} \sum_i \int_{W_i'} (\psi \circ T^n)
 \, |J^sT^n|^t  \, \nu_t
\le 2C e^{-n P_*(t)+ t \sum_{k=0}^{n-1}\log  J^sT( T^{k-n} (x)))} |\nu_t|_w \, ,
\]
where we also used the distortion bounds from  Lemma~\ref{lem:comparable} 
to switch to
$J^sT^n( T^{-n}x)$  since $T^{-n}x$ may not belong to $W_i'$.
This yields  $| 1^B_{n, \ve} \nu_t |_w \le 2C e^{-n P_*(t)  + t \log J^sT^n(T^{-n} x)} |\nu_t|_w$.
Applying  Proposition~\ref{prop:mut def}(a) gives 
\eqref{eq:upper ball}.
\end{proof}

\subsection{Definition of $h_*$. Sparse Recurrence. Proof that $\lim_{t\downarrow 0} P(t)=h_*$}
\label{s:limzero}

In  \cite[Lemma 3.3]{max} we showed that  the limit below exists
$$h_*:= \lim_{n \to \infty} \frac 1 n \log \# \cM_0^n\, .
$$
The number  $h_*$ generalises topological entropy, in particular,
$P(0)\le h_*$ \cite[{ Theorem~2.3}]{max}.

Using $h_*$, we can state the sparse recurrence condition: 

\begin{defin}[Sparse Recurrence to Singularities]\label{sparse} For  $\vf<\pi/2$ and $n \in \mathbb{N}$, define
$s_0(\vf,n) \in (0,1]$ to be the smallest number such that 
any orbit of
length $n$ has at most $s_0n$ collisions whose
angles with the
normal are larger than $\vf$ in absolute value.
We say that $T$ satisfies the sparse recurrence condition if 
there exist $\vf_0<\pi/2$ and $n_0 \in \mathbb{N}$  such that 
$h_*> s_0(\vf_0,n_0) \log 2$.
\end{defin}
We refer to \cite[\S2.4]{max} for a discussion of the sparse recurrence condition. We proved in \cite{max}
that sparse recurrence implies $P(0)=h_*$.
The following proposition connects $h_*$
 to $P_*(t)$ for $t>0$, despite the 
use of different partitions,
$\cM_0^n$ and $\cM_0^{n, \bH}$. 

\begin{proposition}\label{zerolim}
If $T$ satisfies sparse
 recurrence    then $\lim_{t\downarrow 0} P_*(t)=\lim_{t\downarrow 0} P(t)=h_*
$,  and $\lim_{t \downarrow 0} h_{\mu_t} = h_*$. 
\end{proposition}

Assuming the sparse
recurrence condition  \cite[Theorem 2.4]{max} we have  $P(0)=h_*$. So in this case
the function $P(t)$ is continuous on $[0,t_*)$. 
In the general case, we cannot exclude $P(0)<h_*$ even if we can show $\lim_{t\downarrow 0} P(t)={ P(0)}$.

\begin{proof}
Recall that $P(t)=P_*(t)$ for $t\in (0,t_*)$ (using Proposition~\ref{prop:pressure} and Corollary~ \ref{cor:max}).

Showing\footnote{ Note that the limit exists since $P(t) = P_*(t)$ is monotonic.} ${\lim}_{t \downarrow 0} P(t) \le h_*$ does not require the sparse recurrence condition:
Any invariant probability measure
$\mu$ satisfies $\int_M \log J^uT \, d\mu \ge \log \Lambda$ due to \eqref{eq:hyp}.
Also, $h_\mu(T) \le h_*$ by \cite[Theorem~2.3]{max}.  Thus for $t>0$, we have
$
P(t) \le h_* - t \log \Lambda \, ,
$
so that, ${\lim}_{t \downarrow 0} P(t) \le h_*$. 

To prove the lower  bound, assume the sparse recurrence condition,  and
let $\mu_0$ denote the measure of maximal entropy for $T$ constructed in \cite[Theorem~2.4]{max}
(called $\mu_*$ in that paper).  Since $\mu_0$ is T-adapted \cite[Theorem~2.6]{max}, 
the Jacobian $J^uT$ is defined $\mu_0$-almost
everywhere and $\int \log J^uT \, d\mu_0 = \chi^+_{\mu_0} < \infty$.  Thus for $t>0$,
\[
P(t) \ge P_{\mu_0}(-t \log J^uT) = h_{\mu_0} - t \int_M \log J^uT \, d\mu_0 = h_* - t \chi^+_{\mu_0} \, ,
\]
and ${ \lim}_{t \downarrow 0} P(t) \ge h_*$. 

Finally, since $P(t) \le h_{\mu_t} \le h_*$, we must have $\lim_{t \downarrow 0} h_{\mu_t} = h_*$ as well. 
\end{proof}


\subsection{Full Support of $\mu_t$}
\label{sec:full}

It follows from Lemma~\ref{lem:equivalent} that the measure $\nu_t$ is fully supported on $M$.
In this section, we will prove the analogous property for $\mu_t$
combining mixing of the SRB measure and  a direct use of Cantor rectangles, 
bypassing the absolute continuity argument which was used in
\cite[Section~7.3]{max} to show full support of the measure of maximal entropy there. 
Recall the definition of maximal Cantor rectangle $R = R(D)$ comprising the intersection of 
all homogeneous
stable and unstable manifolds {completely}
crossing a solid rectangle $D$ as described in the 
proof of Proposition~\ref{prop:lower bound}.  The boundary of the solid rectangle $D$ comprises
two stable and unstable manifolds which also belong to $R$.
Let $\Xi_R \subset \cW^s$ denote the family of stable manifolds corresponding to $R$ (i.e.~the set
of homogeneous stable manifolds that {completely} cross $D$). 

\begin{lemma}
\label{lem:full}
For any maximal Cantor rectangle $R$, if
$\musrb(\cup_{W \in \Xi_R} W) > 0$ then we also have $\mu_t(\cup_{W \in \Xi_R} W) > 0$.
Consequently, for any nonempty open set $O \subset M$, we have $\mu_t(O) > 0$.
\end{lemma}

\begin{proof}
 Let $\psi \in C^1(M)$ such that $\psi \ge 0$ and
$\psi \equiv 1$ on  $\cup_{W \in \Xi_R}W$.  Due to the spectral decomposition of $\cL_t^*$,
setting $c = \langle \nu_t, \tnu_t \rangle^{-1}$, we have
\begin{equation}
\label{eq:mu_t one}
\mu_t(\psi) = c \lim_{n \to \infty} e^{-n P_*(t)} \langle \psi \nu_t, (\cL_t^*)^n d\musrb \rangle
= c \lim_{n \to \infty} e^{-n P_*(t)} \langle \cL_t^n(\psi \nu_t), d\musrb \rangle \, .
\end{equation}
Then, using the disintegration of $\musrb$, introduced before Lemma~\ref{lem:equivalent}, 
into conditional measures 
on a fixed foliation
$\cF = \{ W_\xi \}_{\xi \in \Xi}$ of stable manifolds, and a transverse measure $\hatmusrb$ on the index
set $\Xi$,
and recalling \eqref{eq:change var}, we estimate for $n \ge 0$,
\begin{align*}
 \langle \cL_t^n(\psi \nu_t), d\musrb \rangle 
& = 
\int_{\Xi} |W_\xi|^{-1} d\hatmusrb(\xi) \int_{W_\xi} \cL_t^n(\psi \nu_t) \, \rho_\xi \\
& =
\int_{\Xi} |W_\xi|^{-1} d\hatmusrb(\xi) 
\sum_{W_i \in \cG_n(W_\xi)} \int_{W_i} \psi \nu_t |J^sT^n|^t \rho_\xi \circ T^n \\
& \ge C 
\int_{\Xi} |W_\xi|^{-1} d\hatmusrb(\xi) 
\sum_{W_i \in \cG_n(W_\xi)} |J^sT^n|^t_{C^0(W_i)} \int_{W_i} \psi \nu_t \, ,
\end{align*}
where in the last line we have used \eqref{eq:rho}, bounded distortion for $J^sT$ and the positivity of
$\nu_t$.  Next, note that if $W_i \in \cG_n(W_\xi)$ properly crosses\footnote{See the proof
of Proposition~\ref{prop:lower bound} for the definition of proper crossing.} $R$, then using again
the positivity of $\nu_t$,  we have
\begin{equation}
\int_{W_i} \psi \, \nu_t \ge   \zeta (\ell_R)  \, ,
\end{equation} 
where $\ell_R$ is the minimum length of a stable manifold in $\Xi_R$ and $\zeta$ is 
a function depending only on $t$ (uniform in $[t_0,t_1]$) and $\delta_1(t_0,t_1)$
from \eqref{eq:delta111} via \eqref{eq:lower weight W}.
Thus letting $\cG_n^R(W_\xi)$ denote those elements of $\cG_n(W)$
that properly cross $R$, we have
\begin{equation}
\label{eq:proper cross}
\langle \cL_t^n(\psi \nu_t), d\musrb \rangle 
\ge C'  \zeta(\ell_R) 
\int_{\Xi} |W_\xi|^{-1} d\hatmusrb(\xi) 
\sum_{W_i \in \cG^R_n(W_\xi)} |J^sT^n|^t_{C^0(W_i)} \, .
\end{equation}

As in the proof of Proposition~\ref{prop:lower bound}, by \cite[Lemma~7.87]{chernov book}, we
choose a finite number of locally maximal homogeneous Cantor rectangles
$\cR(\delta_1) = \{ R_1, \ldots, R_k \}$ 
such that  
there exists $n_* = n_*(\delta_1, R)$ such that  $T^{-n_*}(D(R_i))$ contains 
a homogeneous connected component that properly crosses $R$
for all $i = 1, \ldots, k$.  Therefore, if
$V \in \cW^s$ has $|V| \ge \delta_1/3$, then at least one element of $\cG_{n_*}(V)$ 
properly crosses $R$.
Thus, if $|W_\xi| \ge \delta_1/3$ and $n - n_* \ge n_1$, then using \eqref{eq:delta11}, 
and letting $\delta_1'$ denote the minimum length of a stable manifold belonging
to any of the $R_i$,
\begin{equation}
\label{eq:lower long xi}
\begin{split}
\sum_{W_i \in \cG_n^R(W_\xi)} |J^sT^n|^t_{C^0(W_\xi)}
& \ge e^{-t C_d}  \sum_{W_j \in L^{\delta_1}_{n-n_*}(W_\xi)} |J^sT^{n-n_*}|^t_{C^0(W_j)}
|J^sT^{n_*}|^t_{C^0(W_i)} \\
& \ge \tfrac 34 e^{-t C_d} C (\delta_1')^{t  \big( \frac{2q+1}{q+1} \big)^{n_*} } 
\sum_{W_j \in \cG_{n-n_*}^{\delta_1}(W_\xi)}  |J^sT^{n-n_*}|^t_{C^0(W_j)} \\
& \ge \tfrac 34 e^{-t C_d} C (\delta_1')^{t  \big( \frac{2q+1}{q+1} \big)^{n_*} }  c_1 e^{(n-n_*)P_*(t)} \, ,
\end{split}
\end{equation}
where in the second line we have estimated $J^sT^{n_*}$ from below on $W_i$ as in \eqref{eq:lower}
using the fact that $|W_i| \ge \delta_1'$, and in the third line we have applied 
Propositions~\ref{prop:lower bound} and \ref{prop:exact}.

Substituting \eqref{eq:lower long xi} into \eqref{eq:proper cross} and letting $\Xi^{\delta_1}$ denote
those elements $W_\xi \in \cF$ with $|W_\xi| \ge \delta_1/3$, 
\[
e^{-n P_*(t)} \langle \cL_t^n (\psi \nu_t), d\musrb \rangle \ge C''  \zeta( \ell_R )  \delta_1^{-1} 
(\delta_1')^{t { \big( \frac{2q+1}{q+1} \big)^{n_*} } } e^{-n_*P_*(t)}
\hatmusrb(\Xi^{\delta_1}) \, .
\]
Since this lower bound is independent of $n$, by \eqref{eq:mu_t one} we have
$\mu_t(\psi) >0$, and since this holds for all $\psi \in C^1(M)$ with $\psi \equiv 1$ on 
$\cup_{W \in \Xi_R}W$, the  first statement of the  lemma is proved.
Then  the second statement of the lemma follows from the fact that any nonempty open set $O \subset M$ has a locally
maximal Cantor set $R$ such that $D(R) \subset O$ and $\musrb(R)>0$. 
\end{proof}


\subsection{Uniqueness of Equilibrium State. (Strong) Variational Principle for $P_*(t,g)$}
\label{sec:unique}
In this section, we prove the following uniqueness result:

\begin{proposition}\label{prop:unique}
For any $0<t<t_*$, the measure $\mu_t$ from Theorem~\ref{thm:spectral}
is the unique equilibrium state for $-t \log J^uT$. 
\end{proposition}

The proof of the proposition will give a more general statement (shown at the end of this section):

\begin{theorem}[Strong Variational Principle for
$P_*(t,g)$] \label{thm:variational}
For any $[t_0,t_1]\subset (0, t_*)$ there exists $\upsilon_0>0$
such that for any $C^1$ function $g:M\to \bR$ with 
$|g|_{C^1}\le \upsilon_0$ we have
$$
P_*(t,g)=P(t,g)=\max \{ h_\mu+\int (-t\log J^uT+g)  \, d\mu : \mu \mbox{ a
$T$-invariant probability measure } \}\, , 
$$
and the equilibrium state for $-t\log J^u+g$ is unique.
\end{theorem}

(We restrict to $C^1$
functions $g$ for simplicity. The result also holds H\"older $g$ of suitable exponent.)

Fix $0<t_0<t_1<t_*$.
 For  $\phi \in C^1(M)$, $t \in [t_0,t_1]$, and $\upsilon \in \mathbb{R}$, define the transfer operator 
 $\cL_{t,\upsilon} =\cL_{t,\upsilon,\phi}$ by
\[
\cL_{t,\upsilon}f = \frac{f \circ T^{-1}}{|J^sT|^{1-t} \circ T^{-1}} e^{\upsilon \phi \circ T^{-1}} \, ,
\quad \mbox{for all $f \in C^1(M)$.}
\]
Since $\cL_{t, \upsilon}f = e^{\upsilon \phi \circ T^{-1}} \cL_t f$ and the discontinuities of
$\phi \circ T^{-1}$ are uniformly transverse to the stable cone, \cite[Lemma~5.3]{dz2}
implies that $\cL_{t, \upsilon} f \in \cB$ 
(with $\cB=\cB(t_0,t_1)$ the space for
$\cL_{t}$) and $\| \cL_{t, \upsilon} f \|_{\cB}
\le C \| f \|_{\cB} |e^{\upsilon \phi}|_{C^1}$, so that $\cL_{t, \upsilon}$ defines a
bounded linear operator on $\cB$.  By
\cite[Lemma~6.1]{dz1} 
the map $\upsilon \mapsto \cL_{t,\upsilon}$ is analytic.  Thus 
since $\cL_t = \cL_{t,0}$ has a spectral gap, so does $\cL_{t, \upsilon}$ 
for $| \upsilon |$ sufficiently small, and the leading
eigenvalue $\lambda_{t, \upsilon}$ 
varies analytically in $\upsilon$  \cite[VII, Thm 1.8, II.1.8]{kato};
moreover, $\lambda_{t, 0} = e^{P(t)}$ and,
with $\mu_t$ from Theorem~\ref{thm:spectral}, we have \cite[II.2.1, (2.1), (2.33)]{kato}
\begin{equation}
\label{eq:deriv lambda}
\left. \frac{d}{d\upsilon} \lambda_{t, \upsilon} \right|_{\upsilon = 0} = e^{P(t)} \int \phi \, d\mu_t \, , \, \forall t \in [t_0, t_1]\, .
\end{equation} 

Recalling the definition $P(t,\upsilon \phi)$
 in \eqref{abovv}, the following result will give
Proposition~\ref{prop:unique}:

\begin{proposition}
\label{prop:spec rad}
Fix $0<t_0<t_1<t_*$.
For $\phi \in C^1(M)$, 
$t \in [t_0,t_1]$, and $\upsilon \in \mathbb{R}$, with $|\upsilon|$ sufficiently small,  
the spectral radius of $\cL_{t, \upsilon}$ on $\cB(t_0,t_1)$ is 
$\lambda_{t, \upsilon\phi} = e^{P(t, \upsilon \phi)}$.
\end{proposition}

\begin{proof}[Proof of Proposition~\ref{prop:unique}]
We use  tangent measures, inspired by the proof of 
\cite[Theorem~16]{bruin notes}:
 If $\mu$ is an equilibrium state for 
$-t \log J^uT$ then
  $\mu$ is a $C^1$-tangent measure at $t$ (see e.g.\footnote{The standard definitions use
$C^0$ rather than $C^1$ in \eqref{eq:tangent}  
For our purposes, $C^1$ will suffice.} \cite[Theorem~9.14]{walters}) in the sense that,
\begin{equation}
\label{eq:tangent}
P(t , \phi) \ge P(t,0) + \int \phi \, d\mu \qquad
\mbox{for all $\phi \in C^1(M)$.}
\end{equation}
Thus, Proposition~\ref{prop:spec rad} together with 
\eqref{eq:deriv lambda} imply that
\begin{align*}
\int \phi \, d\mu_t & = \lim_{\upsilon \downarrow 0} \frac{P(t , \upsilon \phi) - P(t, 0)}{\upsilon} \ge \int \phi \, d\mu \, \quad \mbox{and} \\
\int \phi \, d\mu_t & = \lim_{\upsilon \uparrow 0} \frac{P(t,  \upsilon \phi) - P(t,0)}{\upsilon} \le \int \phi \, d\mu \, .
\end{align*}
Thus $\int \phi \, d\mu_t = \int \phi \, d\mu$ for all $\phi \in C^1(M)$.  Since $M$ is a compact
metric space, $C^1(M)$ is
dense in $C^0(M)$ and so  $\mu = \mu_t$ showing the uniqueness claim in the  proposition. 
\end{proof}

\begin{proof}[Proof of Proposition~\ref{prop:spec rad}]
Let $|\upsilon|$  be small enough such that $g:=\upsilon\phi$ satisfies \eqref{eq:up one}, 
\eqref{eq:up two},  \eqref{eq:up three} and their analogues for the number of interpolations needed
to reach $t_1<t_*$ in Section~\ref{sec:boot}. 
 The constants $n_1$, $n_2$, $\delta_1$, $\delta_2$, $C_2$, $c_0$, $c_1$, $c_2$, and $C_\kappa$
from Section~\ref{GrL} then hold for all
$g=\upsilon' \phi$ with $|\upsilon' |< |\upsilon|$
and all $t \in [t_0, t_1]$.
In particular the constants
$c_1(\upsilon')>0$ from Proposition~\ref{prop:lower bound}  and Proposition~\ref{prop:summary}(a)  and $c_2(\upsilon')>0$  in  
Proposition~\ref{prop:exact}  and Proposition~\ref{prop:summary}(b)  are uniform in $|\upsilon' |< |\upsilon|$
and  $t \in [t_0, t_1]$.

\medskip
\noindent
{\em Step 1. The Spectral Radius $\lambda_{t,\upsilon}$ of $\cL_{t, \upsilon}$ on $\cB$ is $e^{P_*(t, \upsilon\phi)}$.}
Possibly reducing $|\upsilon|$ further, 
$\cL_{t, \upsilon}$ has a spectral gap on $\cB$, as observed above.
The upper bound on  $\lambda_{t,\upsilon}\le e^{P_*(t, \upsilon\phi)}$
can thus be proved as in Proposition~\ref{prop:radius}, once we
know that the spectral radius of $\cL_{t, \upsilon}$
on $\cB_w$ is at most  $e^{P_*(t, \upsilon\phi)}$.
For this, by the  upper bound in 
 Proposition~\ref{prop:summary}(b),  it suffices  to find
$C  < \infty$  such that 
\begin{equation}
\label{eq:weak up}
| \cL_{t, \upsilon}^n f |_w   \le   C Q_n(t, \upsilon\phi) |f|_w\, ,\,\,
\forall  f \in C^1  \, . 
\end{equation}

To prove \eqref{eq:weak up}, note that due to \eqref{eq:phi dist}, we have for $W \in \cW^s$ and $W_i \in \cG_n(W)$,
\begin{equation}
\label{eq:phi alpha}
|e^{\upsilon S_n \phi}|_{C^\alpha(W_i)} \le (1+ C_* |\nabla \phi|_{C^0}\cdot \delta_0^{1-\alpha}) |e^{\upsilon S_n\phi}|_{C^0(W_i)} \, ,
\end{equation}
then, for $W \in \cW^s$ and $\psi \in C^\alpha(W)$ with $|\psi|_{C^\alpha(W)} \le 1$, we
follow \eqref{eq:weak start} and apply \eqref{eq:test contract}, \eqref{eq:JC},
Lemma~\ref{lem:extra growth}, and \eqref{eq:phi alpha}
to write,
\begin{align*}
\int_W & \cL_{t, \upsilon}^n f \, \psi \, dm_W 
 \le \sum_{W_i \in \cG_n(W)} |f|_w |\psi \circ T^n|_{C^\alpha(W)} ||J^sT^n|^t e^{\upsilon S_n \phi}|_{C^\alpha(W_i)} \\
& \le |f|_w C_1^{-1} (1+ 2^t C_d)(1+ C_* |\nabla \phi|_{C^0}) \sum_{W_i \in \cG_n(W)} |J^sT^n|^t_{C^0(W_i)}
|e^{\upsilon S_n \phi}|_{C^0(W_i)} \le C |f|_w Q_n(t, \upsilon\phi) \, .
\end{align*}

The lower bound $\lambda_{t,\upsilon}\ge e^{P_*(t, \upsilon\phi)}$
on the spectral radius follows as in the proof of Proposition~\ref{prop:radius}.

\smallskip
\noindent
{\em Step 2.  $P_*(t, \upsilon\phi)= P(-t \log J^uT+\upsilon \phi)$.}
Denoting by $\nu_{t, \upsilon}$ the eigenmeasure associated to $e^{P_*(t, \upsilon\phi)}$
and by $\tnu_{t, \upsilon}$ the  eigenmeasure of the dual operator
$\cL_{t, \upsilon}^*$, defined as in Lemma~\ref{lem:peripheral} and \eqref{eq:dual},
we construct an invariant probability measure $\mu_{t, \upsilon}$ as in
Proposition~\ref{prop:mut def}.

We claim  the following analogue of  Proposition~\ref{prop:max}:
There exists $A<\infty$ such that for all sufficiently small $|\upsilon|$,
all $\epsilon>0$ sufficiently small, all $x \in M$ and $n \ge 1$,
\begin{equation}
\label{eq:upper ball up}
\mu_{t, \upsilon}(B_n(x , \epsilon)) \le A e^{- n P_*(t, \upsilon\phi) + t  \log J^sT^n(T^{-n} x)  + \upsilon S_n\phi( T^{-n}  x)} \, ,
\end{equation} 
where $B_n(x, \epsilon)$ is the Bowen ball defined in \eqref{Bowen}.
Using \eqref{eq:upper ball up}, the proof of Corollary~\ref{cor:max} yields that 
$P_{\mu_{t, \upsilon}}(-t \log J^uT + \upsilon \phi) \ge P_*(t, \upsilon\phi)$, and this, together
with Proposition~\ref{prop:pressure} yields $P(-t \log J^uT + \upsilon \phi) = P_*(t, \upsilon\phi)$.  
By Step 1, this ends the proof of
Proposition~\ref{prop:spec rad}.  
(In addition, we have established that $\mu_{t, \upsilon}$ is an equilibrium state for $-t \log J^uT + \upsilon \phi$.)

Finally, \eqref{eq:upper ball up}  follows easily from the proof of 
Proposition~\ref{prop:max}.  The property in \eqref{eq:pos} extends to $\nu_{t, \upsilon}$
due to its definition as a limit of $e^{- n P_*(t, \upsilon\phi)} \cL_{t, \upsilon}^n 1$.  The analogue
of \eqref{eq:change var} holds for the same reason, so that the modification of
\eqref{eq:unwrap} yields,
\[
\int_W \psi \, 1^B_{n \epsilon} \, \nu_{t, \upsilon} = e^{-n P_*(t, \upsilon\phi)}
\sum_{W_i \in \cG_n(W)} \int_{W_i} (\psi \circ T^n) (1^B_{n, \epsilon} \circ T^n) |J^sT^n|^t e^{\upsilon S_n \phi}
\, \nu_{t, \upsilon} \, , 
\] 
where $1^B_{n, \epsilon}$ denotes the indicator function of $B_n(x, \epsilon)$.  The
subsequent estimates in the proof of Proposition~\ref{prop:max} go through with the obvious
changes,
so that
\[
|1^B_{n, \epsilon} \nu_{t, \upsilon}|_w \le C' e^{-n P_*(t, \upsilon\phi) + t \log J^sT^n( T^{-n}  x) + \upsilon S_n\phi( T^{-n}  x)} \, ,
\] 
where the only additional factor needed is the distortion constant $C_* |\nabla \phi|_{C^0}$ from
\eqref{eq:phi dist}.  Applying the analogue of Proposition~\ref{prop:mut def}(a)
completes the proof of \eqref{eq:upper ball up}.
\end{proof}

\begin{proof}[Proof of Theorem~\ref{thm:variational}]
The upper bound  $P(t,g)\le P_*(t,g)$  is the content of Proposition~\ref{prop:pressure}.
Taking $\upsilon \phi=g$, the equilibrium state for $-t\log J^u+g$
is $\mu_{t,\upsilon}$ constructed in Step 2 of  the proof of Proposition~\ref{prop:spec rad}.
The proof of uniqueness can be obtained by a straightforward
adaptation of the argument proving
uniqueness of the equilibrium state for $-t\log J^u$, up to taking small enough
$|g|_{C^1}$.
\end{proof}


\section{Analyticity,  Derivatives of $P(t)$,
and Strict Convexity (Proof of Theorem~\ref{strconv})}
\label{analytic}

This section contains the proof of Theorem~\ref{strconv}
and   Corollaries~\ref{contin eq},~\ref{unifrate}, and~\ref{CLT}.

The maximal eigenvalue of $\cL_t^n$ is $\exp(nP(t))$. Showing that
$nP(t)$ is analytic for some integer $n\ge 1$ is equivalent to showing that $P(t)$ is analytic.
Recall the one-step expansion factor $\theta^{-1}>1$  from Lemma~\ref{lem:one step}. In the remainder of this section\footnote{The value $1/2$ below is for convenience, giving the number $-\log 2$ in Lemma~\ref{distlog}; what
is important is  $C_0\theta^n<1$.}:
$$
\mbox{Fix }
n_0\ge 1 \mbox{ such that } |J^sT^{n_0}|<C_0\theta^{n_0}\le \frac{1}{2} 
\, ,
\mbox{ and set } \cT:=T^{n_0}\, .$$
By standard results on analytic perturbations of simple isolated eigenvalues \cite{kato}, 
analyticity of $P(t)=P_*(t)$ will be an immediate consequence of the following
result:

\begin{proposition}[Analyticity of $t\mapsto \cL^{n_0}_t$]\label{opan}
Fix $0<t_0<t_1<t_*$.
Then the map $t\mapsto \cL^{n_0}_t$ is analytic from $(t_0,t_1)$
to the space of bounded operators from $\cB$ to $\cB$, with
\begin{equation}
\label{derivative}
\partial^j_t \cL^{n_0}_t(f) |_{t=w}=  \cL^{n_0}_w\bigl( (\log J^s \cT)^j f \bigr)\,,\,\,\forall j\ge 1\, ,\,\,\forall w\in  (t_0,t_1)\, , \,\,
\forall f \in \cB\, .
\end{equation}
\end{proposition}

\begin{proof}
We claim that it suffices to prove that, for any
 $0< t_0 <t_1<t_*$,  we have
\begin{align}
\label{analyticbound}
&\mbox{ there exists $C<\infty$ such that }\|\cL^{n_0}_w\bigl( (\log J^s \cT)^j f \bigr)\|_\cB\le { j } ( C j)^j\,  \|f\|_\cB\, , \,\,
\forall f \in \cB\, ,
\end{align}
for all  $w \in (t_0,t_1)$ and all $j\ge 0$.
(The bound \eqref{analyticbound} is the content of Proposition~\ref{prop:analyticboundprop}.)

Indeed, by the Stirling formula, \eqref{analyticbound} implies
\begin{equation}
\label{eq:power0}
\frac{\| \cL_w^{n_0}(f (\log J^s\cT)^j) \|_{\cB} }{j!} \le j C^j C_{Stirling}^j 
\|f\|_{\cB} \, .
\end{equation}
Now, for $w \in (t_0, t_1)$ and $t \in \mathbb{C}$, first write
\begin{align*}
\cL_t^{n_0} f & = \frac{f \circ \cT^{-1}}{|J^s\cT|^{1-w} \circ \cT^{-1}} e^{(w-t) \log J^s\cT \circ \cT^{-1}}
= \frac{f \circ \cT^{-1}}{|J^s\cT|^{1-w} \circ \cT^{-1}} \sum_{j=0}^\infty \frac{(w-t)^j}{j!} (\log J^s\cT)^j \circ \cT^{-1}
\end{align*} 
where  \eqref{eq:power0}, with $w=1$ and $f\equiv 1$, 
gives that the series converges in norm for $|w-t| < { (C C_{Stirling})^{-1} } $.
Then note that
\begin{equation}
\label{eq:power}
\frac{f \circ \cT^{-1}}{|J^s\cT|^{1-w} \circ \cT^{-1}} \sum_{j=0}^\infty \frac{(w-t)^j}{j!} (\log J^s\cT)^j \circ \cT^{-1}  = \sum_{j=0}^\infty \frac{(w-t)^j}{j!} \cL_w^{n_0} (f (\log J^s\cT)^j) \, ,
\end{equation}
where the sum commutes with $\cL_w^{n_0}$ due to \eqref{eq:power0}, with $w$ and $f$,
so that this series also converges in norm for $|w-t| < { ( C C_{Stirling} )^{-1} }$.  
The  radius of convergence is independent
of $w \in (t_0,t_1)$, giving the claimed analyticity there.  The power series representation 
\eqref{eq:power} immediately implies \eqref{derivative}.
\end{proof}

\medskip

The key to the analyticity result in this section is the following elementary lemma
which extends the distortion 
estimate Lemma~\ref{lem:distortion} to expressions of the type $(\log |J^s\cT|)^j|J^s \cT|^t$:

\begin{lemma}[Distortion for $\exp(\Psi) (\Psi)^j$]\label{distlog}
Fix  $I\subset \bR$ a compact interval and let
 $\Psi:I\to \bR_-$. Then, for any $\upsilon>0$, there exists $C_\upsilon<\infty$ such that
\begin{equation}\label{firststep}
|\exp(\upsilon \Psi )|\Psi|^j|_{C^0(I)} \le (C_\upsilon j)^j  \, ,\,\,\forall j\ge 1\, .
\end{equation}
In addition, if  $|\sup \Psi|=\inf|\Psi| \ge\log 2$
and there exist $\alpha\in (0,1)$ and $C_\Psi<\infty$ such that\footnote{The bound \eqref{distbd} is equivalent to
$\exp(-C_{\Psi} |x-y|^\alpha)\le \exp(\Psi(x))/\exp(\Psi(y))\le \exp(C_{\Psi} |x-y|^\alpha)$ or,
for small enough $|x-y|$, to
$|1-\exp(\Psi(x))/\exp(\Psi(y))|\le C_{d,\Psi} |x-y|^\alpha$.}
\begin{equation}
\label{distbd}
|\Psi(x)-\Psi(y)|\le C_\Psi |x-y|^\alpha\, , \, \,\forall x, y \in I\, ,
\end{equation}
then
\begin{equation}\label{distconcl0}
|\log |\Psi(x)|-\log |\Psi(y)||
\le  4 C_\Psi |x-y|^\alpha
\, ,\quad \forall x, y \in I\, ,
\end{equation}
and, for any $t>0$,
\begin{equation}\label{distconcl}
|\exp(t\Psi ) |\Psi|^j|_{C^\alpha(I)}
\le   (1 +  e C_\Psi (4j +t ) ) |\exp(t\Psi )|\Psi|^j|_{C^0(I)}
\, ,\quad \forall j\ge 0 \, .
\end{equation}
\end{lemma}

\noindent(The lemma will be applied to $\Psi=\log| J^s\cT|$, with $\alpha \le { 1/(q+1)}$,
and $I$ an interval giving an arc length parametrisation of
a weakly homogeneous stable manifold.)

\begin{proof}
The proof of \eqref{firststep} is a straightforward exercise in calculus (with
$C_\upsilon=(e\cdot \upsilon)^{-1}$): It suffices to show that 
$\sup_{X\in [0,1]}|\log X |^j X^\upsilon \le (\frac{j}{ e  \cdot \upsilon})^j$.

\medskip

Next, for any $x,y \in I$, the Mean Value Theorem applied to the logarithm yields,
for some $Z$ between $|\Psi(x)|$ and  $|\Psi(y)|$,
\begin{equation}
\label{eq:both}
|\log |\Psi(x)| - \log |\Psi(y)||
\le \frac{1}{Z} |\Psi(x) - \Psi(y)| \le \frac{C_\Psi |x-y|^\alpha}{\log 2} \le 4 C_{\Psi} |x-y|^\alpha \, . 
\end{equation}

 From \eqref{eq:both}  we get \eqref{distconcl0} and also,  for any $x,y\in I$,
\begin{align*}
\biggl |\log \frac{\exp( t\Psi(x)) |\Psi(x)|^j}
{\exp (t\Psi(y)) |\Psi(y)|^j}
\biggr |
&\le j |\log |\Psi (x)|- \log |\Psi(y)||+ t|\Psi(x) - \Psi(y)|\\
&\le ( j  4 C_\Psi  + tC_\Psi) |x-y|^\alpha \, .
\end{align*}
This implies
\begin{equation}\label{shorter}
\exp (- C_\Psi (4j+t)|x-y|^\alpha) \le \frac {\exp (  t \Psi(x)) |\Psi(x)|^j}{\exp ( t \Psi(y)) |\Psi(y)|^j}\le \exp( C_\Psi (4j + t)|x-y|^\alpha)\, .
\end{equation}
For $|x-y|^\alpha <  (4j C_\Psi+tC_\Psi)^{-1}$ (other pairs $(x,y)$ are trivial
to handle), \eqref{shorter} implies
$$
\biggl|1- \frac {\exp ( t \Psi(x)) |\Psi(x)|^j}{\exp ( t \Psi(y)) |\Psi(y)|^j}
\biggr | \le    eC_\Psi(4j+t) |x-y|^\alpha \, .
$$
Multiplying both sides  above
by $\exp (t\Psi(y)) |\Psi(y)|^j\le |\exp(t\Psi )|\Psi|^j |_{C^0(I)}$,
proves \eqref{distconcl}.
\end{proof}

\medskip

Recalling that $\cT=T^{n_0}$ for fixed ${n_0}$, we define, for all integers $j\ge 0$,
\begin{equation}\label{defMj}
\cM^{(j)}_t f:= \cL^{n_0}_t\bigl( (\log |J^s \cT|)^j f )\, ,
\end{equation}
acting on measurable functions.
We first
prove \eqref{analyticbound}:

\begin{proposition}\label{prop:analyticboundprop}
For any
 $0<t_0<t_1<t_*$, there exists $C<\infty$ such that
\begin{equation}
\label{analyticboundprop}
\|\cM^{(j)}_t f\|_\cB\le  C^j\, j^{j+1} \,  \|f\|_\cB\, , \,\,
\forall j\ge 1\, , \forall f \in \cB \, ,\, \forall t \in [t_0,t_1] \, .
\end{equation}
\end{proposition}

\begin{remark}
\label{rem:Mj}
A modification of the proof of Lemma~\ref{lem:image} shows that
for any $f \in C^1(M)$, $\cM_t^{(j)}f$ can be approximated by $C^1(M)$ functions
in the $\cB$ norm, using the fact that Lemma~\ref{lem:smooth} holds
for the function $(\log |J^s\cT|)^j |J^s\cT|^t$ by Lemma~\ref{distlog}.
By density of $C^1(M)$ in $\cB$, this, together with Proposition~\ref{prop:analyticboundprop},
implies $\cM_t^{(j)}f \in \cB$ for all $f \in \cB$ and $j \ge 0$.   
\end{remark}

\begin{proof}[Proof of Proposition~\ref{prop:analyticboundprop}]
It is enough to consider $f\in C^1(M)$.
We first bound the  stable norm.
Fix  $W \in \cW^s$, and $\psi \in C^\beta(W)$ such that  $|\psi|_{C^\beta(W)} \le |W|^{-1/p}$.  
For  $t >0$, we have
\begin{equation}
\label{eq:strTaylor start}
\begin{split}
\int_W \cM_t^{(j)} f \, \psi \, dm_W & = \sum_{W_i \in \cG_{ n_0}(W)} \int_{W_i} f (\psi \circ \cT) |J^s\cT|^t (\log |J^s \cT|)^j \, dm_{W_i} \\
& \le \sum_{W_i \in \cG_{ n_0}(W)} \|f\|_s |\psi \circ \cT|_{C^\beta(W)}
\frac{|W_i|^{1/p} }{|W|^{1/p}}||J^s\cT|^t (\log |J^s \cT|)^j|_{C^\beta(W)} \, .
\end{split}
\end{equation}
On the one hand, we have seen in \S\ref{sec:weak norm} that
$|\psi \circ \cT|_{C^\beta(W_i)} \le \tilde C |\psi|_{C^\beta(W)}$.
On the other hand, recalling that $\sup_{W, W_i}|J^s\cT |_{C^0(W_i)} <1$,
and using \eqref{firststep} from Lemma~\ref{distlog},
for any $\upsilon>0$, there exists $C_\upsilon$ such that for any $W_i\in \cG_n(W)$, all $t \in [t_0,t_1]$,
and all $j\ge 1$,
\begin{align}
\label{supinf}\sup_{W_i}(|\log |J^s\cT|^j||J^s\cT|^t)
\le (j C_\upsilon)^j \, 
\sup_{W_i}|J^s\cT|^{t-\upsilon}
 \, .
\end{align}
Therefore, since $\beta <\alpha$,  choosing\footnote{ This is always possible since 
$p > q+1$ and $t_0 q \ge 4$ from Definition~\ref{def:k0d0}.} $\upsilon < t_0/2 - 1/p$ and applying Lemma~\ref{distlog}, we deduce from \eqref{eq:strTaylor start}
and \eqref{supinf}  that  for all $j\ge 1$ and $f\in C^1$,
taking $C_d' = { 1+} e C_\Psi (4+t)$  from \eqref{distconcl},
\begin{align*}
\int_W \cM^{(j)}_t f \, \psi \, dm_W & \le C_\upsilon^j j^j \, j C_d'
\sum_{W_i \in \cG_{n_0}(W)} \|f\|_s  { \frac{|W_i|^{1/p}}{|W|^{1/p}} } |J^s\cT|^{t-\upsilon}_{C^0(W_i)} \\
& \le   C_d' C_2[0] \,  j (j C_\upsilon)^j \|f\|_s Q_{n_0}(t-\upsilon { - 1/p} ) \, ,
\end{align*}
where
$C_\Psi = C_d$ by \eqref{eq:distortion}, and we  used Lemma~\ref{lem:extra growth} with $\varsigma = 1/p$.
Taking the suprema over
$\psi \in C^\beta(W)$ with 
$|\psi|_{C^\beta(W)} \le |W|^{-1/p}$ and $W \in \cW^s$ yields $C_s<\infty$ such that
\begin{equation}
\label{stlog}
\|\cM^{(j)}_t( f )\|_s\le {j} \frac{(j C_s)^j}{2}  \,  \|f\|_s\, , \,\,
\forall j\ge 1\, , \,\, \forall f \in C^1\, ,\, \forall t \in [t_0,t_1] \, .
\end{equation}

For the unstable norm, let  $\ve < \ve_0$ and let $W^1, W^2 \in \cW^s$ with 
$d_{\cW^s}(W^1, W^2) \le \ve$.  For  $\ell=1,2$, we partition
$\cT^{-1}W^\ell$ into matched pieces $U^\ell_k$ and unmatched pieces $V^\ell_i$ as
in \S\ref{sec:unstable}, and we find, for any $\psi_\ell \in C^\alpha(W^\ell)$ with $|\psi_\ell|_{C^\alpha(W^\ell)} \le 1$ and $d(\psi_1, \psi_2)=0$,
\begin{align}
\label{eq:unstable decompl}
&\left| \int_{W^1} \cM_t^{(j)} (f )\, \psi_1 - \int_{W^2} \cM_t^{(j)} (f )\, \psi_2 \right|\\
\nonumber &\qquad \le \sum_k \left| \int_{U^1_k} f \,( \psi_1 \circ \cT) \, |J^s\cT|^t 
 (\log |J^s \cT|)^j- \int_{U^2_k} f \, (\psi_2 \circ \cT) \, |J^s\cT|^t (\log |J^s \cT|)^j\right| \\
\nonumber& \qquad\qquad\qquad\qquad\qquad + \sum_{\ell, i} \left| \int_{V^\ell_i} f \, (\psi_\ell \circ \cT )\, |J^s\cT|^t (\log |J^s \cT|)^j \right|\, .
\end{align}
For the unmatched pieces, adapting \eqref{eq:unmatched}, 
by using Lemma~\ref{distlog} combined with
Lemma~\ref{lem:distortion}
and \eqref{supinf},  we find, for $\ell=1,2$ { (choosing again $\upsilon < t_0/2-1/p$ so that
$t - \upsilon - 1/p > t_0/2$),}
\begin{align}
\label{eq:unmatchedl}
&\sum_{\ell, i} \left| \int_{V^\ell_i} f \, (\psi_\ell \circ \cT) \, |J^s\cT|^t (\log |J^s \cT|)^j\right|\\
\nonumber &\qquad \le \| f \|_s C_1^{-1} j C'_d(j C_\upsilon)^j\sum_{\ell,i} |\cT V^\ell_i|^{1/p}
|  |J^s\cT|^{t-\upsilon-1/p}|_{C^0(V^\ell_i)} \\
\nonumber &\qquad \le  \| f \|_s 4  C_2[0] C_1^{-1} j C'_d ( j C_\upsilon)^j \ve^{1/p} Q_{n_0}(t-\upsilon-1/p)\,  ,\, \forall t \in [t_0,t_1]\, ,\,
\forall j\ge 1\, ,
\end{align}
using Lemma~\ref{lem:extra growth} with $\varsigma = 0$.
Next,  we consider matched pieces. 
Recalling \eqref{eq:match},
we   define 
\[
(\widetilde \log J^s\cT)^j(x)
 := (\log J^s\cT)^j \circ G_{U^2_k} \circ G_{U^1_k}^{-1}(x) \, ,\,\,
\forall x \in U^1_k\,,\,_,
\forall j\ge 1\, .
\]
Now,  using $\tpsi_2$ and $\tJ^s\cT=\tJ^sT^{n_0}$ as defined above  \eqref{eq:match split}, and injecting
Lemma~\ref{distlog} in the proof of Sublemma~\ref{lem:matching}(b) gives, for $\upsilon$ as above,
\begin{align}
 \nonumber|(\psi_1 \circ \cT ) (\log J^s\cT)^j|J^s\cT|^t -& \tpsi_2 (\widetilde \log J^s\cT)^j\tJ^s\cT|^t |_{C^\beta(U^1_k)}\\
\label{logmatchedb}&\qquad 
\le C ( j C_\upsilon)^j  j C'_d 2^t |J^sT^{n_0}|^{t-\upsilon}_{C^0(U^1_k)} \ve^{\alpha-\beta}\, ,
\,\,\forall k\, ,\,\,\forall j\,,\,\forall
t\in[t_0,t_1]\, .
\end{align}
Then  we split 
\begin{align}
\nonumber
&\left| \int_{U^1_k} f \, (\psi_1 \circ \cT) \,( \log J^s\cT)^j |J^s\cT|^t \right.  - \left. \int_{U^2_k} f \, (\psi_2 \circ \cT )\,( \log J^s\cT)^j |J^s\cT|^t \right|\\
\label{eq:match splitl} &\qquad\qquad\le \left| \int_{U^1_k} f \, \left( (\psi_1 \circ \cT )\,( \log J^s\cT)^j |J^s\cT|^t - \tpsi_2  (\widetilde \log J^s\cT)^j|\tJ^s\cT|^t \right) \right| \\
\label{logmatch split}& \qquad\qquad\qquad + \left|   \int_{U^1_k} f \, \tpsi_2 (\widetilde \log J^s\cT)^j|\tJ^s\cT|^t
-   \int_{U^2_k} f \, (\psi_2 \circ \cT) \,( \log J^s\cT)^j |J^s\cT|^t  \right| \, .
\end{align}
We estimate \eqref{eq:match splitl} for all 
$t\in [t_0,t_1]$ and $j\ge 1$ using \eqref{logmatchedb},
\[
\left| \int_{U^1_k} f \, \left( (\psi_1 \circ \cT)  (\log J^s\cT)^j\, |J^s\cT|^t - \tpsi_2 (\widetilde \log J^s\cT)^j|\tJ^s \cT|^t \right ) \right|
 \le \| f \|_s \delta_0^{1/p}  j C_d' (C_\upsilon j )^j 2^t |J^s\cT|^{t-\upsilon}_{C^0(I_k)}  \ve^{\alpha - \beta} \, .
\]
Then, noting that $d(\psi_1 \circ \cT \,  (\log J^s\cT)^j |J^s\cT|^t, \tpsi_2  (\widetilde \log J^s\cT)^j |\tJ^s\cT|^t) = 0$ by definition, and 
that the $C^\alpha$ norms of both 
test functions are bounded by $C (j C_\upsilon )^jj C'_d |J^s\cT|^{t-\upsilon}_{C^0(I_k)}$, using 
Lemma~\ref{distlog}, we estimate   \eqref{logmatch split}
for all $f\in C^1$ and
$t\in [t_0,t_1]$ as follows
\begin{align*}
\left| \int_{U^1_k} f \, ((\psi_2 \circ \cT) (\log J^s\cT)^j\, |J^s\cT|^t - \tpsi_2 (\widetilde \log J^s\cT)^j|\tJ^s\cT|^t  )\right|
& \le \| f \|_u d_{\cW^s}(U^1_k, U^2_k)^\gamma  j C'_d (j C_\upsilon)^j C |J^s\cT|^{t-\upsilon}_{C^0(I_k)} \\
& \le C'( j C_\upsilon)^j j C'_d\| f \|_u {n_0}^\gamma \Lambda^{-{n_0} \gamma} \ve^\gamma |J^s\cT|^{t-\upsilon}_{C^0(I_k)} \, ,
\end{align*}
where we used Lemma~\ref{lem:matching}(a) in the second inequality.  Putting these estimates into \eqref{eq:match splitl},
combining with \eqref{eq:unmatchedl} in \eqref{eq:unstable decompl}, and summing over $k$ 
gives, for all
$t\in [t_0,t_1]$ and $j \ge 1$,
\begin{align*}
&\left| \int_{W_1} \cM_t^{(j)} f \, \psi_1 - \int_{W_2} \cM_t^{(j)} f \, \psi_2 \right|\\
&\quad\le   j (j\bar C)^j \, \bigl( \| f \|_u {n_0}^\gamma \Lambda^{-{n_0} \gamma} \ve^\gamma Q_{n_0}(t-\upsilon) + 
\| f \|_s (\ve^{1/p} Q_{n_0}(t-1/p-\upsilon) + \ve^{\alpha- \beta} Q_{n_0}(t-\upsilon)) \bigr)  .
\end{align*}
Finally, since $\alpha-\beta \le \gamma$ and $1/p\le \gamma$, while ${n_0}$ is fixed,
we have found $C_u<\infty$ such that
$$ \|\cM_t^{(j)} f \|_u
\le {j } \frac{(jC_u)^j}{2} (\|f\|_s + \|f\|_u)\,,\,\,\forall f\in C^1\,,\,\,\forall
t\in [t_0,t_1] \, ,\, \forall j\ge 1\, .
$$
With \eqref{stlog}, taking $C=\max\{C_s,C_u\}$, this concludes the proof of Proposition \ref{prop:analyticboundprop}.
\end{proof}

\begin{proof}[Proof of Theorem~\ref{strconv}] 
Since $\exp(n_0P(t))>0$ is a simple isolated eigenvalue of $\cL^{n_0}_t$,
analyticity
of $\exp (n_0P(t))$ is an immediate consequence of Proposition~\ref{opan}
and \cite[VII, Theorem 1.8, II.1.8]{kato}.
Since 
$\inf_{[t_0,t_1]}\exp(n_0P(t))>0$, the function $P(t)$ is also analytic.
The formulas 
$$n_0P'(t)\exp (n_0P(t))\, ,\,\, \quad n_0P''(t)\exp (n_0P(t))+{n_0}^2P'(t)^2 \exp (n_0P(t))
$$ 
 can be read off
\cite[II.2.2, (2.1), (2.33) p.79]{kato} (taking $m=1$ there). It is then easy to 
extract the claimed formula \eqref{formula0}
 for $P'(t)$. In order to\footnote{Formulas \eqref{formula0}--\eqref{formula} are  classical in 
smooth hyperbolic dynamics, see \cite[Chap. 5, ex. 5b]{Ru} for $P''(t)$.} establish  \eqref{formula} for $P''(t)$, use \eqref{formula0},
and note that, recalling 
$\chi_t=P'(t)=\int \log J^s T \, d \mu_t$,
\begin{align*}
\sum_{k\ge 0}
\biggl [ \int  (\log |J^s T| \circ T^k) & \log |J^s T|\, d \mu_t -\chi_t^2\biggr ]\\
&=\langle  \log |J^s T |\bigl (1- e^{-P_*(t)} \cL_t\bigr )^{-1}\biggl ( (\log |J^s T|-\chi_t) \nu_t  \biggr ),
\tnu_t
\rangle  \, .
\end{align*}

\smallskip

If there exists $f\in L^2(\mu_t)$ such that
$\log |J^sT|-\chi_t=f-f\circ T$ then it is easy to see that $P''(t)=0$.
For the converse statement, we will use a martingale CLT  result \`a la Gordin (see e.g. Viana \cite[Theorem E.11]{BDV})
as in \cite{DRZ}:
Let $\cA_0$  be the sigma-algebra generated by the ($\mu_t$-mod $0$)
partition of $M$ into maximal connected, strongly homogeneous local stable manifolds for
$T$ (this partition is measurable since it has a countable generator, see e.g. \cite[\S5.1]{chernov book}). 
Then $\cA_n=T^{-n}\cA_0$, for $n\in \bZ$, is a decreasing sequence of sigma algebras.  
Therefore, if $P''(t)=0$, to obtain $f\in L^2(\mu_t)$
such that $\log| J^sT|=\chi_t+f -f\circ T$
from Gordin's Theorem (\cite[Theorem E.11]{BDV}
or \cite[Theorem 5.1]{DRZ}), we only need to check the following two conditions: 
\begin{align}\label{CLT1}
&\sum_{n=0}^\infty \|\log |J^s T|- E((\log |J^s T|-\chi_t)|\cA_{-n})\|_{L^2(\mu_t)}<\infty\, ,
\\
\label{CLT2}
&\sum_{n=0}^\infty
 \| E((\log |J^s T|-\chi_t)|\cA_{n})\|_{L^2(\mu_t)}<\infty
\, .
\end{align}
We first discuss \eqref{CLT1}.
If $n\ge 0$, then  the elements of $\cA_{-n}$ 
are of the form $T^n(V^s(x))$ where $V^s(x)$ is the maximal  connected, strongly homogeneous
stable manifold  of (almost
every) 
$x$. From  Lemma~\ref{distlog},
the function $\log |J^s T|$
is (H\"older)  continuous on $T^n(V^s(x))$ 
for any $n\ge 1$, so,
letting $\cA_{-n}(x)$ be 
the element of $\cA_{-n}$ containing $x$,  we have
$$
E(\log |J^sT| |\cA_{-n})(x)=\log |J^sT|(y)
\, ,
$$  
for some $y\in  \cA_{-n}(x)$. Thus  (see the proof of \cite[(5.3)]{DRZ}),
\eqref{distconcl0} with $\alpha = 1/(q+1)$ gives
\begin{align*}
&\|\log |J^s T|- E((\log |J^s T|-\chi_t)|\cA_{-n})\|_{L^2(\mu_t)}\\
&\qquad \le
\|\log |J^s T|- E((\log |J^s T|-\chi_t)|\cA_{-n})\|_{L^\infty(\mu_t)}\le C \Lambda^{-n  /(q+1)} \, , \,\forall n\ge 1\, ,\,\,\forall
t \in [t_0,t_1]\, .
\end{align*}
(The length of any element $T^n(V^s(x))$ in $\cA_{-n}$ 
 is bounded by $C_0\Lambda^{-n}$.)
This proves \eqref{CLT1}.

To establish \eqref{CLT2}, we also adapt the argument in \cite{DRZ}, starting from
\begin{align*}
\sum_{n=0}^\infty
& \| E((\log |J^s T|-\chi_t)|\cA_{n})\|_{L^2(\mu_t)}\\
&=\sum_{n=0}^\infty
\sup \Big\{
\int (\log |J^s T|-\chi_t) \cdot(\psi \circ T^n) \, d\mu_t \mid \psi \in L^2(\cA_0,\mu_t) \mbox{ with }
\|\psi\|_{L^2(\mu_t)=1} \Big\} \, .
\end{align*}
The key new ingredient is the fact that,\footnote{Since  $f\equiv 1\in \cB$,
  Remark~\ref{rem:Mj}  also implies
$\log |J^s T^{n_0}| \circ T^{-n_0}=\log |J^s \cT| \circ \cT^{-1}=\cM^{(1)}_1 (1)\in \cB$.}
since { $\nu_t = e^{-n_0 P(t)} \cL_t^{n_0}(\nu_t)$,} 
with $\nu_t\in \cB$, we get from
Proposition~\ref{prop:analyticboundprop}  and Remark~\ref{rem:Mj}
that
\begin{equation}
\label{forCLT2}
(\log |J^s T^{n_0}| \circ T^{-n_0})\nu_t =  e^{-n_0P(t)} \cM^{(1)}_t(\nu_t)
\in \cB \subset \cB_w\, .
\end{equation}
By definition, any $\cA_0$-measurable function $\psi$ is constant on each curve
in $\cA_0$.
 If in addition $\psi$ is bounded, then for any $k \ge 0$, 
$\psi \circ T^k \in C^\alpha(\cW^s_{\bH})$ and 
$|\psi \circ T^k |_{C^\alpha(\cW^s_{\bH})} = |\psi|_{C^0(\cW^s_{\bH})} =: |\psi|_\infty$.
Thus by Lemma~\ref{lem:distrib} and \eqref{eq:weak dual}, 
\begin{equation}
\label{forCLT2b}
\psi  \circ T^k \tnu_t \in \cB^*_w \, \mbox{ and } \,
|\langle f , \psi \circ T^k \tnu_t \rangle |\le 
C'   |f|_w |\psi|_\infty\, ,
\forall f\in \cB_w \, .
\end{equation}
Then, recalling that $\mu_t(f)=\langle f \nu_t , \tilde \nu_t\rangle / \langle \nu_t, \tnu_t \rangle$ 
for suitable $f$, following \cite{DRZ}, we write
 for $n\ge n_0$,
and any bounded $\cA_0$-measurable function $\psi$,
\begin{align}
\nonumber \int (\log |J^s T|-\chi_t) & \cdot(\psi \circ T^n) \, d\mu_t
=\frac 1 {n_0} \int (\log |J^s T^{n_0}|\circ T^{-{n_0}}-n_0\chi_t) \cdot(\psi \circ T^{n-n_0}) \, d\mu_t\\
\nonumber &=\frac 1 {n_0} \langle 
(\log |J^s T^{n_0}|\circ T^{-{n_0}}-n_0\chi_t)\nu_t \,  , \,  (\psi \circ T^{n-n_0}) \tilde \nu_t\rangle 
 / \langle \nu_t, \tnu_t \rangle   \\
\label{forCLT4}&=\frac 1 {n_0} \langle 
e^{(n-n_0)P(t)}\LL_t^{n-n_0}\bigl ((\log |J^s T^{n_0}|\circ T^{-{n_0}}-n_0\chi_t) \nu_t \bigr )\,  , \,    \psi \tilde \nu_t\rangle   / \langle \nu_t, \tnu_t \rangle  \, .
\end{align}
(The expressions in the first line are well defined and coincide
because $(\log |J^s T|-\chi_t)\in L^1(d\mu_t)$ and
$\psi$ is bounded. 
The expression in the second line is
well defined by \eqref{forCLT2} and \eqref{forCLT2b}. 
 Therefore, the second equality  holds due to the definition of $\mu_t$ in 
Proposition~\ref{prop:mut def}(b). 
The last equality is clear.) 
Clearly $\tilde \nu_t\bigl (\nu_t(\log |J^s T^{n_0}|\circ T^{-{n_0}})\bigr ) =n_0\chi_t$, so that Corollary~\ref{unifrate}
and \eqref{forCLT2}
give constants $\rho<1$ and $C'_1, C'_2<\infty$ such that
for all $n\ge n_0$ and all $t\in [t_0,t_1]$
\begin{align}
\nonumber
|e^{(n-n_0)P(t)}\LL_t^{n-n_0}&\bigl (\nu_t(\log |J^s T^{n_0}|\circ T^{-{n_0}}-n_0\chi_t)\bigr )|_w\\
\nonumber &\le \|e^{(n-n_0)P(t)}\LL_t^{n-n_0}\bigl (\nu_t(\log |J^s T^{n_0}|\circ T^{-{n_0}})-n_0\chi_t)\bigr )\|_\cB\\
\label{CLTlast}&\le C_1'\rho^{n-n_0} \|\nu_t(\log |J^s T^{n_0}|\circ T^{-{n_0}})\|_\cB\le  C'_2 \rho^n 
\, .
\end{align}
Next,  \eqref{forCLT4} together with   the bounds \eqref{forCLT2b} and \eqref{CLTlast}, gives  $C<\infty$ such that,
for any bounded function $\psi$ which is $\cA_0$-measurable,
\begin{equation}
\label{forCLT3}
|\int (\log |J^s T|-\chi_t) \cdot(\psi \circ T^n) \, d\mu_t|\le 
C\rho^{n} |\psi|_{L^\infty(M)}\,,\,\, \forall n\ge 1\, ,\,\, \forall
t \in [t_0,t_1]\, .
\end{equation}
Then  the proof of Lemma~\ref{nbhdprop}(c) (the $T$-adapted property of $\mu_t$) not only implies that $\log |J^sT|(x)\le C \log(d(x,\cS_1))$ is in $L^1(d\mu_t)$  but  also in
$L^\ell(d\mu_t)$ for all $\ell\ge 1$ (use that 
$\sum_{j\ge 1}(\log (j+1))^\ell j^{-p'/(2p)}<\infty$ for all $\ell$
if $p'>2p$).
It follows that \cite[Lemma~5.2]{DRZ} holds for $\bar s = (\log |J^sT|) - \chi_t$,
bootstrapping the $L^\infty$ bound \eqref{forCLT3} to the required $L^2$ control \eqref{CLT2}.  
The only
change required in the proof (since the observable $\bar s$ in \cite[Lemma~5.2]{DRZ} is 
bounded while ours is not), is to replace the second term on the right-hand side of 
\cite[eq.~(5.7)]{DRZ} by the H\"older bound
$(\int |\bar s|^3 d\mu_t)^{1/3}( \int |\psi - \psi_L|^{3/2} d\mu_t)^{2/3}$
(where $\psi_L(x)=\psi(x)$ if $|\psi(x)|\le L$ and $\psi_L(x)=0$ otherwise).
Then using the fact that $\psi \in L^2(\mu_t)$,
\[
\int |\psi - \psi_L|^{3/2} d\mu_t = \int 1_{|\psi| > L} \cdot |\psi|^{3/2} \, d\mu_t \le 
L^{-1/2} |\psi|^2_{L^2} \, ,
\]
using the Markov bound $\mu(|\psi| > L) \le L^{-2} |\psi|^2_{L^2}$.
Setting $L = \rho^{-3n/4}$ instead of $L=\rho^{-n/2}$ in \cite[eq.~(5.8)]{DRZ} completes the
proof of \cite[Lemma~5.2]{DRZ} with modified rate $\rho^{n/4}$ for
our observable $\bar s$.  This verifies \eqref{CLT2} and concludes
the proof
that $\log|J^sT|$ is cohomologous in $L^2(\mu_t)$  to the constant $\chi_t<0$
if $P''(t)=0$.

\smallskip
Finally, 
$P''(t) \ge 0$ implies that $P'(t)$ is increasing so that $\int \log J^uT \, d\mu_t = - P'(t)$ is decreasing, while
 $h_{\mu_t} = P(t) - tP'(t)$ is decreasing since  $P(t)$ and $-tP'(t)$ are decreasing.
\end{proof}

\begin{proof}[Proof of Corollary~\ref{contin eq}]
The bounds on $\cM_t^{(j)}$ from Proposition~\ref{prop:analyticboundprop} apply also to the dual operator
$(\cM_t^{(j)})^*$ acting on $\cB^*$.  Thus, by Proposition~\ref{opan}, $(\cL_t^{n_0})^*$ has the analogous 
representation as a power series and is analytic.
It follows that both $\nu_t$ and $\tnu_t$ are analytic for $t \in (0, t_*)$.

By \cite[Lemma~5.3]{dz2}, if $\psi \in C^\alpha(M)$, then $\psi \nu_t \in \cB$ and 
$\| \psi \nu_t \|_{\cB} \le \bar C \| \nu_t \|_{\cB} |\psi|_{C^\alpha(M)}$, for some $\bar C >0$ 
independent of $\psi$ and $\nu_t$.  So for each $\psi \in C^\alpha(M)$, $\mu_t(\psi)$ is analytic on
$(0, t_*)$.

To prove continuity of $\mu_t$ on $C^0$ functions, we shall use the following bound.
Fix $[t_0, t_1] \subset (0, t_*)$, and define
\[
C_\nu = \sup_{t \in [t_0, t_1]} \max \{ \| \nu_t \|_{\cB}, \| \tnu_t \|_{\cB^*} \} , 
\quad C'_{\nu} = \sup_{t \in [t_0, t_1]} \max \{ \| \nu'_t \|_{\cB}, \| \tnu'_t \|_{\cB^*} \} \, , 
\]
where $\nu'_t$ and $\tnu'_t$ denote the derivatives of $\nu_t$ and $\tnu_t$ as operators on $\cB$ and $\cB^*$, respectively.
Then for $\psi \in C^\alpha(M)$,
\begin{equation}
\label{eq:uni close}
\begin{split}
\mu_t(\psi) - \mu_s(\psi) & = \langle \psi \nu_t, \tnu_t - \tnu_s \rangle + \langle \psi (\nu_t - \nu_s), \tnu_s \rangle \\
& \le \bar C |\psi|_{C^\alpha(M)} \| \nu_t \|_{\cB} \| \tnu_t - \tnu_s \|_{\cB^*} + \bar C | \psi |_{C^\alpha(M)} \| \nu_t - \nu_s \|_{\cB} \| \tnu_s \|_{\cB^*} \\
& \le C_\star |t-s| |\psi|_{C^\alpha(M)} \, ,
\end{split}
\end{equation}
where $C_\star = 2 \bar C C_\nu C'_\nu$. 

Next, let $\rho: \mathbb{R}^2 \to [0, \infty)$ denote a $C^\infty$ bump function, supported on 
the unit disk and with $\int \rho \, dm = 1$.  For $\ve >0$, let $\rho_\ve(z) = \ve^{-2} \rho(z/\ve)$,
and define $M^\ve$ to be the extension of the phase space $M$ by a strip of width $\ve$ along each component
of $\cS_0$. 

For $\psi \in C^0(M)$, extend $\psi$ to $M^\ve$ by making $\psi$ constant on vertical lines outside $M$.  Then
$\psi \in C^0(M^\ve)$.  For $x \in M$, define the convolution
\[
\psi_\ve(x) = \int \rho_\ve(x-y) \psi(y) dm(y) \, .
\]
Note that $|\psi_\ve|_{C^\alpha} \le C \ve^{-\alpha}$.  Since $M^\ve$ is compact, $\psi$ is uniformly continuous,
so there exists a decreasing function $\eta: \mathbb{R}^+ \to \mathbb{R}^+$,
$\lim_{\ve \to 0} \eta(\ve) = 0$, such that for all $\ve>0$,
$| \psi_\ve - \psi |_{C^0(M)} \le \eta(\ve)$.

Now fix $t \in (t_0, t_1)$.  For $\delta>0$, choose $\ve>0$ such that $\eta(\ve) < \delta/3$ and
choose $s_0 \in (t_0, t_1)$ such that $C_\star |s_0 - t| C \ve^{-\alpha} \le \delta/3$.  Then using \eqref{eq:uni close},
if $|s-t| < |s_0 - t|$,
\[
| \mu_t(\psi) - \mu_s(\psi)| \le | \mu_t(\psi - \psi_\ve) | + |\mu_t(\psi_\ve) - \mu_s(\psi_\ve)| + |\mu_s(\psi_\ve - \psi)| 
\le 2 \eta(\ve) + C_\star |s-t| |\psi_{\ve}|_{C^\alpha} < \delta \, .
\]
Thus $\mu_s(\psi) \to \mu_t(\psi)$ as $s \to t$.
\end{proof}

\begin{proof}[Proof of Corollary~\ref{unifrate}]
For  a compact subinterval $I$ of $(0,t_*)$ 
the bound
$\sigma(t)$ in Proposition~\ref{prop:radius} 
satisfies $\sigma_I:=\sup_{t \in I} \sigma(t)<1$.
If each $\cL_t$, for $t \in I$,  has its spectrum on $\cB$ contained in 
$e^{P_*(t)}\cup \{|z|< \sigma_I \cdot e^{P_*(t)}\}$, the corollary follows.
Otherwise,  use  Proposition~\ref{opan} and 
continuity \cite[\S IV.3.5]{kato} of any (finite) set of eigenvalues of 
finite multiplicities of bounded operators. 
\end{proof}

\begin{proof}[Proof of Corollary~\ref{CLT}]
The
corollary follows from \eqref{CLT1} and \eqref{CLT2}, using Gordin's Theorem (\cite[Theorem E.11]{BDV}
or \cite[Theorem 5.1]{DRZ}).
\end{proof}


\end{document}